\documentclass[11pt, a4paper, reqno]{amsart}
\usepackage[foot]{amsaddr}
\makeatletter
\renewcommand{\email}[2][]{%
	\ifx\emails\@empty\relax\else{\g@addto@macro\emails{,\space}}\fi%
	\@ifnotempty{#1}{\g@addto@macro\emails{\textrm{(#1)}\space}}%
	\g@addto@macro\emails{#2}%
}
\makeatother

\usepackage{amsthm,amsmath, amsfonts, amssymb, amscd,enumerate,enumitem, amscd, xcolor, url, mathrsfs, esint, cancel}
\usepackage[colorlinks]{hyperref}
\usepackage[utf8]{inputenc}
\usepackage{mathtools}

\usepackage[top=3cm, bottom=3cm, left=2cm, right=2cm, 
]{geometry}

\newtheorem{thm}{Theorem}[section]
\newtheorem{cor}[thm]{Corollary}
\newtheorem{lem}[thm]{Lemma}
\newtheorem{prop}[thm]{Proposition}
\theoremstyle{definition}
\newtheorem{defn}[thm]{Definition}
\theoremstyle{remark}
\newtheorem{rem}[thm]{Remark}

\newcommand{\RR}{\mathbb R}
\newcommand{\NN}{\mathbb N}

\newcommand{\BMO}{\textup{BMO}}
\newcommand{\loc}{\textup{loc}}

\newcommand{\aloc}{A^{\rho,\loc}_{p}}

\DeclareMathOperator{\supp}{supp}

\newcounter{pagina}

\usepackage{marginnote}

\numberwithin{equation}{section}
\usepackage[defaultlines=4, all]{nowidow}
\setnoclub[2]
\allowdisplaybreaks[2]
\setlist[enumerate,1]{label=(\roman*)}

\begin{document}

\title[A $T1$ criterion for Schr\"odinger--Calder\'on--Zygmund operators with exponential...]{A $T1$ criterion for Schr\"odinger--Calder\'on--Zygmund operators with exponential decay}

\author[E. Dalmasso]{Estefan\'ia Dalmasso$^{1}$}
\address{$^1$Instituto de Matem\'atica Aplicada del Litoral, UNL, CONICET, FIQ.\newline Colectora Ruta Nac. Nº 168, Paraje El Pozo,\newline S3007ABA, Santa Fe, Argentina}
\email{edalmasso@santafe-conicet.gov.ar}

\author[G. R. Lezama]{Gabriela R. Lezama$^{2}$}
\address{$^2$Instituto de investigaciones en Energía no convencional, UNSa, CONICET.\newline Av. Bolivia 5150 - UNSa - Dpto de Física,\newline A4408FVY, Salta, Argentina}
\email{rlezama@santafe-conicet.gov.ar}

\author[M. Toschi]{Marisa Toschi$^{3}$}
\address{$^3$Instituto de Matem\'atica Aplicada del Litoral, UNL, CONICET, FHUC.\newline Colectora Ruta Nac. Nº 168, Paraje El Pozo,\newline S3007ABA, Santa Fe, Argentina}
\email{mtoschi@santafe-conicet.gov.ar}

	\date{\today}
	

	%%%----------------------------------------------------------------------
	\begin{abstract}
		We establish the boundedness of exponential Schr\"odinger–Calderón–Zygmund operators on weighted $\BMO_\rho^\alpha(w)$ spaces via a $T1$ criterion, where the weights belong to classes that capture the exponential decay of the operators, and $\rho$ is a critical radius function. Specifically, we prove that the boundedness of such an operator $T$ on $\BMO_\rho^\alpha(w)$ is equivalent to a natural oscillation condition on $T1$
		over sub-critical balls. The weight classes considered, introduced in connection with $\rho$, include and extend the classical $A_p^\rho$ weights, and are well-adapted to the exponential decay of the kernels. As applications, we derive weighted endpoint estimates for several operators associated to the generalized Schr\"odinger operator $\mathcal L_\mu=-\Delta+\mu$, including Riesz transforms, Laplace transform-type multipliers, maximal operators for the heat and Poisson semigroups, Littlewood–Paley functions and fractional integral operators. When $d\mu(x)=V(x)dx$, the results above extend the known 
		endpoint estimates to larger classes of weights.
	\end{abstract}
	%%%----------------------------------------------------------------------
	\maketitle

\section{Introduction and main result}

A classical result in harmonic analysis, due to David and Journé (\cite{DJ}), establishes that a
Calderón--Zygmund operator $T$ is bounded on $L^2(\mathbb{R}^d)$ if and only if $T1$
and $T^*1$ belong to $\BMO(\mathbb{R}^d)$. This $T1$ theorem has since been
extended in various directions, including endpoint estimates on $\BMO$ spaces
and their weighted counterparts. In the context of Schr\"odinger operators, where the
geometry is governed by a critical radius function $\rho$, analogues of these spaces
and operators have been developed, and $T1$-type criteria have been obtained (see, for instance, \cite{BHQ19, BLL, MSTZ14}).

The operators arising in the Schr\"odinger setting present an additional feature compared with the classical case, as their kernels may exhibit not only the standard
Calderón--Zygmund size and smoothness conditions, but also an extra decay depending on the ratio $|x-y|/\rho(x)$. This decay reflects the influence of the potential on the geometry and requires weight classes adapted to it.
In~\cite{BHQ19}, a $T1$ criterion was established for operators with polynomial decay, with
weights in the class $A^\rho_p$ introduced in~\cite{BHS11}, on weighted regularity spaces $\BMO_\rho^\alpha(w)$ (see definition below). 

The aim of this article is to give a $T1$ criterion for the boundedness on $\BMO_\rho^\alpha(w)$ for exponential Schr\"odinger--Calder\'on--Zygmund
operators, defined in~\cite{DLT25}, as well as for a fractional version of these operators
introduced here, where the weights are associated with a critical radius function $\rho$ belonging to the class introduced in~\cite{Bailey21}, and are larger than $A^\rho_p$.

Before stating our main theorem, we begin by introducing some notation and definitions; further details will be provided in the following sections.

By a \textbf{critical radius function} we will mean a function $\rho: \mathbb R^d\to [0, \infty)$ for which there exist constants $k_0, C_0 \geq 1$ such that 
\begin{equation}\label{eq: critical radius function}
	C_0^{-1}\rho(x)\left(1 + \frac{|x-y|}{\rho(x)}\right)^{-k_0}\leq \rho(y)\leq C_0\rho(x)\left(1+\frac{|x-y|}{\rho(x)}\right)^{\frac{k_0}{k_0+1}},
\end{equation}
for $x,y \in \RR^d$. This function behaves like a constant at a local scale, since it is easy to check that
\begin{equation}\label{eq: equiv critical radius}
	C_0^{-1}2^{-k_0}\rho(x)\leq \rho(y)\leq C_02^{\frac{k_0}{k_0+1}}\rho(x),
\end{equation}
when $|x-y|\leq \rho(x)$. On the other hand, at a global scale it can neither increase nor decrease faster than a polynomial depending on the normalized distance $|x-y|/\rho(x)$. 

In the context of Schr\"odinger operators with potential $V$, a specific function $\rho$
can be defined that satisfies \eqref{eq: critical radius function} and shows how $V$ influences the geometry around each point. Nevertheless, the abstract formulation above allows one to develop the theory independently of the particular potential.

A ball of the form $B(x, \rho(x))$, with $x \in \RR^d$, will be a \textbf{critical ball}. When the radius  $r\leq  \rho(x)$ we will say the ball $B(x,r)$ is \textbf{sub-critical}, and when $r>\rho(x)$ we will call it \textbf{super-critical}. We denote by $\mathcal{B}_\rho$ the family of all sub-critical balls, that is,
\begin{equation}\label{eq: B_rho}
	\mathcal{B}_\rho = \left\{B(x,r): x \in \RR^d, r\leq \rho(x)\right\}.
\end{equation}   

From \eqref{eq: equiv critical radius}, note that even though the radii can be comparable when $|x-y|\leq \rho(x)$, it does not necessarily follow that $|x-y|\leq \rho(y)$. In fact, if $y\in B(x,\rho(x))$ (a critical ball), then \linebreak $x\in B(y, C_0 2^{k_0}\rho(y))$, which is super-critical since $C_0,k_0\geq 1$.

Given a weight $w$, i.e., a nonnegative locally integrable function, and a number $0\leq \alpha<1$, we say that a measurable function $f$ belongs to $\BMO_\rho^\alpha(w)$ if $f\in L^1_\loc(\RR^d)$ and satisfies
\begin{equation}\label{eq: average osc BMO}
	\frac{1}{w(B)}\int_B| f-f_B| \leq C| B |^{\alpha/d},  \quad \text{for every ball } B=B(x,r)
\end{equation}  
and 
\begin{equation}\label{eq: average BMO}
	\frac{1}{w(B)}\int_B | f | \leq C | B |^{\alpha/d}, \quad \qquad \text{for every ball }B=B(x,r) \text{ with }r\geq \rho(x).
\end{equation}
Here, $f_B$ stands for the average of $f$ over the ball $B$. As usual, the norm $\|f\|_{\BMO_\rho^\alpha(w)}$ is defined as the smallest constants in \eqref{eq: average osc BMO} and \eqref{eq: average BMO}. When $w\equiv 1$, we will denote this space by $\BMO_\rho^\alpha(\mathbb R^d)$.

Indeed, since the validity of \eqref{eq: average BMO} for some ball $B$ implies \eqref{eq: average osc BMO} for the same ball, it is sufficient to ask for the first condition to hold only on balls $B=B(x,r)$ with $r< \rho(x)$. Moreover, by virtue of \cite[Proposition~3]{BHS08}, if $w$ is a doubling weight, then we might ask condition \eqref{eq: average BMO} to be true only on critical balls.

We now define the classes of singular and fractional operators we will consider, which are included in the classes studied in \cite{BCH13Extrapol, BHQ19, BHQ20}. The singular case ($\nu=0$) was already presented in our previous article \cite{DLT25}.

\begin{defn}\label{def: s-delta} 
	For $0\leq \nu<d$, $\frac{d}{d-\nu}<s< \infty$ and $0<\delta \leq 1$ we say that $T$ is an \textbf{exponential Schr\"odinger--Calder\'on--Zygmund operator of $(\nu, s, \delta)$ type with parameters $c$ and $m$} if 
	\begin{enumerate}
		\item \label{itm: weak type T}  $T$ is of weak type $\left(s', \frac{s'd}{d-s'\nu}\right)$; 
		\item $T $ has an associated kernel $K: \RR^d \times \RR^d \rightarrow \RR$ in the sense that 
		\[Tf(x)=\int_{\RR^d}K(x,y)f(y)dy, \qquad f \in L^{s'}(\RR^d) \text{ and  }x \notin \supp (f),\]
		where the kernel $K$ satisfies the following conditions: there exists a constant $C $  such that 
			\begin{equation}\label{eq: size Hormander fract}
				\left(\frac{1}{R^d}\int_{R<|x_0-y|\leq 2R}|K(x,y)|^s dy \right)^{\frac1s}
				\leq 
				\frac{C}{R^{d-\nu}}\exp\left(-c \left(1+\frac{R}{\rho(x)}\right)^{m}\right),
			\end{equation}
			whenever $|x-x_0|<R/2$, and 
			\begin{equation}\label{eq: smoothness Hormander fract}
				\left(\frac{1}{R^d}\int_{R<|x_0-y|\leq 2R}|K(x,y)-K(x_0,y)|^s dy \right)^{\frac1s}
				\leq  
				\frac{C}{R^{d-\nu}}\left(\frac{r}{R}\right)^\delta, 
			\end{equation}
			for every $|x-x_0|<r\leq \rho(x_0)$ and $r<R/2$.
	\end{enumerate}

\end{defn}

 The limiting case $s=\infty$ is defined by pointwise estimates as we will see below.
	
	\begin{defn}
		For $0\leq \nu<d$ and $0 <\delta\leq 1$ we say that $T$ is an  \textbf{exponential  Schr\"odinger--Calder\'on--Zygmund operator of $(\nu,\infty,\delta)$ type  with parameters $c$ and $m$} if
		\begin{enumerate}
			\item $T$ is of weak type $(1, \frac{d}{d-\nu})$; 
			\item $T$ has an associated kernel $K:\mathbb{R}^d \times \RR^d \rightarrow \RR$ in the sense that \[Tf(x)=\int_{\RR^d}K(x,y)f(y)dy, \qquad f \in L^{1}(\RR^d) \text{ and  }x \notin \supp (f),\]
				where the kernel $K$ satisfies the following conditions: there exists a constant $C $  such that 
				\begin{equation}\label{eq: size pointwise fract}
					|{K(x,y)}| \leq  \frac{C}{|{x-y}|^{d-\nu}} \exp \left(-c \left(1+\frac{|{x-y}|}{\rho(x)}\right)^{m}\right), 
				\end{equation}
			for every  $x\neq y$, and 
				\begin{equation} \label{eq: smoothness pointwise fract}
					|K(x,y)-K(x_0,y)|\leq  \frac{C}{|x-y|^{d-\nu}}\left(\frac{|x-x_0|}{|x-y|}\right)^\delta,
				\end{equation}
				for every $|x-y|>2|x-x_0|$.
		\end{enumerate}  
	\end{defn}

For any $1<s\leq \infty$, we may refer to the above operators as \textit{exponential SCZOs of $(\nu, s,\delta)$ type}, for short. For these types of operators, we will specify in Section~\ref{sec: prelim} the meaning of $Tf$ when $f\in \BMO_\rho^\alpha(w)$.

\begin{rem}\label{obs: Hormander to pointwise}
	Note that if $K$ satisfies conditions \eqref{eq: size pointwise fract} and \eqref{eq: smoothness pointwise fract}, then it also satisfies  \eqref{eq: size Hormander fract} and \eqref{eq: smoothness Hormander fract} for all $1<s<\infty$, with the same parameters  $c,m$ and $\delta$.
\end{rem} 

\begin{rem}\label{obs: expSCZO Hormander Lp} 
	Although Definition~\ref{def: s-delta} only requires the weak type~$(s',s')$ when $\nu=0$, it follows from \cite[Proposition~6 and Theorem~5]{BCH13Extrapol} that exponential SCZOs of $(0,s,\delta)$ type are actually bounded on $L^p(\mathbb{R}^d)$ for every $s' < p < \infty$.  
	Moreover, exponential SCZOs of $(0,s,\delta)$ type are bounded on $L^p(w)$ for every $s' < p < \infty$ and every weight $w \in A_{p/s'}^\rho$, the class of weights introduced in \cite{BHS11}. Consequently, exponential SCZOs of $(0, \infty,\delta)$ type are bounded on $L^p(\mathbb{R}^d)$ for every $1< p < \infty$ and on $L^p(w)$ for every $1 < p < \infty$ and every weight $w \in A_{p}^\rho$. 
\end{rem} 
\begin{rem}
	Clearly, exponential SCZOs of $(0, s,\delta)$ type are exponential SCZOs of $(s,\delta)$ type as considered in \cite[Definition~2.6]{DLT25} for $1<s<\infty$. However, we observe that in \cite[Definition~2.7]{DLT25} we gave a slightly different definition for exponential SCZOs of $(0, \infty,\delta)$ type. In spite of that, from Remarks~\ref{obs: Hormander to pointwise} and \ref{obs: expSCZO Hormander Lp}, if $T$ is an exponential SCZO of $(0, \infty,\delta)$ type, it satisfies \cite[Definition~2.7]{DLT25} and all the results already obtained in \cite{DLT25} still hold for this new class. 
\end{rem}

\begin{rem}\label{obs: smoothness for smaller delta}
Notice that if condition \eqref{eq: smoothness Hormander fract} holds for some $0<\delta\leq 1$, it is also valid for any $\delta'\in (0,\delta)$.
\end{rem}

In this article, we will be dealing with one of the classes of weights studied in \cite{Bailey21}, which captures the exponential decay of the operators under consideration. 

For $1<p<\infty$ and $c,m \geq 0$ we say a weight $w\in H^{\rho,m}_{p,c}$ if there exists a constant $C>0$ for which
\begin{equation*}%\label{eq: class Hp}
	\left(\int_B w\right)^{\frac{1}{p}}\left(\int_B w^{-\frac{1}{p-1}}\right)^{\frac{p-1}{p}} \leq C|B| \exp \left(c\left(1+ \frac{r}{\rho(x)}\right)^m\right),
\end{equation*}
for every ball $B=B(x,r)$ with $x \in \RR^d$ and $r>0$.

When $p=1$, for $c,m\geq 0$ we denote by $H^{\rho,m}_{1,c}$ the collection of weights $w$ such that
\begin{equation*}
	\int_B w \leq C|B|\exp\left(c\left(1+ \frac{r}{\rho(x)}\right)^m\right)\inf_B w,
\end{equation*}
for each ball $B=B(x,r)$, and some constant $C>0$ independent of $B$.

We will also consider a class of doubling weights, and reverse--H\"older classes adapted to this context (already considered in \cite{DLT25}); they will be denoted by $D^{\rho}_{\kappa,c}$ with $\kappa \geq1$ and $c>0$ and $RH_{\eta,c}^\rho $ with $c>0$ and $\eta > 1$, respectively. Their precise definitions will be provided in Section~\ref{sec: prelim}. 

For simplicity, we denote by $E_{s,c_1,c_2}^{\rho,m}$ to the class of weights
\[E_{s,c_1,c_2}^{\rho,m}=\bigcup_{\eta>s'}H_{s/\eta',c_1}^{\rho,m}\cap RH_{\eta,c_2}^{\rho,m},\]
where $1< s<\infty$, $c_1,c_2,m\geq 0$. And we will write
\[E_{\infty,c_1,c_2}^{\rho,m}=\bigcup_{s> 1}E_{s,c_1,c_2}^{\rho,m}.\]
 
We are now ready to state the central theorem of this work.

\begin{thm}\label{thm: T1 criterion}
	Let $T$ be an exponential SCZO of $(\nu, s,\delta)$ type with parameters $c$ and $m$, for some $0\leq \nu<d$, $\frac{d}{d-\nu}<s<\infty$  and $0<\delta \leq 1$. 
	 Let $0\leq\alpha<\delta$  and $1\leq\kappa < \frac{\delta-\alpha-\nu}{d}+1$.
	 Then, the following statements are equivalent.
	\begin{enumerate}
	\item \label{itm: T1 condition} There exists a constant $C$ such that, for every ball $B=B(x_0,r)$, with $x_0 \in \RR^d$ and $0<r\leq\frac{1}{2}\rho(x_0)$, the function $T1$ satisfies
		\begin{equation}\label{eq: T1 power condition}
		\frac{1}{|B|^{1+\nu/d}}\int_B |T1(y)-(T1)_B|dy\leq C \left(\frac{r}{\rho(x_0)}\right)^{\alpha+d(\kappa-1)},
		\end{equation}
		when $\alpha>0$ or $\kappa>1$, or
		\begin{equation}\label{eq: T1 log condition}
			\frac{1}{|B|^{1+\nu/d}}\int_B |T1(y)-(T1)_B|dy\leq C \log^{-1}\left(\frac{\rho(x_0)}{r}\right),
		\end{equation}
		 when $\alpha=0$ and $\kappa=1$.
		
	\item \label{itm: T1 weights}  $T$ is bounded from $\BMO_\rho^\alpha(w)$ into $\BMO_\rho^{\alpha+\nu}(w)$ for every weight $w\in E_{s,c_1,c_2}^{\rho,m}\cap D_{\kappa,c_3}^{\rho,m}$
	with $c_1, c_2, c_3\geq 0$ such that $(c_2+c_3)<c\left(1-\frac{d(\kappa-1)+\alpha+\nu}{\delta}\right)(4C_0)^{-m}$. The operator norm depends on $w$ only through the parameters of the classes $E_{s,c_1,c_2}^{\rho,m}$ and $D_{\kappa,c_3}^{\rho,m}$. 
	
	\item \label{itm: T1 power weights} $T$ is bounded from $\BMO_\rho^\alpha(w)$ into $\BMO_\rho^{\alpha+\nu}(w)$ for every weight of the form \linebreak ${w(x)=|x-x_0|^ {d(\kappa-1)}}$, $x_0 \in \RR^d$, with operator norm independent of $x_0$.
	\end{enumerate}
\end{thm}

It is easy to see  from Remark~\ref{obs: Hormander to pointwise} and  Theorem~\ref{thm: T1 criterion} that for $s=\infty$ we get the following consequence.

\begin{cor}\label{cor: T1 criterion infty}
	Let $T$ be an exponential SCZO of $(\nu, \infty,\delta)$ type with parameters $c$ and $m$,  for some $0\leq \nu<d$ and $0<\delta \leq 1$. Let $0\leq\alpha<\delta$  and $1\leq\kappa < \frac{\delta-\alpha-\nu}{d}+1$. Then, the boundedness of $T$ from $\BMO_\rho^\alpha(w)$ into $\BMO_\rho^{\alpha+\nu}(w)$ is equivalent to either condition \ref{itm: T1 condition} or condition \ref{itm: T1 power weights} from Theorem~\ref{thm: T1 criterion}, whenever $w\in E_{\infty,c_1,c_2}^{\rho,m}\cap D_{\kappa,c_3}^{\rho,m}$ with $c_1, c_2, c_3\geq 0$ such that $(c_2+c_3)<c\left(1-\frac{d(\kappa-1)+\alpha+\nu}{\delta}\right)(4C_0)^{-m}$.
\end{cor}

\begin{rem}\label{obs: reduce to beta=0}
	Clearly, if \eqref{eq: T1 power condition} holds for some $\alpha > 0$ or $\kappa > 1$, then \eqref{eq: T1 log condition} follows, yielding the boundedness of the operator on $\BMO_\rho(w)$ (i.e., $\alpha = 0$). However, the significance of \eqref{eq: T1 log condition} when $\alpha = 0$ and $\kappa = 1$ lies in ensuring $\BMO_\rho(w)$ boundedness even for operators that fail \eqref{eq: T1 power condition}.
\end{rem}
 \begin{rem}\label{obs: vector valued}
	The results above can be applied to linear operators related to the generalized Schr\"odinger semigroup. However, non-linear operators in this setting, such as maximal or Littlewood-Paley functions, can be viewed as linear operators in appropriate Banach spaces. By using the corresponding norms in Theorem~\ref{thm: T1 criterion} and Corollary~\ref{cor: T1 criterion infty}, we obtain valid vector-valued versions of these results.
\end{rem}

 The paper is organized as follows. In Section~\ref{sec: prelim} we collect some important definitions and intermediate results to fully understand the theorem stated above. In Section~\ref{sec: proof T1 criterion} we prove Theorem~\ref{thm: T1 criterion}, as well as some auxiliary results related to it. Finally, in Section~\ref{sec: examples} we establish the $T1$ condition for several operators  associated to the generalized Schr\"odinger operator $\mathcal L_\mu$, defined in terms of a certain measure $\mu$ studied in \cite{Shen99} and \cite{Bailey21}, and we also apply it to the study of some operators associated to a potential~$V$.
 
 Throughout this paper $C$ and $c$ will always denote positive constants that may change in each occurrence. Also by $A \lesssim B$ will indicate that there exists a constant $c>0$ such that $A \leq cB$ is satisfied and $A \sim B$ will denote that, at the same time,  $A \lesssim B$ and  $B \lesssim A$ hold.

\section{Preliminaries}\label{sec: prelim}
  
  \subsection{Weights associated to a critical radius function}
  
  When dealing with operators associated to a critical radius function, we may first mention the work of Bongioanni, Harboure and Salinas \cite{BHS11}, where the authors introduced the class $A^{\rho}_{p}$ of weights for $1\leq p<\infty$ to obtain the weighted $L^p$-boundedness of many operators in the Schr\"odinger setting.
  
  For $1<p<\infty$, a weight $w$ is said to be in $A^{\rho}_p$ if $w\in A^{p,\theta}_p$ for some $\theta\geq 0$. This means that there exist $C>0$ and some $\theta\geq 0$ such that
  \begin{align*}
  	\left(\int_B w\right)^{\frac{1}{p}}\left(\int_B w^{-\frac{1}{p-1}}\right)^{\frac{p-1}{p}} \leq C|B| \left(1+ \frac{r}{\rho(x)}\right)^\theta,
  \end{align*}
  for every ball $B=B(x,r)$ with $x \in \RR^d$ and $r>0$. 
  
  Similarly, when $p=1$, a weight $w$ belongs to $A^{\rho}_1$ if  $w\in A_1^{\rho,\theta}$ for some $\theta\geq0$; equivalently, if there exist $C>0$ and $\theta\geq 0$ such that
  \begin{equation*}
  	\frac{1}{|B|}\int_B w \leq C \left(1+\frac{r}{\rho(x)}\right)^\theta \inf_B w,
  \end{equation*}
  for every ball $B=B(x,r)$.

  Observe that when $B\in \mathcal{B}_\rho$, the family defined in \eqref{eq: B_rho}, the polynomial term can be neglected, and one has that $A_p^\rho$ resembles the Muckenhoupt $A_p$ class  for sub-critical balls. However, on super-critical balls, the $A_p^\rho$ weights can be larger than the $A_p$ weights (and indeed, $A_p\subsetneq A_p^\rho$ for every $1\leq p<\infty$, as shown in \cite[p.~564]{BHS11}). In view of this remark, it will also be useful to have at our disposal the classes of local weights introduced in \cite{BCH19}. We say that $w\in A^{\rho,\loc}_{p}$ if there exists a constant $C>0$ such that 
  \begin{align*}
  \left(\int_{B} w\right)^{\frac{1}{p}}\left(\int_{B} w^{-\frac{1}{p-1}}\right)^{\frac{p-1}{p}} \leq C|B|,
  \end{align*}
  for every ball $B \in \mathcal{B}_\rho$, and for $p=1$ if 
  \begin{equation*}
  	\int_B w \leq C|B|\inf_B w,
  \end{equation*}
  holds for every $B \in \mathcal{B}_\rho$.

As we have established in \cite[Proposition~3.1]{DLT25}, for every $1\leq p<\infty$ and $c>0$, 
\begin{equation}\label{eq: classes inclusions}
	A_p^\rho\subsetneq H_{p,c}^{\rho}\subseteq A^{\rho,\loc}_{p},
\end{equation}
where $H_{p,c}^{\rho}:=\bigcup_{m\geq 0} H_{p,c}^{\rho,m}$. That is, the class $H_{p,c}^{\rho}$ is strictly larger than the one given in \cite{BHS11}, but smaller than the local class of weights $A^{\rho,\loc}_{p}$. 
  
  We now recall the definitions of the doubling and reverse--H\"older weight classes from \cite{DLT25}.
  
  \begin{defn}
  	Let $\kappa \geq1$ and $c,m\geq 0$. We denote by  $D^{\rho,m}_{\kappa,c}$ the class of weights $w$ such that there exists a constant $C$ for which
  	\begin{align*}
  		w(B(x,R))\leq C \left(  \frac{R}{r}\right)^{d\kappa}\exp{\left(c\left(1+ \frac{R}{\rho(x)}\right)^m \right)}w(B(x,r)),
  	\end{align*}
  	for all  $x \in \RR^d$ and $r\leq R$.
  \end{defn}
  
  \begin{rem}\label{obs: HincludeinD} It is easy to see that $H^{\rho,m}_{p,c}\subseteq D^{\rho,m}_{p,cp}$ for any $1\leq p<\infty$ and $c,m\geq 0$.
  	\end{rem}  
  	
  \begin{rem}
  	The condition $w\in D_{\kappa,c}^{\rho,m}$ says, in particular, that $w$ is doubling on critical balls. Indeed, for some $C=C(d,\kappa,c,m)$, we have
  	\[w(B(x,2\rho(x)))\leq Cw(B(x,\rho(x))).\]
  	We also note that, for any $x\in \RR^d$, $r>0$ and $j\in \NN$, there exists $C>0$ such that
  	\begin{equation}\label{eq: doubling 2^jr}
  	w(B(x,2^jr))\leq C 2^{d\kappa j}\exp\left(c\left(1+\frac{2^{j}r}{\rho(x)}\right)^m\right)w(B(x,r)),
  	\end{equation}
  	which, in the case $r= \rho(x)$ gives
  	\[w(B(x,2^j\rho(x)))\leq C 2^{d\kappa j}\exp\left(c\left(1+2^{j}\right)^m\right)w(B(x,\rho(x))).\]
  \end{rem}
  	
  Reverse--H\"older classes adapted to this context are defined as follows. 
  \begin{defn}
  	Let $\eta >1$ and $c,m\geq 0$. We denote by $RH_{\eta,c}^{\rho,m}$ the set of weights $w$ satisfying
  	\begin{equation*}
  		\left(\frac{1}{|B|}\int_B w ^{\eta}\right)^{\frac{1}{\eta}}\leq C  \exp \left(c\left(1+ \frac{r}{\rho(x)}\right)^m\right)  \left(\frac{1}{|B|}\int_B w\right),
  	\end{equation*}
  	for every ball $B=B(x,r)$, and some $C>0$ independent of $B$.
  \end{defn}

 \subsection{Spaces of functions}
 
 In the introduction we have already defined the spaces $\BMO_\rho^\alpha(w)$.

The first property that we present below shows us that it is enough to verify condition \eqref{eq: average BMO} only on critical balls when the weight belongs to the class $A_p^{\rho,\loc}$. 

\begin{lem}[{\cite[Proposition~3.2]{BCH19}}]\label{lem: sufficient BMO cond critical}
	Let $ w \in A^{\rho,\loc}_p$ for some $1 \leq p  < \infty$ and $f \in L^1_{\loc}(\mathbb R^d)$. If
	\begin{equation*}
	A:=\sup_{x \in \RR^d}\frac{1}{w(B(x,\rho(x)))\rho(x)^\alpha}\int_{B(x,\rho(x))}|f| < \infty,
	\end{equation*}
	then there exists a constant $C$ such that 
	\begin{equation*}
	\sup_{x \in \RR^d,\ r \geq \rho(x)}\frac{1}{w(B(x,r))r^{ \alpha}}\int_{B(x,r)} |f| < CA.
	\end{equation*}
\end{lem}

The following lemma gives us a property of norm equivalences for the space $\BMO_\rho^\alpha(w)$ for $A_p^{\rho,\loc}$ weights. As usual, the notation $p'$ stands for the conjugate exponent of $p$.  

\begin{lem}[{\cite[Lemma~3.3]{BCH19}}]\label{lem: lower bounds BMO norm}
	Let $0\leq \alpha < 1 $, $w \in \aloc$, $1<q\leq p'$ and $f \in \BMO_\rho^\alpha(w)$. Then,
	\begin{equation*}
	\frac{1}{|B|^{\alpha/d}}\left(\frac{1}{w(B)} \int_B |f|^q w^{1-q}\right)^{\frac{1}{q}}\leq C \|f\|_{\BMO_\rho^\alpha(w)},
	\end{equation*}
	for all  $B=B(x,r)$ with $r \geq \rho(x)$, and
	\begin{equation*}
	\frac{1}{|B|^{\alpha/d}}\left(\frac{1}{w(B)} \int_B |f-f_B|^q w^{1-q}\right)^{\frac{1}{q}}\leq C \|f\|_{\BMO_\rho^\alpha(w)},
	\end{equation*}
	for all $B=B(x,r)$  with $r\leq \rho(x)$. 
\end{lem}

By virtue of \eqref{eq: classes inclusions}, both lemmas above can be applied to weights in the $H_{p,c}^{\rho,m}$ classes for any $c,m\geq 0$ and the corresponding values of $p$. This fact will be useful in the proof of Theorem~\ref{thm: T1 criterion}.

\subsection{Schr\"odinger-Calder\'on-Zygmund operators of exponential decay} 

We will need to make some considerations about the meaning of $Tf$ being $T$ an exponential SCZO of $(\nu, s, \delta)$ type with parameters $c$ and $m$, for $0\leq \nu<d$, $1<s\leq \infty$ and $0<\delta\leq 1$,  and $f$ a function in $\BMO_\rho^\alpha(w)$.

Let $x_0 \in \RR^d$ and $R \geq \rho(x_0)$. We define 
\begin{equation}\label{eq: def Tf}
	Tf(x)=T(f\chi_{B(x_0,2R)})(x)+ \int_{B(x_0,2R)^c}K(x,y)f(y)dy, \qquad \quad x \in B(x_0,R).
\end{equation}
The first term on the right hand side makes sense if $w$ satisfies the hypothesis of Theorem~\ref{thm: T1 criterion}  since in that case $f\chi_{B(x_0,2R)}  \in L^{s'}(\RR^d) $ for $f \in \BMO_\rho^\alpha(w)$ as we can see in \eqref{eq: bound s' norm g}. 
The integral in the second term is absolutely convergent (see  estimate of $Tf_2$ in  page~\pageref{pag: proof Tf2}). This fact follows in the same manner as in \cite[p.~603]{BHQ19}, as does its well-definedness in the sense that equation \eqref{eq: def Tf} is independent of $R$. Moreover, we need  to verify that \eqref{eq: def Tf} is well-defined for $f = 1$. Indeed, $\chi_{B(x_0 ,2R)} \in L^{s'}(\mathbb{R}^d)$, which ensures the finiteness of the first term and,  for $x \in B(x_0 , R)$ 
\begin{align*}
\int_{B(x_0 ,2R)^c} |K(x, y)|  dy
&
\leq \sum_{j \geq 1} \left(\int_{2^j R \leq |x_0 - y| < 2^{j+1} R} |K(x, y)|^s  dy\right)^{1/s} |B(x_0, 2^{j+1} R)|^{1/s'} \\
&
\leq C \sum_{j \geq 1} (2^{j}R)^{\nu}\exp\left( -c \left( 1 + \frac{2^j R}{\rho(x)} \right)^m \right) \\
&
\leq C R^\nu \sum_{j \geq 1}  (2^\nu)^j\exp\left( -c^* \left( 1 + 2^j \right)^m \right) < \infty,
\end{align*}
where $c^*=c\, C_0^{-m}2^{-\frac{k_0m}{k_0+1}}$ by \eqref{eq: critical radius function}. Here, the constant $C > 0$ is independent of $R > 0$ and $x \in \mathbb{R}^d$.

The following lemma will be useful to show that exponential SCZO of $(\nu, s, \delta)$ type  actually exhibits a smoothness condition on their kernels with exponential decay. 

\begin{lem}\label{lem: mixed smooth condition} 
	Let $K$ be a kernel such that \eqref{eq: size Hormander fract} and \eqref{eq: smoothness Hormander fract} hold for some $0\leq \nu<d$, $1<s<\infty$, $0<\delta\leq 1$, and $c,m\geq 0$. Then, there exists a constant $C>0$ such that for every $\delta'\in (0,\delta)$ and $R>0$
	\begin{equation*}
		\left(\frac{1}{R^d} \int_{R<|x_0-y|\leq 2R} |K(x,y)-K(x_0,y)|^s \, dy\right)^{\frac1s}
		\leq 
		\frac{C}{R^{d-\nu}} \left(\frac{r}{R}\right)^{\delta'} \exp\left(-c\sigma \left(1+\frac{R}{{2C_0\rho(x_0)}}\right)^{m}\right),
	\end{equation*}
	 where $|x-x_0|<r\leq \rho(x_0)$, $r<R/2$ and $\sigma=1-\frac{\delta'}{\delta}$.
\end{lem}

\begin{proof}
	Let $\delta' \in (0,\delta)$. We define  $\sigma \in (0,1)$ such that $(1-\sigma)\delta=\delta'$. Let $x_0,x \in \RR^d$ and $r,R>0$ such that $|x-x_0|<r\leq \rho(x_0)$, $r<R/2$. 
	
	If we call
	\[A=\left(\frac{1}{R^d}\int_{R<|x_0-y|\leq 2R}|K(x,y)-K(x_0,y)|^s dy \right)^{\frac1s},\]
	by \eqref{eq: size Hormander fract} and \eqref{eq: equiv critical radius} we get
	\begin{align*}
		A^\sigma
		&\leq 
		\left(\frac{1}{R^d}\int_{R<|x_0-y|\leq 2R}|K(x,y)|^s+|K(x_0,y)|^s\, dy
		\right)^{\sigma/s}\\
		&\lesssim \frac{1}{R^{(d-\nu)\sigma}}\exp \left(-c\sigma \left(1+\frac{R}{\rho(x)}\right)^{m}\right)+\frac{1}{R^{(d-\nu)\sigma}}\exp \left(-c\sigma \left(1+\frac{R}{\rho(x_0)}\right)^{m}\right)\\
		&\lesssim \frac{1}{R^{(d-\nu)\sigma}}\exp \left(-c\sigma \left(1+\frac{R}{{2C_0\rho(x_0)}}\right)^{m}\right).
	\end{align*}
	On the other hand, by  \eqref{eq: smoothness Hormander fract}  we have
	\begin{align*}
		A^{1-\sigma}
		&\lesssim
		\frac{1}{R^{(d-\nu)(1-\sigma)}}\left(\frac{r}{R}\right)^{{\delta}(1-\sigma)}\sim 
		\frac{1}{R^{(d-\nu)(1-\sigma)}}\left(\frac{r}{R}\right)^{\delta'}.
	\end{align*}
	Since $A=A^\sigma A^{1-\sigma}$, the result follows.
\end{proof}

The following inequality is a generalization of the classical Kolmogorov inequality for sublinear operators of weak-type $(1,1)$ (see, for instance, \cite[Lemma~5.16]{D01}) to sublinear operators of weak-type $(p,q)$ for any $p,q\geq 1$, and the proof can be obtained in a similar way (see \cite[Theorem~3.3.1]{dGM81}). It will allow us to deduce sharper results when dealing with exponential SCZO of $(s,\delta)$ type. 
 
\begin{lem}\label{lem: Kolmogorov gral}
	Let $S$ be a sublinear operator of weak-type $(p,q)$ for some $p,q\geq 1$. Then, there exists $C>0$ such that for any measurable set $E$ with $0<|E|<\infty$ and $0<\gamma<q$, 
	\[\int_E |Sf(x)|^\gamma dx\leq C|E|^{1-\frac{\gamma}{q}}\|f\|_p^{\gamma},\]
for $f\in L^p(\RR^d)$.
\end{lem}

\section{Proof of Theorem~\ref{thm: T1 criterion}}\label{sec: proof T1 criterion}

Before presenting the proof of Theorem~\ref{thm: T1 criterion}, we first state and prove some technical lemmas. Although we essentially follow the approach in \cite{BHQ19}, we outline the arguments here with the necessary modifications concerning the classes of weights, the doubling condition, the reverse--H\"older property, among other considerations.

	 \begin{lem}\label{lem: bounds fB}
		Let $f \in \BMO_\rho^\alpha(w)$  with $0\leq \alpha <1$, $ w \in D^{\rho,m}_{\kappa,c}$ for some $\kappa\geq 1$ and $c>0$, and $B=B(x_0,r)$ with $r<\rho(x_0)$. Then, if $\kappa>1$ or $\alpha>0$, there exists a constant $C=C(c,m,d)$ such that 
		\begin{equation}\label{eq: fB power bound}
			|f_B|  	 \leq C \|f\|_{\BMO_\rho^\alpha(w)}r^{\alpha} \frac{w(B)}{|B|} \left(\frac{\rho(x_0)}{r}\right)^{d(\kappa-1)+\alpha}.
		\end{equation}
		When $\kappa=1$ and $\alpha=0$, there exists a constant  $C=C(c,m,d)$ such that 
		\begin{equation}\label{eq: fB log bound}
			|f_B|  	 \leq C \|f\|_{\BMO_\rho^\alpha(w)}r^{\alpha} \frac{w(B)}{|B|} \left(1+\log_2\left(\frac{\rho(x_0)}{r}\right)\right).
		\end{equation}
	\end{lem}
	\begin{proof}
		Let $f \in \BMO_\rho^\alpha(w)$ and fix $B=B(x_0,r)$ with $r<\rho(x_0)$. We choose $j_0 \in \NN$ such that $2^{j_0-1}r < \rho(x_0) \leq 2^{j_0}r$, and define $B_j= 2^j B$ for every $j=0,\dots, j_0-1$. We have			
		\begin{align*}
			|f_B| &
			  \leq \frac{1}{|B|}\int_B |f-f_B|+ \sum_{j=1}^{j_0-1} |f_{B_{j-1}}- f_{B_j}|+ |f_{B_{j_0-1}}|\\
			&
			 \leq   \frac{1}{|B|}\int_B |f-f_B|+ \sum_{j=1}^{j_0-1}\frac{2^d}{|B_j|}\int_{B_j}|f-f_{B_j}|+ \frac{2^d}{|B_{j_0}|}\int_{B_{j_0}} |f|\\
			&
			\leq   \sum_{j=0}^{j_0-1} \frac{2^d}{|B_j|}\int_ {B_j} |f-f_{B_j}|+ \frac{2^d}{|B_{j_0}|}\int_{B_{j_0}}|f|\\
			&
			 \leq  2^d\|f\|_{\BMO_\rho^\alpha(w)} \sum_{j=0}^{j_0} \frac{w(B_j)}{|B_j|}|B_j|^{\frac{\alpha}{d}},
		\end{align*}
		where the last inequality was obtained from \eqref{eq: average osc BMO} and  \eqref{eq: average BMO} since $2^{j_0-1}r \leq \rho(x_0) \leq 2^{j_0}r$. Applying the doubling property \eqref{eq: doubling 2^jr} of $w$ we get that
		\begin{align*}
			|f_B|
			& 
			\leq 2^d  \|f\|_{\BMO_\rho^\alpha(w)} \frac{w(B)}{|B|}|B|^{\frac{\alpha}{d}}\sum_{j=0}^{j_0} 2^{(d \kappa + \alpha-d)j} \exp \left(c\left(1+\frac{2^jr}{\rho(x_0)}\right)^m\right)\\
			& \leq C \|f\|_{\BMO_\rho^\alpha(w)} \frac{w(B)}{|B|}r^{\alpha} 2^{(j_0+1)(d(\kappa-1)+\alpha)}\\
			& 
			\leq C \|f\|_{\BMO_\rho^\alpha(w)} \frac{w(B)}{|B|}r^{\alpha} \left(\frac{\rho(x_0)}{r}\right)^{d(\kappa-1)+\alpha},
		\end{align*}
	whenever $\kappa>1$ or $\alpha>0$, where $C$ depends on the parameters of the doubling condition and on the dimension.
	
	If $\kappa=1$ and $\alpha=0$, the sum obtained above equals $j_0+1$. Thus, from the choice of $j_0$, in this case we have
	\begin{equation*}
		|f_B|
		\leq 2^d e^{c3^m} \|f\|_{\BMO_\rho^\alpha(w)} \frac{w(B)}{|B|}r^\alpha (j_0+1)\leq C\|f\|_{\BMO_\rho^\alpha(w)} \frac{w(B)}{|B|}r^\alpha \left(1+\log_2\left(\frac{\rho(x_0)}{r}\right)\right),
	\end{equation*}
	where $C$ is as in the previous case, and depends on $c,m$ and $d$.
	\end{proof}
	
	\begin{lem}
		Let $f \in \BMO_\rho^\alpha(w)$  with $0\leq \alpha<1$, and a weight $ w \in H^{\rho,m}_{\sigma',c_1}\cap D^{\rho,m}_{\kappa,c_3}$ for some $\sigma> 1$, $\kappa\geq 1$, $c_1,c_3\geq0$ and $m\geq 0$. Then, when $\kappa>1$ or $\alpha>0$, there exists $C>0$ such that		
	\begin{align*}
			 w(2^jB)^{1/\sigma'}&\left(\int_{2^{j}B}|f-f_B|^\sigma w^{1-\sigma} \right)^{1/\sigma} \nonumber\\
			 &\leq C \|f\|_{\BMO_\rho^\alpha(w)}w(B)r^\alpha  2^{j(\alpha+d\kappa)}\exp\left((c_1+c_3)\left(1+\frac{2^jr}{\rho(x_0)}\right)^m\right), \nonumber
		\end{align*}
	 for every $j\in \NN$ and every ball $B=B(x_0,r)$.
	 
	 When $\kappa=1$ and $\alpha=0$, 
	there exists $C>0$ such that		
	\begin{align*}
		w(2^jB)^{1/\sigma'}&\left(\int_{2^{j}B}|f-f_B|^\sigma w^{1-\sigma} \right)^{1/\sigma} \nonumber\\
		&\leq C \|f\|_{\BMO_\rho^\alpha(w)}w(B)j  2^{jd}\exp\left((c_1+c_3)\left(1+\frac{2^jr}{\rho(x_0)}\right)^m\right), \nonumber
	\end{align*}
	for every $j\in \NN$ and every ball $B=B(x_0,r)$.
	\end{lem}
	
	\begin{proof}
		We can write, for $j\in \NN$,
		\begin{align*}
			w(&2^jB)^{1/\sigma'}\left(\int_{2^{j}B}|f-f_B|^\sigma w^{1-\sigma}\right)^{1/\sigma} \\
			&
				\leq w(2^jB)^{1/\sigma'}\left[\left(\int_{2^{j}B} |f-f_{2^{j}B}|^\sigma w^{1-\sigma} \right)^{1/\sigma}+ \left(w^{1-\sigma}(2^jB)\right)^{1/\sigma} \sum_{i=1}^{j}|f_{2^iB}-f_{2^{i-1}B}|\right] \\
			& 
				\leq w(2^jB)\left(\frac{1}{w(2^jB)}\int_{2^{j}B} |f-f_{2^{j}B}|^\sigma w^{1-\sigma} \right)^{1/\sigma}\\
				&\quad + w(2^jB)^{1/\sigma'}\left(w^{-\frac{1}{\sigma'-1}}(2^jB)\right)^{1/\sigma} \sum_{i=1}^{j}\frac{2^d}{|2^{i}B|}\int_{2^i B}|f-f_{2^i B}|\\
			& 
			:= I_1 + I_2.  
		\end{align*}
		
		 To estimate $I_1$, we can apply Lemma~\ref{lem: lower bounds BMO norm} with $q=p'=\sigma'$ since $w\in A_{\sigma'}^{\rho,\loc}$ (by  \eqref{eq: classes inclusions}), together with the doubling condition of $w$, to get  
		\begin{align*}
			I_1 
			&
			\lesssim \|f\|_{\BMO_\rho^\alpha(w)} w(2^jB) |2^jB|^{\alpha/d}
			\lesssim\|f\|_{\BMO_\rho^\alpha(w)}w(B)r^\alpha 2^{j(d\kappa+\alpha)} \exp\left(c_3\left(1+  \frac{2^jr}{\rho(x_0)}\right)^{m}\right).
		\end{align*}
		
		For $I_2$, note that since $w\in H_{\sigma',c_1}^{\rho,m}$,
		\[w(2^jB)^{1/\sigma'}\left(w^{-\frac{1}{\sigma'-1}}(2^jB)\right)^{1/\sigma}\leq C 2^{jd}r^d\exp\left(c_1\left(1+\frac{2^jr}{\rho(x_0)}\right)^{m}\right).\]

		On the other hand, since $f \in \BMO_\rho^\alpha(w)$ and $w\in D^{\rho,m}_{\kappa,c_3}$, if we first assume $\kappa>1$ or $\alpha>0$,
		\begin{align*}
			\sum_{i=1}^{j}\frac{1}{|2^{i}B|}\int_{2^i B}|f-f_{2^i B}|&\leq \|f\|_{\BMO_\rho^\alpha(w)} \sum_{i=1}^{j} |2^{i}B|^{\alpha/d-1} w(2^iB)\\
			&\leq C\|f\|_{\BMO_\rho^\alpha(w)} r^{\alpha-d}w(B) \sum_{i=1}^{j} 2^{i(\alpha-d)} 2^{id\kappa} \exp\left(c_3\left(1+\frac{2^i r}{\rho(x_0)}\right)^{m}\right)\\
			&\leq C\|f\|_{\BMO_\rho^\alpha(w)} r^{\alpha-d}w(B) \exp\left(c_3\left(1+\frac{2^j r}{\rho(x_0)}\right)^{m}\right)2^{j(d(\kappa-1)+\alpha)},
		\end{align*}
		where we have used that $d(\kappa-1)+\alpha>0$ to bound the sum in this case.
		
		Combining the above estimates, we get
		\begin{align*}
			I_1+I_2&\leq C \|f\|_{\BMO_\rho^\alpha(w)}w(B) r^\alpha 2^{j(d\kappa+\alpha)} \exp\left((c_1+c_3)\left(1+\frac{2^j r}{\rho(x_0)}\right)^{m}\right).
		\end{align*}
		
		If $\kappa=1$ and $\alpha=0$, we get instead
		\begin{align*}
			\sum_{i=1}^{j}\frac{1}{|2^{i}B|}\int_{2^i B}|f-f_{2^i B}|&\leq C \|f\|_{\BMO_\rho^\alpha(w)} j r^{-d}w(B) \exp\left(c_3\left(1+\frac{2^j r}{\rho(x_0)}\right)^{m}\right),
		\end{align*}
		which combined with the estimate for $I_1$, leads to the result. 
				In either case, the constant $C$ is independent of $j$ and $B$.
	\end{proof}
		
\begin{lem}\label{lem: bound g function}
	Let $f \in \BMO_\rho^{\alpha}(w)$ with $0\leq \alpha<1$, and  $ w\in E_{s,c_1,c_2}^{\rho,m}\cap D_{\kappa,c_3}^{\rho,m}$ for some $s>1$, $c_1,c_2, c_3\geq0$, $\kappa\geq 1$ and $m\geq 0$. Given a ball $B=B(x_0,r)$, consider the function
	\[g=\begin{dcases*}
			f\chi_{2B} & if $r\geq \rho(x_0)$;\\
			(f-f_B)\chi_{2B} & if $r< \rho(x_0)$.
		\end{dcases*}\]
	Then, there exists $C$ such that for every $j\in \NN$,
	\begin{equation*}
			\left(\int_{{2^j}B}|g|^{s'} \right)^\frac{1}{s'}\leq C \|f\|_{\BMO_\rho^\alpha(w)}{2^{j\left(d\left(\kappa-\frac{1}{s}\right)+\alpha\right)}w(B)|B|^{\frac{\alpha}{d}- \frac{1}{s}}}\exp\left({(c_2+c_3)}\left(1+\frac{  {2^j}r}{\rho(x_0)}\right)^{m}\right).
		\end{equation*}
\end{lem}

\begin{proof}
	Fix $\eta>s'$ such that $w \in H_{s/\eta',c_1}^{\rho,m}\cap RH_{\eta,c_2}^{\rho,m}$. We set $\zeta=\frac{s-1}{s-\eta'}>1$ and apply H\"older's inequality with this exponent to get
	\begin{equation}\label{eq: bound s' norm g}
		\left(\int_{{2^j}B}|g|^{s'} \right)^\frac{1}{s'}
		\leq \left(\int_{{2^j}B}|g|^{s'\zeta} w^{1-s'\zeta}\right)^\frac{1}{s'\zeta}\left(\int_{{2^j}B} w^{\left(s'-\frac{1}{\zeta}\right)\zeta'}\right)^\frac{1}{s'\zeta'}.
	\end{equation}
	
	Since $w \in A^{\rho,\loc}_{s/\eta'}$ (by \eqref{eq: classes inclusions}), and noting that $s'\zeta=\left(\frac{s}{\eta'}\right)'$, we can use Lemma~\ref{lem: lower bounds BMO norm} in order to obtain
		\begin{equation*}
			\left(\int_{{2^j}B}|g|^{s'\zeta} w^{1-s'\zeta}\right)^\frac{1}{s'\zeta}
			\leq C\| f \|_{\BMO_\rho^\alpha(w)} w({2^j}B)^{\frac{1}{\left(\frac{s}{\eta'}\right)'}}|{2^j}B|^{\frac{\alpha}{d}}=C\| f \|_{\BMO_\rho^\alpha(w)} w({2^j}B)^{\frac{1}{\left(\frac{s}{\eta'}\right)'}}|B|^{\frac{\alpha}{d}}2^{j\alpha},
			\end{equation*}
		for the corresponding cases of $g$. 

	To estimate the second factor in \eqref{eq: bound s' norm g}, we observe that $\left(s'-\frac{1}{\zeta}\right)\zeta'=\eta$ and $s'\zeta'=\frac{s\eta}{\eta'}$, so we can use that  $w \in RH_{\eta,c_2}^{\rho,m}$ to get
	\begin{align*}
		\left(\int_{2^jB} w^{\left(s'-\frac{1}{\zeta}\right)\zeta'}\right)^\frac{1}{s'\zeta'}&=\left(\int_{{2^j}B} w^{\eta}\right)^{\frac{1}{\eta}\frac{\eta'}{s}}\lesssim \left(|{2^j}B|^{-\frac{1}{\eta'}} w({2^j}B) \exp\left(c_2\left(1+\frac{{2^j}r}{\rho(x_0)}\right)^m\right)\right)^{\frac{\eta'}{s}}\\
		&\leq C  2^{-j\frac{d}{s}}|B|^{-\frac{1}{s}} w({2^j}B)^{\frac{1}{\frac{s}{\eta'}}}\exp\left(c_2\left(1+\frac{{2^j}r}{\rho(x_0)}\right)^m\right).
	\end{align*}
	Combining both estimates above, and since $w\in D_{\kappa,c_3}^{\rho,m}$, from \eqref{eq: doubling 2^jr}  we get
	\begin{align*}
		\left(\int_{{2^j}B}|g|^{s'} \right)^\frac{1}{s'} & 
		\leq C \|f\|_{\BMO_\rho^\alpha(w)}{2^{j\left(\alpha-\frac{d}{s}\right)}}w({2^j}B)|B|^{\frac{\alpha}{d}- \frac{1}{s}}\exp\left(c_2\left(1+\frac{  {2^j}r}{\rho(x_0)}\right)^{m}\right)\\
&\leq C  \|f\|_{\BMO_\rho^\alpha(w)}{2^{j\left(d\left(\kappa-\frac1s\right)+\alpha\right)}}w(B)|B|^{\frac{\alpha}{d}- \frac{1}{s}}\exp\left({(c_2+c_3)} \left(1+\frac{ {2^j}r}{\rho(x_0)}\right)^{m}\right), 
	\end{align*}
	as claimed.
\end{proof}

 	Now, we are in a position to prove Theorem~\ref{thm: T1 criterion}. 
	\begin{proof}[Proof of Theorem~\ref{thm: T1 criterion}]
			To see that \ref{itm: T1 condition}$\Rightarrow$\ref{itm: T1 weights}, let  $ f \in \BMO_\rho^{\alpha}(w)$ and fix 
\begin{equation*}w\in E_{s,c_1,c_2}^{\rho,m}\cap D_{\kappa,c_3}^{\rho,m}= \left(\bigcup_{\eta>s'}H_{s/\eta',c_1}^{\rho,m}\cap RH_{\eta,c_2}^{\rho,m}\right)\cap D_{\kappa,c_3}^{\rho,m},
	\end{equation*}
			with $(c_2+c_3)(4C_0)^m<c$.
			 To verify condition \eqref{eq: average BMO}, by  \eqref{eq: classes inclusions} and  Lemma~\ref{lem: sufficient BMO cond critical}, it is enough to consider critical balls. Let $B_\rho=B(x_0,\rho(x_0))$ with $x_0 \in \RR^d$. 
			 
			 According  to the definition of  $Tf$,  we decompose the function as $f=f_1+f_2=f\chi_{2B_\rho}+f\chi_{(2B_\rho)^c}.$
		This allows us to bound
		\begin{equation*}
			\int_{B_\rho}| Tf(x)| dx 
				\leq \int_{B_\rho} | Tf_1(x) | dx + \int_{B_\rho}| Tf_2(x)| dx.
		\end{equation*}
	
		We begin by controlling the first term.
		From the hypothesis on $w$ we can use Lemma~\ref{lem: bound g function} with $j=1$ and $r=\rho(x_0)$ to get
		\begin{align*}
			\left(\int_{2B_\rho}|f|^{s'}\right)^{\frac{1}{s'}}
			&\leq C \|f\|_{\BMO_\rho^{\alpha}(w)} |B_\rho|^{\frac{\alpha}{d}-\frac{1}{s}}w(B_\rho),
		\end{align*}
		where the constant $C$ depends on the parameters of the classes of weights.
		
		Since $T$ is of weak-type $\left(s', {\frac{s'd}{d-s'\nu}}\right)$ with $s',\frac{s'd}{d-s'\nu}\geq 1$, combining Lemma~\ref{lem: Kolmogorov gral} with $\gamma=1$ and the previous inequality, we deduce
		\begin{align*}
			\int_{B_\rho} |Tf_1(x)| dx &\leq C |B_\rho|^{1-{\frac{1}{s'}+\frac{\nu}{d}}}\left( \int_{2B_\rho} |f|^{s'} \right) ^{\frac{1}{s'}}\leq C \|f \|_{\BMO_\rho^{\alpha}(w)}|B_\rho|^{\frac{\alpha+\nu}{d}} w(B_\rho).
		\end{align*} 
		
		We next estimate the integral involving $Tf_2$, for which the kernel representation of $T$ will be used. We fix $x \in B_\rho$ and make the standard decomposition into annuli, \refstepcounter{pagina}\label{pag: Tf2}
		\begin{align*}
			|Tf_2(x)|
			&\leq
			\sum_{j \in \NN}\int_{ 2^{j+1}B_\rho \setminus 2^jB_\rho} |K(x,y)||f(y)|dy\nonumber\\
			&\leq 
			\sum_{j \in \NN} \left(\int_{2^j\rho(x_0)\leq |y-x_0|<2^{j+1}\rho(x_0)} |K(x,y)|^s dy\right)^{\frac1s} \left(\int_{ 2^{j+1}B_\rho} |f|^{s'}\right)^{\frac{1}{s'}}.
		\end{align*}
	
		By the size condition \eqref{eq: size Hormander fract} on the kernel with $R=2^j\rho(x_0)$, $j\geq 1$, applying again Lemma~\ref{lem: bound g function} on each term, and  using \eqref{eq: equiv critical radius}, we can see that for every $x\in B_\rho$, 
			\begin{align*}
			|Tf_2(x)|&\lesssim \|f\|_{\BMO_\rho^\alpha(w)}w(B_\rho)|B_\rho|^{\frac{\alpha}{d}-\frac1s}\\
			&\quad \times \sum_{j \in \NN}  \frac{2^{(j+1)\left(d\left(\kappa-\frac1s\right)+\alpha\right)}}{\left(2^j\rho(x_0)\right)^{\frac{d}{s'}{-\nu}}}\exp\left((c_2+c_3)(1+2^{j+1})^m-c\left(1+\frac{2^{j}\rho(x_0)}{\rho(x)}\right)^m\right)\\
			&\lesssim \|f\|_{\BMO_\rho^\alpha(w)}w(B_\rho)|B_\rho|^{\frac{\alpha{+\nu}}{d}-1}\\
			&\quad \times\sum_{j \in \NN}   \exp\left((c_2+c_3)(1+2^{j+1})^{m}-c\left(1+\frac{1}{4C_0}2^{j+1}\right)^{m}\right)2^{j(d(\kappa-1)+\alpha{+\nu})}\\
			&\lesssim  \|f\|_{\BMO_\rho^\alpha(w)}w(B_\rho)|B_\rho|^{\frac{\alpha{+\nu}}{d}-1}\\
			&\quad \times\sum_{j \in \NN}   \exp\left(\left((c_2+c_3)-\frac{c}{(4C_0)^m}\right)(1+2^{j+1})^{m}\right)2^{k(d(\kappa-1)+\alpha{+\nu})}\\
			& \lesssim \|f\|_{\BMO_\rho^\alpha(w)}w(B_\rho)|B_\rho|^{\frac{\alpha{+\nu}}{d}-1},
		\end{align*}
		where the series converges {since $(c_2+c_3)<c(4C_0)^{-m}$}. Finally, integrating on $B_\rho$, we have that $Tf_2$ satisfies the condition of averages \eqref{eq: average BMO} over critical balls. This finishes the proof of \eqref{eq: average BMO} for $Tf$.
		
		  Now, we shall see that condition  \eqref{eq: average osc BMO} holds for $Tf$, for every $B=B(x_0,r)$ with $r<\rho(x_0)$ as we already observed it is sufficient to check it in this type of balls. 
		 
		 Fix $x_0 \in \RR^d, r<\rho(x_0)$ and let  $B=B(x_0,r)$. We split $f$ in the following way,
		 \begin{align*}
		 	f= (f- f_B)\chi_{2B}+(f-f_B)\chi_{(2B)^c}+f_B= f_1+ f_2+ f_3.
		 \end{align*}
		 We choose $R \geq \max\{\rho(x_0),2r\}$, which yields $\tilde{B}:= B(x_0,R) \supset 2B$. Applying the definition given in \eqref{eq: def Tf} to the above decomposition of $f$, and adding and subtracting $f_B$ in the integral over $\tilde{B}^c$,
		 \begin{align*}
		 Tf(x)&=T(f\chi_{\tilde{B}})(x)+ \int_{\tilde{B}^c} K(x,y)f(y)dy\\
		 	&\left. = T((f-f_B)\chi_{2B})+T((f-f_B)\chi_{\tilde{B}\setminus 2B})(x)+f_B T(\chi_{\tilde{B}})(x) \right.\\
		 	 &\left.  + \int _{\tilde{B}^c} K(x,y)(f-f_B)+f_B\int_{\tilde{B}^c}K(x,y)dy \right.\\
		 	&=T((f-f_B)\chi_{2B})(x)+\int_{(2B)^c} K(x,y)(f(y)-f_B)dy+f_BT1(x), \qquad \text{a.e. } x \in 2B.
		 \end{align*}
		 So, to estimate \eqref{eq: average osc BMO} for $Tf$, we write
		 \begin{equation*}
		 	\begin{split}
		 		\int_B |Tf(x)-(Tf)_B|dx  &\leq \frac{1}{|B|}\int_B \int_B |Tf_1(x)-Tf_1(z)|dzdx\\
		 		& + \frac{1}{|B|}\int_B \int_B \int_{(2B)^c}|K(x,y)-K(z,y)||f(y)-f_B|dydzdx \\
		 		& + \int_B |Tf_3(x)-(Tf_3)_B| dx.
		 	\end{split}
		 \end{equation*}
		 
		 For the first term, we can proceed as before, applying Lemma~\ref{lem: bound g function} to $f_1=(f- f_B)\chi_{2B}$ and Lemma~\ref{lem: Kolmogorov gral} to obtain the desired estimate
		 \[\frac{1}{|B|}\int_B \int_B |Tf_1(x)-Tf_1(z)|dzdx\leq 2\int_B |Tf_1(x)|dx \leq C \|f \|_{\BMO_\rho^{\alpha}(w)}|B|^{\frac{\alpha+\nu}{d}} w(B).\]
		 
		 \refstepcounter{pagina}\label{pag: proof Tf2}For the second term, given $x,z \in B$, we repeat the argument given on page~\pageref{pag: Tf2}. In this case, we apply it on the ball $B$ instead of $B_\rho$ and use Lemma~\ref{lem: bound g function} together with the smoothness condition obtained in Lemma~\ref{lem: mixed smooth condition} for $\delta'<\delta$ twice, with $R=2^jr$, $j\geq 1$. In fact, 	\begin{align*}
		 	&\int_{(2B)^c}|K(x,y)-K(z, y)||f(y)-f_B|dy\nonumber\\
		 	&\leq  \sum_{j\in \NN}\left(\int_{2^jr\leq |x_0-y|<2^{j+1}r}\left(|K(x,y)-K(x_0,y)|+|K(z,y)-K(x_0,y)|\right)^s dy\right)^{\frac1s}\left(\int_{ 2^{j+1}B}|f-f_B|^{s'}\right)^{\frac{1}{s'}}\\
		 	&\lesssim  \|f\|_{\BMO_\rho^\alpha(w)}w(B)|B|^{\frac{\alpha{+\nu}}{d}-1}\sum_{j\in \NN} 2^{j\left(d\left(\kappa-1\right)+\alpha{+\nu}-\delta'\right)}\\
		 	&\quad\times  \exp\left(-c{\sigma}\left(1+\frac{2^{j}r}{2C_0\rho(x_0)}\right)^m+(c_2+c_3)\left(1+\frac{2^{j+1}r}{\rho(x_0)}\right)^m\right)\\
		 	&\lesssim  \|f\|_{\BMO_\rho^\alpha(w)}w(B)|B|^{\frac{\alpha{+\nu}}{d}-1}\sum_{j\in \NN } 2^{j\left(d\left(\kappa-1\right)+\alpha{+\nu}-\delta'\right)}\exp\left(\left(c_2+c_3-\frac{c\sigma}{(4C_0)^m}\right)\left(1+\frac{2^{j+1}r}{\rho(x_0)}\right)^m\right)\\
		 	&\lesssim  \|f\|_{\BMO_\rho^\alpha(w)}w(B)|B|^{\frac{\alpha{+\nu}}{d}-1}.
		 	 \end{align*}
		 Choosing first $\sigma\in(0,1)$ such that $\sigma < 1-\frac{d(\kappa-1)+\alpha}{\delta}$ and $(c_2+c_3)<c\sigma(4C_0)^{-m}$, we also have that  $d(\kappa-1)+\alpha<(1-\sigma)\delta=\delta'$ and  the series converges.

		  Finally, to estimate the term with $f_3$ we use Lemma~\ref{lem: bounds fB} and the  $T1$ condition. In the case $\kappa>1$ or $\alpha>0$, we use \eqref{eq: fB power bound} and \eqref{eq: T1 power condition} to get
		  \begin{align*}
		  	\int_B |Tf_3(x)-(Tf_3)_B|dx
		  	&=
		  	\int_B\left|f_BT1(x)-\frac{1}{|B|}\int_B|f_B T1(z)|dz\right|dx\\
		  	& \leq 
		  	|f_B|\int_B |T1(x)-(T1)_B|dx\\
		  	& \lesssim
		  	 \|f\|_{\BMO_\rho^\alpha(w)} \frac{w(B)}{|B|} r^{\alpha} \left(\frac{\rho(x_0)}{r}\right)^{d(\kappa-1)+\alpha}\int_B |T1(x)-(T1)_B|dx\\
		  	& \lesssim
		    \|f\|_{\BMO_\rho^\alpha(w)} w(B)|B|^{\frac{\alpha{+\nu}}{d}} \left(\frac{\rho(x_0)}{r}\right)^{d(\kappa-1)+\alpha}\left(\frac{r}{\rho(x_0)}\right)^{\alpha+d(\kappa-1)}\\
		  	& \lesssim 
		  	 \|f\|_{\BMO_\rho^\alpha(w)} w(B) |B|^{\frac{\alpha{+\nu}}{d}}.
		  \end{align*}
		  
		  If $\kappa=1$ and $\alpha=0$, we use instead \eqref{eq: fB log bound} and \eqref{eq: T1 log condition} to have
		  \begin{align*}
		  	\int_B |Tf_3(x)-(Tf_3)_B|dx
		  	&\lesssim  \|f\|_{\BMO_\rho^\alpha(w)} |B|^{\frac{\alpha{+\nu}}{d}} w(B)\left(1+\log_2\left(\frac{\rho(x_0)}{r}\right)\right) \log^{-1}\left(\frac{\rho(x_0)}{r}\right)\\
		  	&\lesssim \|f\|_{\BMO_\rho^\alpha(w)} w(B) |B|^{\frac{\alpha{+\nu}}{d}},
		  \end{align*}
		  since $\frac{\rho(x_0)}{r}\geq 2$. The proof of the first implication is now complete.
		  
		  To see that  \ref{itm: T1 weights}$\Rightarrow$\ref{itm: T1 power weights}, let  $w_{x_0}(x)=|x-x_0|^{d(\kappa -1)}$, which is the same type of weight considered in \cite[Theorem~2(c)]{BHQ19}. From that proof, we know that $w_{x_0}\in D_\kappa\cap A_{s/\eta'}\cap RH_\eta$ for every $\eta\geq 1$ and $s>\eta' \kappa$, being $\kappa\geq 1$. That is, $w_{x_0}$ belongs to $E_{s,0,0}^{\rho,m}\cap D_{\kappa,0}^{\rho,m}$, which trivially satisfies $(c_2+c_3)(4C_0)^m=0<c$, and thus the implication is proved.

		  Finally, the proof of \ref{itm: T1 power weights}$\Rightarrow$ \ref{itm: T1 condition}  is analogous to the corresponding one given in \cite[Theorem~2]{BHQ19}. Here, we use that $w_{x_0}\in E_{s,0,0}^{\rho,m}\cap D_{\kappa,0}^{\rho,m}$ and the first part of the proof where the condition on the parameters is trivially satisfied. 
		\end{proof}
		  
		  \section{Applications}\label{sec: examples}

		In this section	we consider the generalized 
		Schr\"odinger operator with measure $\mu$,
		  \[\mathcal{L}_\mu=-\Delta +\mu, \]
		  where $\mu$ is a nonnegative Radon measure on $\RR ^d$ and  $d\geq 3$. This differential operator and some singular integrals associated with it were first investigated in detail by Z. Shen in \cite{Shen99}. Other operators related with $\mathcal{L}_\mu$ were later considered in \cite{Bailey21,  CT21, WuYan16}, and the authors in \cite{DLT25}. 
		  	
		  	Associated with $\mathcal{L}_\mu$, we will study several operators in order to establish weighted endpoint results by means of Theorem~\ref{thm: T1 criterion}. In each case, we will briefly recall their boundedness properties on weighted and unweighted $L^p$ spaces. In many cases, the results we shall obtain are not only new for the operators associated with $\mathcal{L}_\mu$, but also for $\mathcal{L}_V$ when $V$ is a potential.  
		  
		  In order to achieve positive results, we will deal with a measure $\mu$   that satisfies the following conditions, as considered in \cite{Shen99}: there exist constants   $\delta_\mu, C_\mu>0$ and $D_\mu \geq 1$  such that
		\begin{align}\label{eq: prop1 mu}
			\mu(B(x,r)) \leq C_\mu \left( \frac{r}{R}\right)^{d-2+\delta_\mu} \mu(B(x,R))
		\end{align}
		and 
		\begin{align}\label{eq: prop2 mu}
			\mu(B(x,2r))\leq D_\mu \left(\mu(B(x,r))+r ^{d-2}\right)
		\end{align}
		for all $x \in \RR^d$ and $0<r<R$. 
	
	When $d\mu(x)=V(x)dx$, with $V$ a nonnegative function in the class $RH_{q}$ for some $q>\frac d2$, the above conditions hold. In this case, it can be seen that the parameter $\delta_\mu=:\delta_V=2-\frac qd>0$ (see \cite[Lemma~1.2]{Shen95}).
		
	From  \eqref{eq: prop1 mu} it can be proved that (see \cite[Remark~0.13]{Shen99})
	\begin{equation}\label{eq: prop1 integral mu}
		\int_{B(x,R)}\frac{d\mu(y)}{|y-x|^{d-2}}\lesssim \frac{\mu(B(x,R))}{R^{d-2}}, 
	\end{equation} 
	and if $\delta_\mu>1$ then we also have 
	\begin{equation}\label{eq: prop2 integral mu}
		\int_{B(x,R)}\frac{d\mu(y)}{|y-x| ^{d-1}}\lesssim\frac{\mu(B(x,R))}{R^{d-1}}, 
	\end{equation}
	for all $x \in \RR^d$ and $R>0$. Both conditions were already known for $d\mu(x)=V(x)dx$ (see \cite[Lemma~2.6]{BMR21}). 
	
	For any nonnegative Radon measure $\mu$ on $\RR ^d$, $d\geq 3$, satisfying \eqref{eq: prop1 mu} and \eqref{eq: prop2 mu}, the function given by
	\begin{equation}\label{eq: rho_mu}
		\rho_\mu(x):=\sup \left\{ r>0: \frac{\mu(B(x,r))}{r^{d-2}}\leq 1\right\}, \quad x \in \RR^d,
	\end{equation}
	is a critical radius function.  
	
	A useful property that can be obtained immediately from the definition of $\rho_\mu$ is the following: for every $x \in \RR^d$
	\begin{equation} \label{eq: equiv 1}
		\frac{\mu(B(x,\rho_\mu(x)))}{\rho_\mu(x)^{d-2}}\sim 1.
	\end{equation}
	
	Since all the results on the previous sections were obtained in terms of the inequalities in \eqref{eq: critical radius function}, we may not need, in general, the definition \eqref{eq: rho_mu} of $\rho_\mu$. 
	
	In this setting, the natural metric is given by the Agmon distance, defined for any critical radius function $\rho$ as
	\begin{equation}\label{eq: Agmon def}
		d_\rho(x,y):= \inf_\gamma \int_{0}^{1} \rho(\gamma(t))^{-1}|\gamma'(t)|dt,
	\end{equation} 
	where the infimum is taken over every absolutely continuous function   ${\gamma:[0,1]\rightarrow \mathbb{R}^d}$ with  $\gamma(0)=x$ and $\gamma(1)=y$. When $\rho=\rho_\mu$, we will denote $d_\rho(x,y)=d_\mu(x,y)$ to emphasize the measure dependence. 
	
{In the following sections, for $\lambda\geq 0$ we will denote by $\Gamma_{\mu + \lambda}$ and $\Gamma_\lambda$ the fundamental solutions of $-\Delta+\mu+\lambda$ and $-\Delta+\lambda$, respectively.  
		In particular, $\Gamma_0$ will be the fundamental solution of $-\Delta$. Also, by $d_{\mu+\lambda}$ and $d_\lambda$ we mean the Agmon distances related to the critical radius functions $\rho_{\mu+\lambda}$ and $\rho_\lambda$ with respect to the measures $d\mu(x)+\lambda dx$ and $\lambda dx$, respectively. Actually, from the definition \eqref{eq: Agmon def}, we have that $d_\lambda(x,y)=\sqrt{\omega_d\lambda}|x-y|$, where $\omega_d$ is the volume of the unit ball in $\RR^d$ (see \cite[p.~28]{Bailey21}).}

	\subsection{Riesz transforms}
	
		Associated to $\mathcal{L}_\mu$ we define  the singular integral operators \linebreak ${\mathcal{R}_\mu=\nabla  (-\Delta + \mu)^{-\frac{1}{2}}}$ and  ${\mathcal{R}^*_\mu=(-\Delta + \mu)^{-\frac{1}{2}} \nabla  }$  as in \cite{Shen99}, which are the Riesz transforms in this setting.
	
	The Riesz transform $\mathcal{R}_\mu$ can be expressed, for $f\in C_c^\infty(\RR^d)$ and $x\notin \supp(f)$, as 
	\begin{align*}
		\mathcal{R}_\mu f(x)
		& = \frac{1}{\pi}\int_{0}^{\infty} \lambda ^{-\frac{1}{2}} \nabla (-\Delta + \mu +\lambda)^{-1}f(x) d\lambda\\
		& = \frac{1}{\pi} \int_{0}^{\infty} \lambda^{-\frac{1}{2}} \int_{\RR^d} \nabla_1  \Gamma_{\mu+ \lambda }(x,y)f(y)dy \,d\lambda
\\&  = \int_{\RR^d} K_\mu(x,y) f(y)dy,
	\end{align*}
	where $K_\mu $ is the singular kernel of $\mathcal{R}_\mu$ given by
	\begin{equation}\label{eq: Riesz kernel}
		K_\mu(x,y)= \frac{1}{\pi} \int_{0}^{\infty}\lambda^{-\frac{1}{2}}  \nabla _1 \Gamma_{\mu+ \lambda }(x,y)d\lambda. 
	\end{equation}
	The operator $\mathcal{R}_\mu^*$ also has an associated kernel, which is $K^*_\mu(x,y)=-K_\mu(y,x)$.

	 In \cite[Theorem~0.20]{Shen99} it was established that $\mathcal{R}_\mu$ and $\mathcal{R}_\mu^*$ are Calder\'on-Zygmund operators when $\delta_\mu>1$. So, in that case, $\mathcal{R}_\mu$ and $\mathcal{R}_\mu^*$ are bounded on $L^p(\RR^d)$, for every ${1<p< \infty}$.   Moreover, \cite[Proposition~5.3]{DLT25} assures that,  if $\delta_\mu>1$, both $\mathcal{R}_\mu$ and $\mathcal{R}_\mu^*$ are  exponential SCZO of $(0, \infty, \delta)$ type, with parameters $c=\frac{\epsilon}{2D_1}$ and $m=\frac{1}{k_0+1}$, where $\epsilon$ and $D_1$ depend on the Agmon distance $d_\mu$. Although the parameter $\delta$ was not previously specified, from Lemma~\ref{lem: conditions Riesz kernel delta_mu>1} it follows that $\delta=\min\{1,\delta_\mu-1\}$ for $\mathcal R_\mu$, and by \cite[Eq. (7.29)]{Shen99} there exists some $\delta\in (0,1)$ that gives the type for $\mathcal R_\mu^*$. This fact led us to recover the weighted $L^p$--boundedness result for $\mathcal R_\mu$ and $\mathcal R_\mu^*$ (see \cite[Theorem~5.4]{DLT25}) previously obtained in \cite[Theorem~4.1(i)]{Bailey21}, but gave us a better understanding of the kernel properties that not only these transforms have.

	When $0<\delta_\mu<1$, $\mathcal R_\mu$ and $\mathcal R_\mu^*$ behave differently on $L^p$ spaces, even in the potential case. It is well-known (see \cite{Shen95}) that, whenever $V\in RH_{q}$, $\frac d2<q<d$, $\mathcal{R}_V$ is bounded on $L^p(\RR^d)$ for every $1< p\leq p_0$ and that $\mathcal{R}_V^*$ is bounded on $L^p(\RR^d)$ for every $p_0'\leq p<\infty$, being $\frac{1}{p_0}=\frac1q-\frac1d$. This implies that they are not Calderón-Zygmund operators when $0<\delta_V<1$, where $\delta_V=2-\frac dq$.
	 Analogous results were obtained for the generalized Riesz transforms  $\mathcal{R}_\mu$ and $\mathcal{R}_\mu^*$ in \cite{Shen99} on the unweighted case, and in \cite{Bailey21} and \cite{DLT25} with weights. In fact, in \cite{Bailey21} it is proved that $\mathcal{R}_\mu$ is bounded on $L^p(w)$ for every $1<p<\eta$ and $w^{-1/(p-1)}\in S_{p'/\eta',c}^\mu$ for some $c>0$ and any $2<\eta<(2-\delta_\mu)'$, and $\mathcal{R}_\mu^*$ is bounded on $L^p(w)$ for every $\eta'<p<\infty$ and $w\in S_{p/\eta',c}^\mu$ (for the definitions of these classes, see the mentioned article). In  \cite{DLT25} we obtained by extrapolation a similar result but with weights in the $H_{p,c}^{\rho,m}$ classes, which are equivalent, in some sense, to the $S_{p,c}^\mu$ classes as stated in \cite[Proposition~3.2]{Bailey21}. We also proved (\cite[Proposition~5.5]{DLT25}) that $\mathcal{R}_\mu^*$ is an exponential SCZO of $(0,s, \delta)$ type, with parameters $c=\frac{\epsilon}{4D_1}$ and $m=\frac{1}{k_0+1}$, provided that $1<s<(2-\delta_\mu)'$ and for some $\delta\in (0,1)$ (see \cite[Eq. (7.29)]{Shen99}).

	First, note that, since $\mathcal{R}_\mu^*(1)=0$, condition \ref{itm: T1 condition} of Theorem~\ref{thm: T1 criterion} is trivially verified. Therefore, the following result holds for $\mathcal{R}_\mu^*$, taking into account \cite[Propositions~5.3 and 5.5]{DLT25}, Theorem~\ref{thm: T1 criterion} and Corollary~\ref{cor: T1 criterion infty}.
	
	\begin{thm}Let $m_0=\frac{1}{k_0+1}$, $\delta\in (0,1)$ as above,  $0\leq \alpha<\delta$ and $1\leq \kappa<\frac{\delta-\alpha}{d}+1$, and $w$ a weight.
		\begin{enumerate}
			\item If $\delta_\mu>1$, $\mathcal{R}_\mu^*$ is bounded on $\BMO_\rho^\alpha(w)$ if and only if $w\in E_{\infty,c_1,c_2}^{\rho,m_0}\cap D_{\kappa,c_3}^{\rho,m_0}$ with $c_1, c_2, c_3\geq 0$ such that ${(c_2+c_3)<\frac{\epsilon}{2D_1}\left(1-\frac{d(\kappa-1)+\alpha}{\delta}\right)(4C_0)^{-m_0}}$. 
			\item If $0<\delta_\mu<1$ and $1<s<\frac{2-\delta_\mu}{1-\delta_\mu}$, $\mathcal{R}_\mu^*$ is bounded on $\BMO_\rho^\alpha(w)$ if and only if $w\in E_{s,c_1,c_2}^{\rho,m_0}\cap D_{\kappa,c_3}^{\rho,m_0}$ with $c_1, c_2, c_3\geq 0$ such that ${(c_2+c_3)<\frac{\epsilon}{4D_1}\left(1-\frac{d(\kappa-1)+\alpha}{\delta}\right)(4C_0)^{-m_0}}$.
		\end{enumerate}
	\end{thm}
	
	The behavior of $\mathcal R_\mu$ when $0<\delta_\mu<1$
	says that we can only expect to obtain endpoint results on $\BMO_\rho^\alpha(w)$  when $\delta_\mu>1$. 	In fact, we obtain $\BMO_{\rho}^{\alpha}(w)$ boundedness for $R_{\mu}$ provided $1 < \delta_{\mu} < d$, as well as the analogue for the particular case $R_{V}$. This result characterizes weighted boundedness for weights featuring exponential growth (cf. \cite[Theorem~7]{BHQ19} for weights with polynomial growth).
	
	\begin{thm}\label{thm: BMO Riesz transform}
		Let $m_0=\frac{1}{k_0+1}$, $\delta=\min\{1,\delta_\mu-1\}$, $0\leq \alpha<\delta$ and $1\leq \kappa<\frac{\delta-\alpha}{d}+1$. Then, if   $1<\delta_\mu<d$, $\mathcal{R}_\mu$ is bounded on $\BMO_\rho^\alpha(w)$ if and only if $w\in E_{\infty,c_1,c_2}^{\rho,m_0}\cap D_{\kappa,c_3}^{\rho,m_0}$ with $c_1, c_2, c_3\geq 0$ such that $(c_2+c_3)<\frac{\epsilon}{4D_1}\left(1-\frac{d(\kappa-1)+\alpha}{\delta}\right)(4C_0)^{-m_0}$.
			\end{thm}
	
	The proof of the previous theorem will be given at the end of this section, as several preliminary lemmas will be required. The first one was already obtained by Shen within the proof of \cite[Theorem~7.18]{Shen99}.
	 
	 \begin{lem}[{\cite[(7.20) and (7.26)]{Shen99}}]\label{lem: conditions Riesz kernel delta_mu>1}
	 	Let $\delta_\mu>1$. Then there exist constants $C,\epsilon>0$ such that
	 	\begin{enumerate}
	 		\item  for every $x\neq y$,
	 		\begin{equation}\label{eq: size Riesz kernel delta_mu>1}
	 			|K_\mu(x,y)|\leq C 
	 			\frac{e^{-\epsilon d_\mu(x,y)}}{|x-y|^{d}}.
	 		\end{equation}
	 		\item   When  $|x-y|>2|x-z|$, 
	 			 		\begin{equation} \label{eq: smoothness Riesz kernel delta_mu>1}
	 			|K_\mu(x,y)-K_\mu(z,y)|\leq  \frac{C}{|x-y|^{d}}\left(\frac{|x-z|}{|x-y|}\right)^{{\delta_\mu-1}}.
	 		\end{equation}
	 	\end{enumerate}
	 \end{lem}

	 In the next result we will denote by $K_0$ the kernel of the classical Riesz transform $\mathcal{R}_0=\nabla(-\Delta)^{-\frac12}$.

	\begin{lem}\label{lem: bounds diff K0} Let $1<\delta_\mu<d$ and {$0<\delta\leq \min\{1,\delta_\mu-1\}$.} 	Then there exists a constant $C>0$ such that 
		\begin{enumerate}
			\item \label{itm: diff Riesz mu and classic} For every $ x\neq y$, 
				\begin{equation}\label{eq: comparison Riesz kernels}
				|K_\mu(x,y)- K_0(x,y)| \leq \frac{C}{|x-y|^d}\left(\frac{|x-y|}{\rho(x)}\right)^{\delta_\mu}.
			\end{equation}
			
			\item \label{itm: diff of diff Riesz mu and classic} When $|x-y|\geq 2|x-z|$,
			\begin{equation} \label{itm: diffs of Riesz kernels}
				\left|[K_\mu(x,y)-K_0(x,y)]-[K_\mu(z,y)-K_0(z,y)]\right| \leq C\frac{|x-z|^\delta}{|x-y|^{d+\delta}} \left(\frac{|x-y|}{\rho(x)}\right)^{\delta_\mu}.
			\end{equation}
			 		\end{enumerate}
	\end{lem}
	
	{\begin{rem}
		We note that the restriction $\delta_\mu<d$ is asked only for the proof of Lemma~\ref{lem: bounds diff K0}\ref{itm: diff of diff Riesz mu and classic} and the proof of Theorem~\ref{thm: BMO Riesz transform}, independently. Although the properties \eqref{eq: prop1 mu} and \eqref{eq: prop2 mu} for $\mu$ do not seem to imply that $\delta_\mu$ is smaller than the dimension, we observe that this is automatically satisfied in the potential case. Indeed, if $V\in RH_{q}$ for some $q>d$, $\delta_V=2-d/q<2<d$ since $d\geq 3$ and $q>0$.  Therefore, we recover \cite[Lemma~4]{BHS09Riesz} when $d\mu(x)=V(x)dx$ with $V$ as before since, in that case, the smoothness condition for $K_V$ holds for any $0<\delta<1-d/q=\min\{1,\delta_V-1\}$.
	\end{rem}}
	
	\begin{proof} Regarding \ref{itm: diff Riesz mu and classic}, note that if $|x-y|> \rho(x)$, the result is true since both are Calderón-Zygmund kernels. If $|x-y|\leq \rho(x)$, we use the following estimate (see proof of \cite[Lemma~7.13]{Shen99}):
	\begin{equation*}
		|K_\mu(x,y)-K_0(x,y)| \lesssim   \frac{1}{|x-y|^{d-1}}\left(  \int_{B \left(x,|x-y|/2\right)} \frac{d\mu(z)}{|z-x|^{d-1}} + \frac{1}{|x-y|}  \left(\frac{|x-y|}{\rho(y)}\right)^{\delta_\mu} \right), \quad x\neq y.
	\end{equation*}
	By \eqref{eq: prop2 integral mu}, \eqref{eq: prop1 mu} and \eqref{eq: equiv 1}, we obtain that the integral given above can be estimated as follows
	\begin{equation*}
		\int_{B \left(x,|x-y|/2\right)} \frac{d\mu(z)}{|z-x|^{d-1}} 
		\lesssim  \frac{\mu(B(x,\rho(x)))}{|x-y|^{d-1}}\left( \frac{|x-y|}{2\rho(x)}\right)^{d-2+\delta_\mu} 
		\sim \frac{1}{|x-y|} \left( \frac{|x-y|}{\rho(x)}\right)^{\delta_\mu}. 
	\end{equation*}
	Then, we get 
	\begin{align*}
		|K_\mu(x,y)-K_0(x,y)|
		&
		\lesssim \frac{1}{|x-y|^{d}}  \left(\frac{|x-y|}{\rho(x)}\right)^{\delta_\mu}.
	\end{align*}
	
	To prove item \ref{itm: diff of diff Riesz mu and classic}, let $|x-y|\geq 2 |x-z|$. Note that when $|x-y|\geq \rho(x)$, both operators have kernels that satisfy a Calder\'on-Zygmund smoothness estimate on the first variable. More precisely,
	\begin{align*}
		|K_\mu(x,y)-K_\mu(z,y)|+|K_0(x,y)-K_0(z,y)|&\leq \frac{C}{|x-y|^{d}}\left(\frac{|x-z|}{|x-y|}\right)^{\delta_\mu-1}+\frac{C}{|x-y|^{d}}\frac{|x-z|}{|x-y|}\\
		&\leq C \frac{|x-z|^{\delta}}{|x-y|^{d+\delta}}\left(\frac{|x-y|}{\rho(x)}\right)^{\delta_\mu}
	\end{align*}
	for any $0<\delta\leq \min\{1, \delta_\mu-1\}$. 
	
	Now suppose $|x-y|< \rho(x)$. According to  \eqref{eq: Riesz kernel}, the difference of the kernels can be written as 
	\begin{equation*}
		K_\mu(x,y)-K_0(x,y)= \frac{1}{\pi} \int_{0}^{\infty}\lambda^{-\frac{1}{2}} ( \nabla _1 \Gamma_{\mu+ \lambda }(x,y)-\nabla _1 \Gamma_{\lambda}(x,y)) d\lambda.
	\end{equation*}	
	We will rewrite the integrand above, so we recall that $u(\cdot,y)=\Gamma_{\mu +\lambda}(\cdot,y)-\Gamma_{\lambda}(\cdot,y)$ (as a function of the first variable) satisfies the equation $(-\Delta + \lambda) u=- \mu \Gamma_{\mu +\lambda}$. Then, we get
		\begin{equation*}
		\Gamma_{\mu +\lambda}(x,y)- \Gamma_{\lambda}(x,y)= - \int_{\RR^d}\Gamma_{\lambda}(x,v)\Gamma_{\mu +\lambda}(v,y)d\mu(v).
	\end{equation*}
	Hence, 
	\begin{equation*}
		K_\mu(x,y)-K_0(x,y)=- \frac{1}{\pi} \int_{0}^{\infty}\lambda^{-\frac{1}{2}} \left(\int_{\RR^d}\nabla_1 \Gamma_{\lambda}(x,v)\Gamma_{\mu +\lambda}(v,y  )d\mu(v) d\lambda \right).
	\end{equation*}
	Consequently, the difference of interest is
	\begin{align}\label{eq: integral diff of diff Riesz mu and classic}
	\nonumber	[K_\mu(x,y)-K_0(x,y)]&-[K_\mu(z,y)-K_0(z,y)]\\
		&= - \frac{1}{\pi} \int_{0}^{\infty}\lambda^{-\frac{1}{2}} \int_{\RR^d}(\nabla_1 \Gamma_{\lambda}(x,v)- \nabla_1  \Gamma_{\lambda}(z,v))\Gamma_{\mu +\lambda}(v,y)d\mu(v) d\lambda.
	\end{align}
	We will deal first with the absolute value of the inner integral before performing the integration on $\lambda$. As in the proof of  \cite[Lemma~4]{BHS09Riesz}, we consider four regions covering $\RR^d$,
		\begin{align*}
		E_1&= \left\{ v\in \RR^d : |v-x|<\tfrac{3}{2}|x-z|\right\},\\
		E_2&= \left\{v\in \RR^d : \tfrac{3}{2}|x-z|\leq |v-x| < \tfrac{1}{2}|x-y|\right\},\\
		E_3&= \left\{ v\in \RR^d : \tfrac{1}{2}|x-y|\leq |v-x|<2|x-y|\right\},\\
		E_4&= \left\{ v\in \RR^d : |v-x|\geq 2|x-y|\right\},
	\end{align*}
	and we denote 
	\[I_j= \int_{E_j}|\nabla_1 \Gamma_{\lambda}(x,v)- \nabla_1  \Gamma_{\lambda}(z,v)||\Gamma_{\mu +\lambda}(v,y)|d\mu(v), \quad j=1,2,3,4.\]
	
Provided we can show that for some $N>0$ and any $k\geq N$,
	\begin{equation}\label{eq: bounds Ij}
	I_j\lesssim \frac{C_{k,d}|x-z|^\delta}{(1+\lambda^{1/2}|x-y|)^{k\epsilon_2}|x-y|^{d+\delta-1}}\left(\frac{|x-y|}{\rho(x)}\right)^{\delta_\mu}, \quad \forall \ 0<\delta<\min\{1,\delta_\mu-1\}, \; j=1,2,3,4,
	\end{equation}
	integrating in $\lambda$ the expression \eqref{eq: integral diff of diff Riesz mu and classic}, and 	choosing $k> \max\{N,\frac{1}{\epsilon_2}\}$ we obtain  \eqref{itm: diffs of Riesz kernels}. Therefore, let us prove \eqref{eq: bounds Ij}.
		
		 For $I_1$, we bound by the sum of the gradients and estimate each integral separately. That is
	\[\int_{|v-x|<\frac{3}{2}|x-z|}|\nabla _1\Gamma_{\lambda}(x,v)||\Gamma_{\mu +\lambda}(v,y)|d \mu(v)+\int_{|v-x|<\frac{3}{2}|x-z|}|\nabla _1\Gamma_{\lambda}(z,v)||\Gamma_{\mu +\lambda}(v,y)|d \mu(v):=I_{1,1}+I_{1,2}.\]
	
	On each one, we will use estimates for the gradient of the fundamental solution $\Gamma_{\lambda}$, for $\lambda\geq 0$. It is known, and can be found in \cite[Eq.~(4.4)]{Shen95}, that for any $k>0$, there exists $C_k>0$ such that
	\begin{equation}\label{eq: bound gradient1 Gamma_lambda}
		|\nabla _1\Gamma_{\lambda}(x,v)| \leq \frac{C_k}{(1+ \lambda^{1/2}|x-v|)^{k}}\frac{1}{|x-v|^{d-1}}, \quad x\neq v,
	\end{equation}
	and an analogous bound for $|\nabla _1\Gamma_{\lambda}(z,v)|$.
	
To estimate $\Gamma_{\mu +\lambda}(v,y)$, we use \cite[Eq.~(3.12)]{Shen99} along with the inequality (see \cite[Eq.~(21)]{Bailey21})
\begin{equation}\label{eq: metric mu+lambda} d _{\mu + \lambda}(v,y)\geq \frac{1}{2}d_\mu (v,y)+\frac{\sqrt{\omega_d}}{2}\sqrt{\lambda}|v-y|.
\end{equation}  Thus, for any $k>0$, there exist positive constants $C, \epsilon_1$, and $\epsilon_2$, independent of $v$ and $y$, such that the following holds

\begin{equation}\label{eq: bound Gamma_mu+lambda everywhere}
	0\leq \Gamma_{\mu + \lambda}(v,y) \leq \frac{C e^{- \epsilon_1 d_{\mu+\lambda}(v,y)}}{|v-y|^{d-2}}
		\leq
	\frac{C_{k,d}\, e^{-\epsilon_2 d_\mu(v,y)}}{|v-y|^{d-2}(1+ \lambda^{1/2}|v-y|)^{k\epsilon_2}}	
	, \quad v\neq y,
\end{equation}
	Since 	$|x-y|\geq 2 |x-z|$, for  $v \in E_1$
	we have $|v-y|\geq \frac{1}{4}|x-y|$. Then, 
	\begin{equation}\label{eq: bound Gamma_mu+lambda on E1}
		\Gamma_{\mu + \lambda}(v,y)\leq \frac{C_{k,d}}{|x-y|^{d-2}(1+ \lambda^{1/2}|x-y|)^{k\epsilon_2}}, \quad x\neq y.
	\end{equation}
	
	By \eqref{eq: bound gradient1 Gamma_lambda} and \eqref{eq: bound Gamma_mu+lambda on E1}, and applying the properties \eqref{eq: prop2 integral mu}, \eqref{eq: prop1 mu} and \eqref{eq: equiv 1} for the measure $\mu$, we obtain 
	\begin{align*}
		I_{1,1}
		&
		\leq 
		\frac{C_{k,d}}{(1+\lambda^{1/2}|x-y|)^{k\epsilon_2}|x-y|^{d-2}} \int_{B\left(x,\frac32|x-z|\right)} \frac{d\mu(v)}{|x-v|^{d-1}}\nonumber\\
		&\lesssim  
		\frac{C_{k,d}}{(1+\lambda^{1/2}|x-y|)^{k\epsilon_2}|x-y|^{d-2}} \frac{\mu\left(B(x,\frac32|x-z|)\right)}{|x-z|^{d-1}}\nonumber\\
		&\lesssim  
		\frac{C_{k,d}}{(1+\lambda^{1/2}|x-y|)^{k\epsilon_2}|x-y|^{d-2}} \frac{\mu(B(x,\rho(x)))}{\rho(x)^{d-2}}\frac{|x-z|^{\delta_\mu-1}}{\rho(x)^{\delta_\mu}}\nonumber\\
		&
		\lesssim
		\frac{C_{k,d}|x-z|^{\delta}}{(1+\lambda^{1/2}|x-y|)^{k\epsilon_2}|x-y|^{d-1+\delta}} \left(\frac{|x-y|}{\rho(x)}\right)^{\delta_\mu},
	\end{align*}
 for $\delta=\delta_\mu-1$. Since $|x-y|<\rho(x)$, the above inequality also holds for any $0<\delta<\delta_\mu-1$. 
 
 For $I_{1,2}$ we use again \eqref{eq: bound Gamma_mu+lambda on E1}, \eqref{eq: bound gradient1 Gamma_lambda} with $z$ instead of $x$, and take into account that $B(x,\frac32 |x-z|)\subseteq B(x,\frac52 |x-z|)$ 
so the argument follows in the same way. Combining both bounds, we have 
\begin{equation}\label{eq: estimate of I_1}
	I_1\lesssim \frac{C_{k,d}|x-z|^\delta}{(1+\lambda^{1/2}|x-y|)^{k\epsilon_2}|x-y|^{d-1+\delta}} \left(\frac{|x-y|}{\rho(x)}\right)^{\delta_\mu}, \quad {0<\delta\leq \delta_\mu-1.}
 \end{equation}
 
 Next, to take care of the integrals in the other three regions, we  apply the Mean Value Theorem and use an estimate for $\nabla_1^2\Gamma_{\lambda}$, which is also known (see \cite[Eq.~(4.4)]{Shen95}). Then we have that  
 for every $k>0$, there exists $C_k>0$ such that	
 \begin{align}\label{eq: difference gradients}
 	|\nabla_1 \Gamma_{\lambda}(x,v)- \nabla_1 \Gamma_{\lambda}(z,v)|&=\left|\nabla_1^2 \Gamma_{\lambda}(\xi,v)\right||x-z|\nonumber\\
 	&\leq
 	\frac{C_k|x-z|}{(1+ \lambda^{1/2}|\xi-v|)^k|\xi-v|^d}\leq
 	\frac{C_k|x-z|}{(1+ \lambda^{1/2}|x-v|)^k|x-v|^d},
 \end{align}
 where  $\xi=(1-\theta)x+\theta z$, for some $\theta\in (0,1)$.

Then, using the above inequality and the estimate \eqref{eq: bound Gamma_mu+lambda everywhere} for $\Gamma_{\mu + \lambda}$, for every $j=2,3,4$, we get
\begin{equation*}
	I_{j}
	\leq 
	\int_{E_j} \frac{C_{k,d}|x-z|}{(1+ \lambda^{1/2}|x-v|)^k|x-v|^d} \frac{e^{- \epsilon_2 d_\mu(v,y)}}{|v-y|^{d-2}(1+ \lambda^{1/2}|v-y|)^{k\epsilon_2}}d\mu(v).
\end{equation*}	

For $j=2$ we use that $|v-y|>\frac12 |x-y|$ for every $v\in E_2$ to get
	\begin{align*}
		I_2 & \lesssim  \frac{C_{k,d} |x-z|^\delta}{(1+ \lambda^{1/2}|x-y|)^{k\epsilon_2}|x-y|^{d-2}}\int_{\frac32|x-z|\leq |v-x|<\frac12|x-y|}\frac{|x-z|^{1-\delta}}{|x-v|^d}d\mu(v)\\
		&\lesssim  \frac{C_{k,d} |x-z|^\delta}{(1+ \lambda^{1/2}|x-y|)^{k\epsilon_2}|x-y|^{d-2}}\int_{B(x,|x-y|)}\frac{d\mu(v)}{|x-v|^{d-(1-\delta)}}.
		\end{align*}
	Decomposing the last integral into annulli, recalling that we are in the case $|x-y|<\rho(x)$ and using \eqref{eq: prop1 mu}, it becomes
			\begin{align*}
			\int_{B(x,|x-y|)}\frac{d\mu(v)}{|x-v|^{d-(1-\delta)}}
			&
			\leq 
			\sum_{j=0}^{\infty} \int_{|v-x|<2^{-j}|x-y|}\left(2^{-j-1}|x-y|\right)^{-d+(1-\delta)} d\mu(v)\\
			& 
			\lesssim
			\sum_{j=0}^{\infty} \frac{\mu(B(x,2^{-j}|x-y|))}{\left(2^{-j}|x-y|\right)^{d-(1-\delta)}}\\
			& 
			\lesssim
			\sum_{j=0}^{\infty} \frac{ \mu(B(x, \rho(x)))}{\left(2^{-j}|x-y|\right)^{d-(1-\delta)}}  \left(\frac{2^{-j}|x-y|}{\rho(x)}\right)^{d-2+ \delta_\mu}\\
			& 
			\lesssim  \frac{1}{|x-y|^{1+\delta}}\left(\frac{|x-y|}{\rho(x)}\right)^{\delta_\mu} \sum_{j=0}^{\infty} 2^{-j(\delta_\mu-(1+\delta))} \\
			&
			\lesssim 
			\frac{1}{|x-y|^{1+\delta}}\left(\frac{|x-y|}{\rho(x)}\right)^{\delta_\mu},
		\end{align*}
		provided that $\delta<\delta_\mu-1$. Therefore, 
		\begin{equation}\label{eq: estimate of I_2}
		I_2 \lesssim  \frac{C_{k,d} |x-z|^\delta}{(1+ \lambda^{1/2}|x-y|)^{k\epsilon_1}|x-y|^{d-1+\delta}}\left(\frac{|x-y|}{\rho(x)}\right)^{\delta_\mu}, \quad 0<\delta<\delta_\mu-1.
		\end{equation}
	
		For $v\in E_3$, $\frac12|x-y|\leq |v-x|\leq 2|x-y|$ which also yields $|v-y|<3|x-y|$. 
		Then, using again \eqref{eq: difference gradients}, \eqref{eq: bound Gamma_mu+lambda everywhere} and \eqref{eq: prop1 integral mu}, we have 
		\begin{align*}
			I_3&
			\lesssim \frac{C_{k,d} |x-z|}{(1+ \lambda^{1/2}|x-y|)^{k\epsilon_2}|x-y|^{d}} \frac{\mu(B(x,4|x-y|))}{|x-y|^{d-2}}. 
		\end{align*}		 
		Applying \eqref{eq: prop1 mu},  \eqref{eq: prop2 mu} and \eqref{eq: equiv 1}, we get the estimate \eqref{eq: bounds Ij} for $I_3$ with $\delta=1$ and, therefore, 
		\begin{equation}\label{eq: estimate of I_3}
			I_3\lesssim \frac{C_{k,d} |x-z|^{\delta}}{(1+ \lambda^{1/2}|x-y|)^{k\epsilon_2}|x-y|^{d+\delta}} \left(\frac{|x-y|}{\rho(x)}\right)^{\delta_\mu}, \quad 0<\delta\leq 1.
		\end{equation} 

		 Finally, to deal  with $E_4$ we use again \eqref{eq: difference gradients} and \eqref{eq: bound Gamma_mu+lambda everywhere}, and the fact that $\rho(x)\sim \rho(y)$ by \eqref{eq: equiv critical radius} (we are considering the case $|x-y|<\rho(x)$). We also note that for $v \in E_4$, $|v-x|\sim |v-y|$. 
		 
		 Since we are integrating in the global part, we will take advantage of the exponential decay by considering the lower bound for the Agmon distance $d_\mu$ given in \cite[Lemma~5.1]{DLT25}, i.e., for some $D_1,D_0>1$, there exists $C>0$ depending on them and $k_0$ such that 
		 \begin{equation*}
		 	 d_\mu(v,u)\geq D_1^{-1}\left(1+\frac{|v-u|}{\rho(u)}\right)^{\frac{1}{k_0+1}}+D_0^{-1}\left(1+\frac{\rho(u)}{|v-u|}\right)^{-1}-C, \quad v,u\in \RR^d.
		 \end{equation*}
		 This yields, for every $k>0$, there exists $C_k$ such that 
		 \begin{equation*}
		 	e^{-\epsilon_2d_\mu(v,u)}
\leq C_k \left(1+\frac{|v-u|}{\rho(u)}\right)^{\frac{-k\epsilon_2}{k_0+1}}.
		 \end{equation*}
		 Therefore, from all of the above considerations,
		 \begin{equation}\label{eq: integral over E_4}
		 	I_4\lesssim \frac{C_{k} |x-z| }{(1+ \lambda^{1/2}|x-y|)^{k\epsilon_2}}\int_{|v-x|\geq 2|x-y|}  \frac{\left(1+ \frac{|v-x|}{\rho(x)}\right)^{\frac{-k\epsilon_2}{k_0+1}}}{ |v-x|^{2d-2}}d\mu(v).
		 \end{equation}
		 We shall split the integral above inside the critical ball $B(x,\rho(x))$ and outside of it.		 
		 For the first part, let $j_0\in \NN$ such that $2^{j_0-1}|x-y|\leq \rho(x)< 2^{j_0}|x-y|$. Then, by splitting into dyadic annuli 
		 	\begin{align}\label{eq: E_4 intersec B_rho}
		 \int_{2|x-y|\leq |v-x|<\rho(x)}\frac{\left(1+ \frac{|v-x|}{\rho(x)}\right)^{\frac{-k\epsilon_2}{k_0+1}}}{ |v-x|^{2d-2}}d\mu(v)
		 &
		 \lesssim  \frac{1}{|x-y|^d}\left(\frac{|x-y|}{\rho(x)}\right)^{\delta_\mu}\;\sum_{j=1}^{j_0-1}2^{j(\delta_\mu-d)}\lesssim \frac{1}{|x-y|^d}\left(\frac{|x-y|}{\rho(x)}\right)^{\delta_\mu},
		 \end{align}
		 where we have used \eqref{eq: prop1 mu}, and 
		 that $\delta_\mu-d <0$.
		 
		When $|v-x|>\rho(x)$, by again splitting into dyadic annuli and choosing  $k>\max\left\{\frac{(k_0+1)(\log_2(D_\mu)-d)}{\epsilon_2},0\right\}$, we obtain
		\begin{align*}
		 	\int_{|v-x|>\rho(x)}\frac{\left(1+ \frac{|v-x|}{\rho(x)}\right)^{-\frac{k\epsilon_2}{k_0+1}}}{|v-x|^{2d-2}}d\mu(v)
		 	&\lesssim \rho(x)^{2-2d} \sum_{j \in \NN} 2^{-j\left(2d-2+\frac{k\epsilon_2}{k_0+1}\right)}\left(D_\mu^{j+1}2^{(d-2)(j+1)} \rho(x)^{d-2}\right)\\
		 	&\lesssim \rho(x)^{-d} \sum_{j \in \NN} \left(D_\mu 2^{-d-\frac{k\epsilon_2}{k_0+1}}\right)^j\lesssim \rho(x)^{-d} \lesssim \frac{1}{|x-y|^d} \left(\frac{|x-y|}{\rho(x)}\right)^{\delta_\mu},
		 \end{align*}
		 where we have used \eqref{eq: prop2 mu}, \eqref{eq: equiv 1}, and the fact that $\delta_\mu-d<0$ and $|x-y|<\rho(x)$. 
		
		Thus, from \eqref{eq: integral over E_4} and \eqref{eq: E_4 intersec B_rho} we obtain 
		 	\begin{equation}\label{eq: estimate of I_4}
		 I_4 
		 \leq
		 \frac{C_k |x-z|^{\delta} }{(1+ \lambda^{1/2}|x-y|)^{k\epsilon_1}}\frac{1}{|x-y|^{d+\delta}}\left(\frac{|x-y|}{\rho(x)}\right)^{\delta_\mu}, \quad 0<\delta\leq 1. 
		 \end{equation}
		 
		 Therefore, in \eqref{eq: estimate of I_1}, \eqref{eq: estimate of I_2}, \eqref{eq: estimate of I_3} and \eqref{eq: estimate of I_4} we have proved \eqref{eq: bounds Ij} for each $j=1,2,3,4$.
		 	\end{proof}

		\begin{proof}[Proof of Theorem~\ref{thm: BMO Riesz transform}]

			 By  Corollary~\ref{cor: T1 criterion infty} and Remark~\ref{obs: reduce to beta=0}, we need to verify  that for  $1<\delta_\mu<d$,  $x_0 \in \RR^d$ and $0<r\leq \frac{1}{2}\rho(x_0)$, 	 there exists a constant $C$, independent of  $B=B(x_0,r)$, such that 
			\begin{equation*}
				\frac{1}{|B|}\int_B \left|\mathcal R_\mu 1(y)-(\mathcal R_\mu 1)_B\right|dy \leq C \left(\frac{r}{\rho(x_0)}\right)^{\beta}, \quad 0<\beta<\min\{1,\delta_\mu-1\}.
			\end{equation*}
		
			The proof follows the general strategy of \cite[Proposition~4.12]{MSTZ14}, with several modifications. In particular, we provide a self-contained argument, as some of the estimates used there do not apply directly in our setting.
			
	Let  $y,z \in B$. By \eqref{eq: equiv critical radius} we have  $\rho(y)\sim \rho(x_0)\sim \rho(z)$. 	
	Since 
			\[\mathcal{R}_\mu 1(x)=  \lim_{\epsilon \rightarrow 0^+} \int_{|x-y|>\epsilon}K_\mu(x,y)dy \qquad a.e. \quad x \in \RR^d\]
			we have 
			\begin{align*}
			|\mathcal{R}_\mu 1(y)-\mathcal{R}_\mu 1(z)| 
			&
			\leq    \lim_{\epsilon \rightarrow 0^+} \left|\int_{\epsilon <|x-y|<4\rho(x_0)}K_\mu(y,x)dx - \int_{\epsilon <|z-x|<4\rho(x_0)}K_\mu(z,x)dx \right| \\
			&
			+\left| \int_{|x-y|>4\rho(x_0)}K_\mu(y,x)dx- \int_{|z-x|>4\rho(x_0)}K_\mu(z,x)dx\right|\\
			&
			=: \lim_{\epsilon \rightarrow 0^+} A_\epsilon +\tilde{A}.
			\end{align*}
			
			To analyze  $A_\epsilon$,  
			we consider  $0<\epsilon < 4\rho(x_0)-2r$. Since   $\int_{r_1<|x-y|<r_2} K_0(x,y)dy=0$ for every $0<r_1<r_2$, we can write
				\begin{align*}
			A_\epsilon 
			&
			= \left|\int_{\epsilon <|x-y|\leq 4\rho(x_0)}(K_\mu(y,x)-K_0(y,x))dx - \int_{\epsilon <|z-x|<4\rho(x_0)}(K_\mu(z,x)-K_0(z,x))dx \right|\\
			&
			\leq \int_{\RR^d}\left| (K_\mu(y,x)-K_0(y,x))(\chi_{\epsilon< |x-y|\leq 4\rho(x_0)}(x)-\chi_{\epsilon< |x-z|\leq 4\rho(x_0)}(x))\right|dx\\
			& + \int_{\RR^d}\left| [(K_\mu(y,x)-K_0(y,x))-(K_\mu(z,x)-K_0(z,x))]\chi_{\epsilon< |x-z|\leq 4\rho(x_0)}(x)\right|dx\\
			&
			:= 
			A_{\epsilon,1}+A_{\epsilon,2}.
			\end{align*}
			The term $A_{\epsilon,1}$ is nonzero 
			only in the following 4 cases:
				\begin{enumerate}
				\item $\epsilon < |x-y| \leq 4\rho(x_0)$ and $|x-z|\leq \epsilon$; 
				\item $\epsilon < |x-y| \leq 4\rho(x_0)$ and $|x-z|> 4\rho(x_0)$; 
				\item  $\epsilon < |x-z|\leq 4\rho(x_0)$ and $|x-y|\leq \epsilon$; 
				\item $\epsilon < |x-z|\leq 4\rho(x_0)$ and $|x-y|> 4\rho(x_0)$.
			\end{enumerate}
			
		In the first case we have $ \epsilon <|x-y|\leq |x-z|+|z-y|< \epsilon +2r$. Then, using the estimate \eqref{eq: comparison Riesz kernels}, integrating and using the Mean Value Theorem we get
			\begin{align}\label{eq: Aeps1 case 1}
		A_{\epsilon,1} 
		&
		\lesssim \frac{1}{{\rho(x_0)}^{\delta_\mu}} \int_{\epsilon}^{\epsilon +2r} t ^{\delta_\mu-1}dt
		\sim \frac{1}{{\rho(x_0)}^{\delta_\mu}}  \frac{1}{\delta_\mu} \left((\epsilon+2r)^{\delta_\mu}-\epsilon^{\delta_\mu}\right)
				\lesssim \frac{(4\rho(x_0))^{\delta_\mu-1}}{{\rho(x_0)}^{\delta_\mu}}2r
\sim\frac{r}{\rho(x_0)}. 
		\end{align}
	
		In the second case, observe that $4\rho(x_0)<|x-z|<|x-y|+2r$. Thus, $4\rho(x_0)-2r<|x-y|\leq 4\rho(x_0)$. Then, proceeding as in the previous case,
		\begin{equation*}
		A_{\epsilon,1} 
		\lesssim \int_{4\rho(x_0)-2r<|x-y|<4\rho(x_0)} \frac{1}{|x-y|^d}\left(\frac{|x-y|}{\rho(x)}\right)^{\delta_\mu}dx
		\lesssim \frac{r}{\rho(x_0)}.
		\end{equation*}
		
		With respect to the third and fourth cases, although the conditions seem to be symmetric, the kernels are still evaluated in $y$ instead of $z$.

	In the third case, without loss of generality, as we will take $\epsilon\to 0^+$, we may assume $\epsilon<4r$. Thus, by \eqref{eq: comparison Riesz kernels}, we deduce
		\begin{equation*}
			A_{\epsilon,1}\lesssim \frac{1}{\rho(x_0)^{\delta_\mu}} \int_{|x-y|\leq \epsilon}|x-y|^{\delta_\mu-d} dx\lesssim  \frac{1}{\rho(x_0)^{\delta_\mu}} \epsilon^{\delta_\mu}\lesssim \left(\frac{r}{\rho(x_0)}\right)^{\delta_\mu}\lesssim \frac{r}{\rho(x_0)},
		\end{equation*}
		where in the last inequality we have used that $r<\rho(x_0)$ and $\delta_\mu>1$.
		
		To estimate in the fourth case, note that $4\rho(x_0)<|x-y|<4\rho(x_0)+2r$, which gives, proceeding similarly to \eqref{eq: Aeps1 case 1}, that
			\begin{equation*}
			A_{\epsilon,1}\lesssim \frac{1}{{\rho(x_0)}^{\delta_\mu}}  \left((4\rho(x_0)+2r)^{\delta_\mu}-(4\rho(x_0))^{\delta_\mu}\right)\lesssim \frac{1}{{\rho(x_0)}^{\delta_\mu}} (5\rho(x_0))^{\delta_\mu-1} 2r \lesssim  \frac{r}{\rho(x_0)}.
		\end{equation*}
		
		In conclusion, since $r<\rho(x_0)$, 
		\begin{equation}\label{eq: Aeps1}
				A_{\epsilon,1}\leq C\left(\frac{r}{\rho(x_0)}\right)^{\delta}, \quad 0<\delta\leq 1.
			\end{equation}
	
		Now we consider $	A_{\epsilon,2} $, and we split it as follows, 
		\begin{align*}
		A_{\epsilon,2}&=\int_{\RR^d}\left| [(K_\mu(y,x)-K_0(y,x))-(K_\mu(z,x)-K_0(z,x))]\chi_{\epsilon< |x-z|\leq 4\rho(x_0)}(x)\right|dx\\ 
		&
		\leq \int_{|z-x|\geq 2|z-y|} \left| [(K_\mu(z,x)-K_0(z,x))-(K_\mu(y,x)-K_0(y,x))]\chi_{\epsilon< |x-z|\leq 4\rho(x_0)}(x)\right|dx\\
		&
		+ \int_{|z-x|< 2|z-y|} \left|[(K_\mu(z,x)-K_0(z,x))-(K_\mu(y,x)-K_0(y,x))]\right|dx\\
		&
		= A_{\epsilon,2,1}+ A_{\epsilon,2,2}.
		\end{align*}
		By \eqref{itm: diffs of Riesz kernels},  for any $0<\delta<\min\{1, \delta_\mu-1\}$,
		\begin{equation}\label{eq: A 2,1 epsilon}
		A_{\epsilon,2,1} \lesssim \frac{|y-z|^\delta}{\rho(x_0)^{\delta_\mu}}\int_{|x-z|<4\rho(x_0)}|x-z|^{\delta_\mu-d-\delta}dx \lesssim \frac{r^\delta}{\rho(x_0)^{\delta_\mu}} (4\rho(x_0))^{\delta_\mu-\delta}
		\lesssim \left(\frac{r}{\rho(x_0)}\right)^\delta.
		\end{equation}
		
		On the other hand, by \eqref{eq: comparison Riesz kernels} we get 
		\begin{align}\label{eq: A 2,2 epsilon}
		A_{\epsilon,2,2}
		& 
		\lesssim \int_{|x-z|< 2 |y-z|}  \frac{1}{|x-z|^d}\left(\frac{|x-z|}{\rho(z)}\right)^{\delta_\mu}dx+
		\int_{|x-y|< 3|y-z|} \frac{1}{|x-y|^d}\left(\frac{|x-y|}{\rho(y)}\right)^{\delta_\mu}dx\nonumber\\
		&  
		\lesssim
		\frac{1}{\rho(x_0)^{\delta_\mu}} |y-z|^{\delta_\mu}
		\lesssim 
			 \left(\frac{r}{\rho(x_0)}\right)^\delta \qquad ,
		\end{align}
		for any $\delta\leq  \delta_\mu$. Hence, combining the estimates obtained in \eqref{eq: Aeps1}, \eqref{eq: A 2,1 epsilon} and \eqref{eq: A 2,2 epsilon} we have that for all $\epsilon>0$ sufficiently small, and any $0<\delta<\min\{1, \delta_\mu-1\}$,
		 \begin{equation}\label{eq: final estimate Aeps}
		A_\epsilon\leq C \left(\frac{r}{\rho(x_0)}\right)^\delta.
		\end{equation}
		
		Let us now estimate $\tilde{A}$, by writing
		\begin{align*}
		\tilde{A}
		&
		= \left| \int_{|x-y|>4\rho(x_0)}K_\mu(y,x)dx- \int_{|z-x|>4\rho(x_0)}K_\mu(z,x)dx\right|   \nonumber \\
		&
		\leq
		\int_{|x-y|>4\rho(x_0)}|K_\mu(y,x)-K_\mu(z,x)|dx + \int_{\RR^d}|K_\mu(z,x)||\chi_{|x-z|>4\rho(x_0)}(x)-\chi_{|x-y|>4\rho(x_0)} (x)| dx\nonumber\\
		& 
		=: \tilde{A}_1+ \tilde{A}_2.
		\end{align*}
		In the integral of $\tilde{A}_1$, we have $|x-y|>4\rho(x_0)\geq 8r > 2|y-z|$. Therefore, from the smoothness of the Riesz kernel $K_\mu$ given in \eqref{eq: smoothness Riesz kernel delta_mu>1}, and recalling that $y,z\in B$ and $\delta_\mu>1$, we have 
		\begin{align*}
		\tilde{A}_1\leq C |y-z|^{\delta_\mu-1} \int_{|x-y|>4\rho(x_0)}\frac{1}{|x-y|^{d+\delta_\mu-1}}dx \leq C \left( \frac{r}{\rho(x_0)}\right)^{\delta_\mu-1}.
		\end{align*} 
		
		For $\tilde{A}_2$ we use the size condition \eqref{eq: size Riesz kernel delta_mu>1}
			 to have
			\[\tilde{A}_2\lesssim \int_{\RR^d} \frac{1}{|z-x|^d}\left|\chi_{|x-z|>4\rho(x_0)}(x)-\chi_{|x-y|>4\rho(x_0)} (x)\right| dx.\]
		When $|x-z|>4\rho(x_0)\geq |x-y|$, we get $4\rho(x_0)<|z-x|<4\rho(x_0)+2r$. Thus, 
		\begin{equation*}
			\tilde{A}_2\lesssim \int_{4\rho(x_0)<|z-x|<4\rho(x_0)+2r} \frac{1}{|z-x|^{d}} dx\sim \int_{4\rho(x_0)}^{4\rho(x_0)+2r} \frac1t dt\lesssim  \frac{2r}{4\rho(x_0)}\sim  \frac{r}{\rho(x_0)}.
		\end{equation*}
		
		In the other case, when $|x-y|>4\rho(x_0)\geq |x-z|$, we get $2\rho(x_0)\leq 4\rho(x_0)-2r<|z-x|\leq 4\rho(x_0)$. Proceeding as above, we get $\tilde{A}_2\leq C \frac{r}{\rho(x_0)}.$
		Hence 
		\begin{equation}\label{eq: final estimate Atilde}
		\tilde{A} \leq C \left(\frac{r}{\rho(x_0)}\right)^\delta, \quad 0<\delta<\min\{1,\delta_\mu-1\}.
		\end{equation}	
	
		Then, the estimates \eqref{eq: final estimate Aeps} and \eqref{eq: final estimate Atilde} imply \eqref{eq: T1 power condition} as we wanted to prove. 
	\end{proof}

\subsection{Laplace transform-type multipliers}

Given $\phi \in L^\infty(0, \infty)$, we define 
\[m_\phi(\lambda)=\lambda \int_0^\infty e^{-\lambda t}\phi(t)\, dt, \quad \lambda>0.\]
By means of the Spectral Theorem, we can define the Laplace transform-type multipliers $m_\phi(\mathcal{L}_\mu)$. If we consider the heat semigroup associated with $\mathcal{L}_\mu$,
\begin{equation}\label{eq: heat semigroup}\mathcal W_tf(x):=e^{-t\mathcal{L}_\mu}f(x)=\int_{\mathbb{R}^d} \mathcal W_t(x,y) f(y)dy, \quad f \in L^2(\mathbb{R}^d),\ x \in \mathbb{R}^d,\ t>0,
\end{equation}we can write
\[m_\phi(\mathcal{L}_\mu)f(x)= \int_{0}^{\infty} \phi(t) \mathcal{L}_\mu e^{-t\mathcal{L}_\mu}f(x)dt=-\int_{0}^{\infty}\phi(t) \partial_t \mathcal W_tf(x) dt, \qquad x \in \mathbb{R}^d.\]

Its kernel can be expressed as
\[\mathcal{M}_\phi(x,y)= -\int_{0}^{\infty} \phi(t) \partial_t \mathcal{W}_t(x,y) dt, \quad x,y\in \mathbb{R}^d.\]

From classical theory, it can be deduced that $m_\phi(\mathcal{L}_\mu)$ is bounded on $L^2(\mathbb{R}^d)$ for every $\phi \in L^\infty(0, \infty)$ (see \cite[Corollary~3,~p.~121]{SteinTopics}). In the case of the Schr\"odinger operator $\mathcal{L}_V$, with $V \in RH_q$, $q > d/2$, it is known that $m_\phi(\mathcal{L}_V)$ can be extended to a bounded operator on $L^p(\mathbb{R}^d)$ for all $1 < p < \infty$, and also from $L^1(\mathbb{R}^d)$ to $L^{1,\infty}(\mathbb{R}^d)$ (see \cite[Theorem~2]{BCFR}).

 In the particular case of $\phi(\lambda)=\lambda^{i\gamma}$, for $\lambda>0$ and $\gamma\in \mathbb R$, we get the imaginary powers $\mathcal{L}_V^{i\gamma}$.  
 In \cite{BCH13Lerner} it was proved that $\mathcal{L}_V^{i\gamma}$ is bounded on $L^p(w)$ for all $\gamma > 0$, $1 < p < \infty$, and $w \in A_p^\rho$.
 Recently, it was proved in \cite{DLT25} that the operator $m_\phi(\mathcal{L}_\mu)$ is an exponential SCZO of $(0, \infty, \delta)$ type, for $0 < \delta < \min\{1, \delta_\mu\}$, with parameters $c=c_0$ (the constant from Lemma~\ref{lem: heat semigroup estimates}\ref{itm: size W_t} given below) and $m=\frac{1}{k_0+1}$. Consequently, it is bounded  on $L^p(w)$ for all $1 < p < \infty$ and every $w \in H^{\rho,m^\ast}_{p,c^\ast}$, where $m^\ast = \frac{1}{(k_0+1)^2}$ and $c^\ast < c_0(2^{3k_0+7}C_0^{k_0+3})^{-m^*}$ (see \cite[Theorem~5.9]{DLT25}).
 The proofs of these results rely on estimates for the kernel $\mathcal{M}_\phi(x,y)$, which are derived from estimates for $\mathcal{W}_t(x,y)$ and $\partial_t\mathcal{W}_t(x,y)$. These estimates for the heat kernels are provided here, as they will be employed in subsequent examples.

 \begin{lem}[{\cite[Theorem~1.1, Lemma~3.7]{WuYan16}}]\label{lem: heat semigroup estimates} Let $x,y\in \mathbb{R}^d$ and $t>0$.
 	\begin{enumerate}\item \label{itm: size W_t} There exist positive constants $C, c$, and $c_0$ such that
 		\[0 \leq \mathcal{W}_t(x,y) \leq \frac{C}{t^{\frac d2}} \exp\left(-\frac{ |x-y|^2}{4t}\right)\exp\left(-c_0\left(1+\frac{\max\{|x-y|,\sqrt{t}\}}{\rho(x)}\right)^{\frac{1}{k_0+1}}\right).\]
 		\item \label{itm: smoothnes W_t} For every $0 <{\beta}<\min\left\{1, \delta_\mu\right\}$ and every $N>0$, there exist positive constants $c$ and $C_N$ such that, for all $|h|\leq \sqrt{t}$,
 		\[ | \mathcal{W}_t(x+h,y)-\mathcal{W}_t(x,y)| \leq  \frac{C_N}{t^{\frac d2}} \left(\frac{|h|}{\sqrt{t}}\right)^{{\beta}} \exp\left(-c \frac{|x-y|^2}{t}\right)\left(1+ \frac{\sqrt{t}}{\rho(x)}+ \frac{\sqrt{t}}{\rho(y)}\right)^{-N}.\]
 	\end{enumerate}\end{lem}

 \begin{lem}[{\cite[Lemma~3.8]{WuYan16}}] \label{lem: Q_t estimates} Let $x,y\in \mathbb{R}^d$ and $t>0$.
 	\begin{enumerate}
 		\item \label{item: t deriv parcial k_t} There exist positive constants $C$ and $c$ such that
 		\[\left|t\partial_t\mathcal{W}_t(x,y)\right| \leq \frac{C}{t^{\frac d2}} \exp\left(-\frac{|x-y|^2}{4t}\right)\exp\left(-c_0\left(1+\frac{\max\{|x-y|,\sqrt{t}/2\}}{\rho(x)}\right)^{\frac{1}{k_0+1}}\right),\]
 		where $c_0$ is the constant from Lemma~\ref{lem: heat semigroup estimates}\ref{itm: size W_t}.
 		\item \label{itm: t diff partial W_t} For every $0< {\beta}< \min\{1,\delta_\mu\}$ and $N>0$, there exist positive constants $c$ and $C_N$ such that, for all $|h|\leq \sqrt{t}$,
 		\[|t\partial_t \mathcal{W}_t(x+h,y)-t\partial_t \mathcal{W}_t(x,y)| \leq  \frac{C_N}{t^{\frac d2}} \left(\frac{|h|}{\sqrt{t}}\right)^{{\beta}} \exp\left(-c\frac{|x-y|^2}{t}\right)\left(1+ \frac{\sqrt{t}}{\rho(x)}+ \frac{\sqrt{t}}{\rho(y)}\right)^{-N}.\]
 		\item \label{itm: integral Q_t} For every $N>0$, there exists a constant $C_N$ such that
 		\[\left|\int_{\mathbb{R}^d} t\partial_t \mathcal{W}_t(x,y) dy\right|\leq C_N \left(\frac{\sqrt{t}}{\rho(x)}\right)^{\delta_\mu}\left(1+\frac{\sqrt{t}}{\rho(x)}\right)^{-N}.\]
 \end{enumerate}\end{lem}

 To establish boundedness results on $\BMO_\rho^\alpha(w)$ for $m_\phi(\mathcal{L}_\mu)$, extending those already known for $m_\phi(\mathcal{L}_V)$ given in \cite[Theorem~1.3]{MSTZ14},  
	 we will give some estimates comparing the kernel $\mathcal{W}_t(x,y)$ with the classical heat semigroup kernel $\{W_t\}_{t>0} = \{e^{-t\Delta}\}_{t>0}$.
	 
	 We note first that we can express, in the weak sense, that
	 \begin{equation}\label{eq: integral W_t dmu }\int_{\mathbb{R}^d}\partial_t \mathcal{W}_t(x,y)dy= -e^{-t\mathcal{L}_\mu}\mathcal{L}_\mu 1(x)=-\int_{\mathbb{R}^d} \mathcal{W}_t(x,y)\mathcal{L}_\mu1(y) dy=-\int_{\mathbb{R}^d}\mathcal{W}_t(x,y)d\mu(y).
	 \end{equation}
	 We also have the following inequality given in \cite[Eq. (2.2)]{WuYan16}: for every $x \in \mathbb{R}^d$ and $t > 0$,
	 \begin{equation}\label{eq: integral exp heat bound dmu}
	 	\int_{\mathbb{R}^d}t^{-d/2}\exp \left(-\frac{|x-y|^2}{t}\right)d\mu(y) \leq\begin{cases*}\frac Ct \left(\frac{\sqrt{t}}{\rho(x)}\right)^{{\delta_\mu}} , & $t < \rho(x)^2$\\
	 		\frac Ct \left(\frac{\sqrt{t}}{\rho(x)}\right)^{\log_2(D_\mu)}, & $t \geq \rho(x)^2.
	 		$			\end{cases*}\end{equation}

\begin{prop}[{\cite[Lemma~3.6]{WuYan16}}]\label{prop: comparison heat kernels} Let $x,y, z\in \mathbb{R}^d$ and $t>0$.\begin{enumerate}
		\item \label{itm: diff heat kernels} There exist positive constants $C$ and $c$ such that
		\begin{equation*}|\mathcal{W}_t(x,y)-W_t(x,y)| \leq 
			\left\{
			\begin{array}{ll}
				\frac{C}{t^{d/2}}\left( \frac{\sqrt{t}}{\rho(x)} \right)^{\delta_\mu} \exp\left(-c\frac{|x-y|^2}{t}\right), & \quad \text{if } \sqrt{t} < \rho(x), \\
				\frac{C}{t^{d/2}}\left( \frac{\sqrt{t}}{\rho(y)} \right)^{\delta_\mu} \exp\left(-c\frac{|x-y|^2}{t}\right), & \quad \text{if } \sqrt{t} < \rho(y), \\
				\frac{C}{t^{d/2}}\exp\left(-\frac{|x-y|^2}{4t}\right), & \quad \text{otherwise}.\end{array}\right.\end{equation*}
		
		\item \label{itm: diff of diff heat kernels} For every $0<{\beta}<\min\left\{ 1, \delta_\mu\right\}$ and every positive constant $C$, there exist positive constants $C'$ and $c$ such that
		\begin{equation*}
			|(\mathcal{W}_t(x,y)-W_t(x,y))-(\mathcal{W}_t(z,y)-W_t(z,y))|\leq \frac{C'}{t^{\frac d2}} \left(\frac{|x-z|}{\rho(y)}\right)^{{\beta}}\exp\left(-c\frac{|y-z|^2}{t}\right),
			\end{equation*}
		whenever $|x-z|\leq |x-y|/4$ and $|x-z|\leq C \rho(x)$.
	\end{enumerate}
\end{prop}

\begin{thm}\label{thm: BMO Laplace multipliers}Let $m=\frac{1}{k_0+1}$ and $\delta=\min\{1, \delta_\mu\}$. Then, for all $0 \leq \alpha < \delta$, the operator $m_\phi(\mathcal{L})$ is bounded on $\BMO_\rho^{\alpha}(w)$ provided that $w \in E_{\infty,c_1,c_2}^{\rho,m} \cap D_{\kappa,c_3}^{\rho,m}$ with $c_1, c_2, c_3 \geq 0$ such that $(c_2+c_3) < c_0 \left(1 - \frac{d(\kappa-1)+\alpha}{\delta}\right)(4C_0)^{-m}$ and $1 \leq \kappa < \frac{\delta - \alpha}{d} + 1$.
\end{thm}

\begin{proof}
	By Corollary~\ref{cor: T1 criterion infty} and Remark~\ref{obs: reduce to beta=0} it suffices to show that for every ball $B=B(x_0,r)$ with $x_0 \in \mathbb{R}^d$ and $0 < r \leq \frac{1}{2} \rho(x_0)$, and for every $0 < \beta < \min\{1, \delta_\mu\}$, the following holds:
	\begin{equation}\label{eq: T1 cond Laplace multipliers}
		\frac{1}{|B|}\int_{B} |m_\phi(\mathcal{L}_\mu)1(y)-(m_\phi(\mathcal{L}_\mu)1)_B|dy \leq C\left(\frac{r}{\rho(x_0)}\right)^{\beta}.
	\end{equation}
	
	Consider $y,z \in B$. As $\phi\in L^\infty(0, \infty),$ 
		\begin{align*}
		|m_\phi(\mathcal{L}_\mu)1(y)-(m_\phi(\mathcal{L}_\mu)1)(z)|
		&
		\leq C\int_{0}^{\infty}  \left|\int_{\mathbb{R}^d} (\partial_t\mathcal{W}_t(y,x)-\partial_t\mathcal{W}_t(z,x))dx \right|dt.
	\end{align*}
	
	We split the integral over $t$ into three regions. From $0$ to $4r^2$, we control by the sum and we analyze only the first term, since the other is handled in an analogous manner, keeping in mind that since $y, z \in B$, then $\rho(y) \sim \rho(x_0) \sim \rho(z)$. Applying Lemma~\ref{lem: Q_t estimates}\ref{itm: integral Q_t},
	\begin{align}\label{eq: 1st integral multipliers}
		\int_{0}^{4r^2}\left| \int_{\mathbb{R}^d} \partial_t\mathcal{W}_t(y,x) dx\right| dt 
		& 
		\lesssim 
		\int_{0}^{4r^2}  \left(\frac{\sqrt{t}}{\rho(y)}\right)^{\delta_\mu} \frac{dt}{t}
		\lesssim
		\frac{1}{\rho(x_0)^{\delta_\mu}}
		\int_{0}^{4r^2} t^{\delta_\mu/2-1} dt 
		\lesssim \left(\frac{r}{\rho(x_0)}\right)^{\delta_\mu}.
	\end{align}
	
	Now, let us examine the integral over $t$ from $\rho(x_0)^2$ to $\infty$. Here, we apply Lemma~\ref{lem: Q_t estimates}\ref{itm: t diff partial W_t} since $|y-z| \leq 2r \leq \rho(x_0) \leq \sqrt{t}$. It follows that, for every $0 < \beta < \min\{1, \delta_\mu\}$,
	\begin{align}\label{eq: 2nd integral multipliers}
		\int_{\rho(x_0)^2}^{\infty}  \int_{\mathbb{R}^d} \left|\partial_t\mathcal{W}_t(y,x)-\partial_t\mathcal{W}_t(z,x)\right|dx dt 
		&\lesssim 
		\int_{\rho(x_0)^2}^{\infty}  \int_{\mathbb{R}^d} \left(\frac{|z-y|}{\sqrt{t}}\right)^{\beta} t^{-d/2-1}\exp \left(-c\frac{|x-y|^2}{t}\right)dx dt \nonumber\\
		&\lesssim (2r)^{\beta} 
			\int_{\rho(x_0)^2}^{\infty} t^{-\beta/2-1} dt\lesssim \left(\frac{r}{\rho(x_0)}\right)^{\beta},
	\end{align}
since  $\int_{\mathbb{R}^d} e^{-c\frac{|x-y|^2}{t}} dx=(c\pi t)^{d/2}$.
	
	Finally, for $4r^2 \leq t \leq \rho(x_0)^2$, by \eqref{eq: integral W_t dmu },  Lemma~\ref{lem: heat semigroup estimates}\ref{itm: smoothnes W_t}, the fact  that $|y-z| \leq 2r \leq \sqrt{t}$ in the integration domain and  $t \leq \rho(x_0)^2 < C \rho(y)^2$, together with inequality \eqref{eq: integral exp heat bound dmu}, we have for every \linebreak $0 < \beta < \min\{1, \delta_\mu\}$
		\begin{align}\label{eq: 3rd integral multipliers}
		\int_{4r^2}^{\rho(x_0)^2} 	\left|\int_{\mathbb{R}^d} \partial_t \mathcal{W}_t(y,x) -\partial_t \mathcal{W}_t(z,x) dx \right| dt&\leq  \int_{4r^2}^{\rho(x_0)^2}  \int_{\mathbb{R}^d}\left|\mathcal{W}_t(y,x)- \mathcal{W}_t(z,x))\right| d\mu(x)dt \nonumber \\
		&
		\lesssim 
		\int_{4r^2}^{\rho(x_0)^2} \left(\frac{|z-y|}{\sqrt{t}}\right)^{\beta}\int_{\mathbb{R}^d} t^{-d/2}\exp \left(-c\frac{|y-x|^2}{t}\right) d\mu(x) dt \nonumber \\
		& \lesssim r^{\beta}\rho(x_0)^{-\delta_\mu}\int_{0}^{\rho(x_0)^2} t^{-1 + \frac{\delta_\mu - \beta}{2}} dt \nonumber\\
		& \lesssim r^{\beta}\rho(x_0)^{-\delta_\mu}\rho(x_0)^{\delta_\mu-\beta}\sim \left(\frac{r}{\rho(x_0)}\right)^{\beta},
	\end{align}
	for all $0 < \beta < \min\{1, \delta_\mu\}$.
	
	Therefore, by combining estimates \eqref{eq: 1st integral multipliers}, \eqref{eq: 2nd integral multipliers}, and \eqref{eq: 3rd integral multipliers}, we obtain \eqref{eq: T1 cond Laplace multipliers} as desired. The boundedness property of $m_\phi(\mathcal{L}_\mu)$ follows from Corollary~\ref{cor: T1 criterion infty}.
\end{proof}

	\subsection{Maximal operators for the heat-diffusion semigroup  associated with $e^{-t\mathcal{L}_\mu}$.}
	
	Let $\{ \mathcal{W}_t \}_{t>0}$ be the heat semigroup associated with $\mathcal{L}_\mu$, given for $f\in L^2(\RR^d)$ by
	\[
	\mathcal{W}_t f(x) = e^{-t\mathcal{L}_\mu}f(x) = \int_{\mathbb{R}^n} \mathcal{W}_t(x,y) f(y)\, dy, \quad x \in \mathbb{R}^d,\ t > 0,
	\]
	where $\mathcal{W}_t(x,y)$ is the corresponding heat kernel, for $x,y\in \RR^d$ and $t>0$.
	
		We consider the maximal operator of the heat semigroup
	\begin{equation*}
		\mathcal{W}^{*}f(x)= \sup_{t>0} \left|e^{-t\mathcal{L}_\mu}\right|=\sup_{t>0}|\mathcal{W}_tf(x)|, \quad f\in L^2(\RR^d).
	\end{equation*}
	
We interpret the operator $\mathcal{W}^{*}$ in a vector-valued sense, as in \cite{ST05}, and take into account Remark~\ref{obs: vector valued} in the Banach space $ \mathbb{B}= L^\infty((0, \infty), dt)$. Then,  we can write
$\mathcal{W}^{*}f(x)=\|\boldsymbol{W}f(x)\|_{\mathbb{B}}$, where we denote $\boldsymbol{W}f(x)=\{\mathcal W_t f(x)\}_{t>0}$, for every $x\in \RR^d$ and $f\in L^2(\RR^d)$.

	 $\boldsymbol{W}$ is bounded from $L^2(\RR^d)$ into $L^2_{\mathbb{B}}(\RR^d)$
	and, from the kernel conditions to be proven below, it follows that it is a vector-valued Calderón–Zygmund operator; consequently, it is bounded from $L^p(\mathbb{R}^d)$ into $L_{\mathbb{B}}^p(\mathbb{R}^d)$ for all $1 < p < \infty$. 
	Furthermore, we shall prove that conditions \eqref{eq: size pointwise fract} and \eqref{eq: smoothness pointwise fract} are satisfied in $\mathbb{B}$, with parameters $c=c_0$ and $m=\frac{1}{k_0+1}$. This implies that $\mathcal{W}^*$ is bounded on $L^p(w)$ for $w\in H^{\rho, m^*}_{p,c^*}$, where $m^*=\frac{1}{(k_0+1)^2}$ and $c^*<c_0(2^{3k_0+7}C_0^{k_0+3})^{-m^*}$ (see \cite[Theorem~4.1 and Proposition~4.2]{DLT25}). 
	\begin{prop}\label{prop: class heat maximal}
			Let $x_0,x,y\in \RR^d$ and $t>0$. Then,
		\begin{enumerate}
			\item \label{itm: size heat kernel Banach} for every $x\neq y$,
			\[\|\mathcal{W}_t(x,y)\|_{\mathbb B} \lesssim \frac{1}{|x-y|^{d}} 
			\exp\left(-c_0\left(1+\frac{|x-y|}{\rho(x)}\right)^{\frac{1}{k_0+1}}\right);\]
			\item \label{itm: smoothness heat kernel Banach} for every $0<\delta<\min\{1,\delta_\mu\},$ there exists $C_\delta$ such that
			\[
			\|\mathcal{W}_t(x,y) - \mathcal{W}_t(x_0,y)\|_{\mathbb B} \leq 
			C_\delta \frac{|x-x_0|^\delta}{|x-y|^{d+\delta}},\] 
			provided that $|x-y|>2|x-x_0|$.	
		\end{enumerate}
	\end{prop}
	\begin{proof}
	 By Lemma~\ref{lem: heat semigroup estimates}\ref{itm: size W_t}, we obtain  \ref{itm: size heat kernel Banach} when  $t>|x-y|^2$. For
	  $t\leq |x-y|^2$, observe  that  for each $N$, there exists $C_N$ such that
		\begin{equation}\label{eq: exp polynom}\exp\left(- c\frac{|x-y|^2}{t}\right)\leq C_N \left(\frac{t}{|x-y|^2}\right)^{N}. 
			\end{equation} Then, if $N=d/2$ is taken above, it follows that 
		\begin{align*}
			0 \leq \mathcal{W}_t(x,y) 
			&
			\lesssim \frac{1}{|x-y|^{d}}\exp\left(-c_0\left(1+\frac{|x-y|}{\rho(x)}\right)^{\frac{1}{k_0+1}}\right) .	 	
		\end{align*}
	
	For the smoothness condition \ref{itm: smoothness heat kernel Banach}, we consider   $|x-y|>2|x-x_0|$, that implies    $|x-y|\sim |x_0-y|. $
		
	For $0<\delta'<\min\left\{1, \delta_\mu\right\}$ and	 $|x-x_0|<\sqrt{t}$, by Lemma~\ref{lem: heat semigroup estimates}\ref{itm: diff heat kernels} and taking  $N=\frac{d}{2}+\frac{\delta'}{2}$ in \eqref{eq: exp polynom}
		\begin{align*}
			| \mathcal{W}_t(x,y)-\mathcal{W}_t(x_0,y)| 
			& \leq 
			C \left(\frac{|x-x_0|}{\sqrt{t}}\right)^{\delta'} t^{-d/2}\left(\frac{|x-y|^2}{4t}\right)^{-\frac{d}{2}-\frac{\delta'}{2}} \leq 
			C_{\delta'} \frac{1}{|x-y|^d}\left(\frac{|x-x_0|}{|x-y|}\right)^{\delta'}.
		\end{align*}
		For $|x-x_0|>\sqrt{t}$, we use Lemma~\ref{lem: heat semigroup estimates}\ref{itm: size W_t} and $N=\frac{d}{2}+\frac{\delta'}{2}$ in \eqref{eq: exp polynom} to obtain  
		\begin{align*}
			|\mathcal{W}_t(x,y)|
			& \leq C t^{-d/2} \exp\left(-\frac{|x-y|^2}{4t}\right)
			\leq C_{\delta'} \frac{1}{|x-y|^d}\left(\frac{|x-x_0|}{|x-y|}\right)^{\delta'},
		\end{align*}
	and the same estimate follows for $\mathcal{W}_t(x_0,y)$.
	This completes the proof.
	\end{proof}

	\begin{thm}\label{thm: BMO heat maximal}
		Let  $m=\frac{1}{k_0+1}$,  $c_0$ as in  Lemma~\ref{lem: heat semigroup estimates}\ref{itm: size W_t},   $\delta=\min\{1,\delta_\mu\}$, {with $\delta_\mu<d$} and $0 \leq \alpha < \delta$. Then $\mathcal{W}^*$ is bounded on $\BMO_\rho^\alpha(w)$, for  $w\in  {E_{\infty,c_1,c_2}^{\rho,m}}\cap D_{\kappa,c_3}^{\rho,m}$, where  $1 \leq \kappa < \frac{\delta - \alpha}{d}+1$ and $c_1,c_2,c_3\geq 0$ such that  $(c_2+c_3)<c_0\left(1-\frac{d(\kappa-1)+\alpha}{\delta}\right)(4C_0)^{-m}$. 
	\end{thm}
	\begin{rem}
				The condition $\delta_\mu < d$ in Theorem~\ref{thm: BMO heat maximal} differs from the case $d\mu(x)=V(x)dx$. In the latter, $\delta_V = 2 - d/q$ (where $V \in RH_q$), and $\delta_V < d$ is always satisfied.
			\end{rem}
	\begin{proof} By  Corollary~\ref{cor: T1 criterion infty} and Remarks~\ref{obs: reduce to beta=0} and \ref{obs: vector valued} it is enough to prove that there exists a constant  $C$ such that 
		\begin{equation*}
			\frac{1}{|B|} \int_{B}\| \mathcal{W}_t1(y)- (\mathcal{W}_t1)_B\|_{\mathbb{B}}\ dy \leq C\left(\frac{r}{\rho(x_0)}\right)^{{\beta}},
		\end{equation*}
	for  $B=B(x_0,r)$ with $0<r\leq \frac{1}{2}\rho(x_0)$ and every  $0<\beta<\min\left\{1, \delta_\mu\right\}$.
				To this end, since  
			\[\| \mathcal{W}_t1(y)- (\mathcal{W}_t1)_B\|_{\mathbb{B}}\leq \frac{1}{|B|}\int_{B}\| \mathcal{W}_t1(y)- \mathcal{W}_t1(z)\|_{\mathbb{B}}dz,\]
	we shall estimate the integrand $\| \mathcal{W}_t1(y)- \mathcal{W}_t1(z)\|_{\mathbb{B}}$.
		
	Let $t>0$ and let $y, z \in B$. 		
	Given that $\rho(y) \sim \rho(x_0) \sim \rho(z)$, if $\sqrt{t} < 2r$, then $\sqrt{t} < \rho(x_0) \lesssim \rho(z)$ and $\sqrt{t} \lesssim \rho(y)$. Since $W_t1(x)\equiv 1$, by applying Proposition~\ref{prop: comparison heat kernels}\ref{itm: diff heat kernels}, and changing variables, it follows that   
		\begin{align*}
			|\mathcal{W}_t1(y)-\mathcal{W}_t1(z)|
			&  
				\lesssim \frac{1}{t^{d/2}} 
				\left( \frac{\sqrt{t}}{\rho(y)} \right)^{\delta_\mu} \int_{\mathbb{R}^d} \exp\left(-c\frac{|y-x|^2}{t}\right)dx+ \left( \frac{\sqrt{t}}{\rho(z)} \right)^{\delta_\mu}  \int_{\mathbb{R}^d} \exp\left(-c\frac{|z-x|^2}{t}\right)dx\\
			&  \lesssim 
				\left( \frac{\sqrt{t}}{\rho(x_0)} \right)^{\delta_\mu}
				\lesssim \left( \frac{r}{\rho(x_0)} \right)^{\delta_\mu}.
		\end{align*}
		
		For $\sqrt{t}>\rho(x_0)\geq 2r$, it follows that  $|y-z|<2r<\sqrt{t}$. Therefore, by Lemma~\ref{lem: heat semigroup estimates}\ref{itm: diff heat kernels}, for every $0<{\beta}<\min\{1,\delta_\mu\}$, we have
		\begin{align}\label{eq: 2nd estimate diff W_t1}
			|\mathcal{W}_t1(y)-\mathcal{W}_t1(z)| &\leq \int_{\mathbb{R}^d} |\mathcal{W}_t(y,x)- \mathcal{W}_t(z,x)|dx\lesssim  \left(\frac{|y-z|}{\sqrt{t}}\right)^{{\beta}}t^{-d/2}\int_{\mathbb{R}^d} \exp\left(-c\frac{|y-x|^2}{t}\right)dx\nonumber\\
			& \lesssim \left(\frac{r}{\sqrt{t}}\right)^{\beta}   \lesssim \left(\frac{r}{\rho(x_0)}\right)^{{\beta}}.
		\end{align}
	Finally, if $2r < \sqrt{t} < \rho(x_0)$
	 we can write		
		\begin{align}\label{eq: 3rd estimate diff W_t1}
			|\mathcal{W}_t1(y)-\mathcal{W}_t1(z)|
			&
			\leq \int_{4|y-z|<|x-y|\leq {a}\rho(y)} \left|(\mathcal{W}_t(y,x)- W_t (y,x))-(\mathcal{W}_t(z,x)- W_t (z,x))\right| dx\nonumber\\
			& 
			\quad + \int_{|x-y|\leq 4|y-z|} \left|(\mathcal{W}_t(y,x)- W_t (y,x))-(\mathcal{W}_t(z,x)- W_t (z,x))\right|dx\nonumber\\
			& 
			\quad  + \int_{|x-y|>{a}\rho(y)} \left|(\mathcal{W}_t(y,x)- W_t (y,x))-(\mathcal{W}_t(z,x)- W_t (z,x))\right|dx,
		\end{align}
		where $a=2^{k_0+1}C_0$. 
		
	For the first integral, we apply Proposition~\ref{prop: comparison heat kernels}\ref{itm: diff of diff heat kernels} taking into account that $|x-z| \sim |x-y|$ and $\rho(x) \sim \rho(y)$  in the integration region to have
		\begin{align}\label{eq: 1st integral diffs W_t}
			&\int_{4|y-z|<|x-y|<{a}\rho(y)}| (\mathcal{W}_t(y,x)- W_t (y,x))-(\mathcal{W}_t(z,x)- W_t (z,x))|dx
			\nonumber\\
			&{ \lesssim  \left(\frac{|y-z|}{\rho(y)}\right)^{{\beta}} t^{-d/2}\int_{4|y-z|<|x-y|<{a}\rho(y)}\exp\left(-c\frac{|y-x|^2}{t}\right)dx} \lesssim  \left(\frac{r}{\rho(x_0)}\right)^{{\beta}}. 
		\end{align}
		
	Regarding the second integral, since $x$ is close to both $y$ and $z$, by Proposition~\ref{prop: comparison heat kernels}\ref{itm: diff heat kernels} and  $\sqrt{t} < \rho(x_0) \sim \rho(y) \sim \rho(z)$ we have
		\begin{align}\label{eq: 2nd integral diffs W_t}
			&\int_{|x-y|<4|y-z|}
			|(\mathcal{W}_t(y,x)- W_t (y,x))-(\mathcal{W}_t(z,x)- W_t (z,x))|dx\nonumber\\
			& 
			\leq \int_{|x-y|<4|y-z|} |\mathcal{W}_t(y,x)- W_t (y,x)|dx+\int_{|x-z|<5|y-z|} |\mathcal{W}_t(z,x)- W_t (z,x)|dx\nonumber\\
			&
			\lesssim \frac{1}{t^{d/2}}\left( \frac{\sqrt{t}}{\rho(x_0)} \right)^{\delta_\mu} \left[ \int_{|x-y|<4|y-z|} \exp\left(-c\frac{|x-y|^2}{t}\right)dx+\int_{|x-z|<5|y-z|} \exp\left(-c\frac{|x-z|^2}{t}\right)dx \right]\nonumber\\
			&
			\lesssim \frac{1}{t^{d/2}} \left( \frac{\sqrt{t}}{\rho(x_0)} \right)^{\delta_\mu}  \int_0^{5|y-z|} e^{-c\frac{u^2}{t}} u^{d-1} du  
			\lesssim \left( \frac{\sqrt{t}}{\rho(x_0)} \right)^{\delta_\mu} \left( \frac{|y - z|}{\sqrt{t}} \right)^d
			\lesssim \left( \frac{r}{\rho(x_0)} \right)^{\delta_\mu},
		\end{align}
	where we have used that $|y-z|<2r<\sqrt{t}$ and $\delta_\mu<d$.
	
	Finally, for the third term in \eqref{eq: 3rd estimate diff W_t1}, we use Lemma~\ref{lem: heat semigroup estimates}\ref{itm: diff heat kernels} for $\mathcal{W}_t$, and the smoothness condition for $W_t$. Also, since $|y-z|<\rho(x_0)$ and by the choice of $a$, it follows that 
	   $|x-y|>2|y-z|$. Therefore, for every $0<{\beta}<\min\{1,\delta_\mu\}$, we have
	   
	   \begin{align}\label{eq: 3rd integral diffs W_t}
	   	\int_{|x-y|>{a}\rho(y)}&|(\mathcal{W}_t(y,x)- W_t(y,x))-(\mathcal{W}_t(z,x)- W_t (z,x))|dx \nonumber\\
		&
		\leq \int_{|x-y|>{a}\rho(y)}|\mathcal{W}_t(y,x)-\mathcal{W}_t(z,x)|dx+\int_{|x-y|>a\rho(y)}| W_t (y,x)- W_t(z,x)|dx\nonumber\\
		&
		\lesssim |y-z|^{{\beta}}\int_{|x-y|>{a}\rho(y)}\frac{1}{|x-y|^{d+{\beta}}}dx
		\lesssim r^{\beta}\int_{{a}\rho(y)}^{\infty}\frac{1}{s^{d+{\beta}}}s^{d-1}ds
		\lesssim \left(\frac{r}{\rho(x_0)}\right)^{{\beta}}.
		\end{align}
		
		By combining \eqref{eq: 1st integral diffs W_t}, \eqref{eq: 2nd integral diffs W_t}, and \eqref{eq: 3rd integral diffs W_t}
	we obtain
	\begin{equation}\label{eq: diff heat kernel Banach}
		\| \mathcal{W}_t1(y)- \mathcal{W}_t1(z)\|_{\mathbb{B}}\leq  C  \left(\frac{r}{\rho(x_0)}\right)^{{\beta}}, \quad 0<{\beta}<\min\{1,\delta_\mu\},\end{equation}
	and finish the proof.	\end{proof}

		\subsection{Maximal operators for the generalized Poisson operators $\mathcal{P}^\sigma_t$ .}
		
	For $0 < \sigma < 1$ and $f \in L^2(\mathbb{R}^d)$, the generalized Poisson operators $\mathcal{P}^\sigma_t$ are defined through the subordination formula
	\begin{equation}\label{eq: subordination}\mathcal{P}^\sigma_tf(x)= \frac{t^{2\sigma}}{4^\sigma \Gamma(\sigma)} \int_{0}^{\infty} e^{-\frac{t^2}{4r}}\mathcal{W}_rf(x)\frac{dr}{r^{1+\sigma}} \
	= \frac{1}{\Gamma(\sigma)} \int_{0}^{\infty} e^{-r} \mathcal{W}_{\frac{t^2}{4r}}f(x) \frac{dr}{r^{1-\sigma}},
	\end{equation}
	for $x \in \mathbb{R}^d$ and $t > 0$, where $\mathcal{W}_{\frac{t^2}{4r}} f$ is as in \eqref{eq: heat semigroup}.
	Thus,
	\begin{equation*}\mathcal{P}^{\sigma}_t f(x)= \int_{\mathbb{R}^d} \mathcal{P}^{\sigma}_t(x,y)f(y)dy, \quad x \in \mathbb{R}^d,
	\end{equation*}
	where the expression for the kernel is given by
	\begin{equation}\label{eq: Poisson sigma kernel}
		\mathcal{P}^\sigma_t(x,y)
		= \frac{1}{\Gamma(\sigma)}\int_{0}^{\infty} e^{-r}\mathcal{W}_{\frac{t^2}{4r}}(x,y) \frac{dr}{r^{1-\sigma}}, \quad x,y \in \mathbb{R}^d.\end{equation}

	We shall consider the maximal operator associated with $\mathcal{P}^{\sigma}_{t}$, for $f \in L^2(\mathbb{R}^d)$, as
	\begin{equation}\label{eq: Poisson maximal}
		\mathcal{P}^{\sigma,*}f(x):= \sup_{t>0}|\mathcal{P}^{\sigma}_{t}f(x)|=\|\mathcal{P}^{\sigma}_{t}f(x)\|_{\mathbb{B}},
	\end{equation}
	in the Banach space $\mathbb{B}= L^\infty((0, \infty),dt)$ as before.  
Moreover, the inequality $\mathcal{P}^{\sigma,*}f(x) \leq \mathcal{W}^{*}f(x)$ for all $x \in \mathbb{R}^d$ and $f \in L^2(\mathbb{R}^d)$  derived from \eqref{eq: subordination}  yields the boundedness from  $L^2(\mathbb{R}^d)$ into $L^2_\mathbb{B}(\mathbb{R}^d)$. Additionally, we have the  same result for $\mathcal{P}^{\sigma,*}$ as for the maximal operator $\mathcal{W}^{*}$ on weighted Lebesgue spaces, that is, $\mathcal{P}^{\sigma,*}$ is bounded on $L^p(w)$  for $w\in H^{\rho, m^*}_{p,c^*}$, where $m^*=\frac{1}{(k_0+1)^2}$ and $c^*<c_0(2^{3k_0+7}C_0^{k_0+3})^{-m^*}$.
	
	In  $\BMO_\rho^\alpha(\mathbb R^d)$ spaces, these maximal operators $\mathcal{P}^{\sigma,*}$ were studied in \cite{MSTZ14}  for the particular case where $d\mu(x)=V(x)dx$. To obtain a weighted version of these results, for a general measure $\mu$, we shall prove first that size and smoothness conditions  with norm  $\mathbb{B}$ are satisfied.
	
	\begin{prop}\label{prop: class Poisson maximal} Let $x_0, x,y \in \mathbb{R}^d$, $t>0$ and $0<\sigma<1$.
		\begin{enumerate}\item \label{itm: size P_tsigma}There exists a constant $C$ such that		\[\|\mathcal{P}_t^{{\sigma}}(x,y)\|_{\mathbb{B}} \leq C \frac{1}{|x-y|^{d}}\exp\left(-c_0\left(1+\frac{|x-y|}{\rho(x)}\right)^{\frac{1}{k_0+1}}\right),\]
			where $c_0$ is the constant from Lemma~\ref{lem: heat semigroup estimates}\ref{itm: size W_t}.
			
			\item \label{itm: smoothness P_tsigma}
			For every $0<\delta<\min\{1,\delta_\mu\}$, there exists a constant $C$ such that		\[\|\mathcal{P}_t^{{\sigma}}(x,y) - \mathcal{P}_t^{{\sigma}}(x_0,y)\|_{\mathbb{B}} \leq C\frac{|x-x_0|^\delta}{|x-y|^{d+\delta}},\]
			 whenever $|x-y|>2|x-x_0|$.	\end{enumerate}\end{prop}
	
\begin{proof}
	The proof of item~\ref{itm: size P_tsigma} is a consequence of definitions \eqref{eq: Poisson maximal} and \eqref{eq: Poisson sigma kernel}, Proposition~\ref{prop: class heat maximal} and the definition of the Gamma function $\Gamma(\sigma) = \int_{0}^{\infty} r^{\sigma-1}e^{-r} dr$.

	To prove the smoothness condition \ref{itm: smoothness P_tsigma}, we use Proposition~\ref{prop: class heat maximal}\ref{itm: smoothness heat kernel Banach} to obtain
	\begin{align*}
		\| \mathcal{P}^\sigma_t(x,y) - \mathcal{P}^\sigma_t(x_0,y) \|_{\mathbb{B}}
		&
		\leq
		\frac{1}{\Gamma(\sigma)}\int_{0}^{\infty} e^{-r} \left\| \mathcal{W}_{\frac{t^2}{4r}}(x,y)-\mathcal{W}_{\frac{t^2}{4r}}(x_0,y) \right\|_{\mathbb{B}} \frac{dr}{r^{1-\sigma}} 
		\lesssim \frac{|x-x_0|^\delta}{|x-y|^{d+\delta}},
	\end{align*}
	whenever $|x-y|>2|x-x_0|$ and $0<\delta<\min\{1,\delta_\mu\}$.
\end{proof}
	\begin{thm}
		{Let $0 < \sigma < 1$, $m = \frac{1}{k_0+1}$, $c_0$ be the constant from Lemma~\ref{lem: heat semigroup estimates}\ref{itm: size W_t}, and ${\delta = \min\{1, \delta_\mu\}}$.} Then, for every $0 \leq \alpha < \delta$, the operator $\mathcal{P}^{\sigma,*}$ is bounded on $\BMO_\rho^\alpha(w)$ provided that { $w \in E_{\infty,c_1,c_2}^{\rho,m} \cap D_{\kappa,c_3}^{\rho,m}$ with $c_1, c_2, c_3 \geq 0$ such that $1 \leq \kappa < \frac{\delta - \alpha}{d} + 1$ and ${(c_2+c_3)} < c_0 \left(1 - \frac{d(\kappa-1)+\alpha}{\delta}\right) (4C_0)^{-m}$.}
	\end{thm}

	\begin{proof}Fix $0 < \sigma < 1$. By Corollary~\ref{cor: T1 criterion infty} and Remark~\ref{obs: vector valued}, and in view of the kernel conditions given in Proposition~\ref{prop: class Poisson maximal}, it is sufficient to prove that there exists a constant $C$ such that
		\begin{equation}\frac{1}{|B|} \int_{B} \| \mathcal{P}^\sigma_t 1(y) - (\mathcal{P}^\sigma_t 1)_B \|_{\mathbb{B}} dy \leq C \left(\frac{r}{\rho(x_0)}\right)^{\beta},
		\end{equation}
		for every ball $B=B(x_0,r)$ with $x_0 \in \mathbb{R}^d$, $0 < r \leq \frac{1}{2}\rho(x_0)$, and $0 \leq \beta < \min\{1, \delta_\mu\}$.
		
	To show this, we shall estimate $\| \mathcal{P}_t^\sigma 1(y) - \mathcal{P}^\sigma_t 1(z) \|_{\mathbb{B}}$ for all $y,z \in B=B(x_0,r)$. 
		According to the kernel representation of the generalized Poisson operator \eqref{eq: Poisson sigma kernel}, and using the estimate \eqref{eq: diff heat kernel Banach}, it follows that for every $0 < \beta < \min\{1, \delta_\mu\}$:
		\begin{align*}
			\| \mathcal{P}_t^\sigma 1(y) - \mathcal{P}_t^\sigma 1(z) \|_{\mathbb{B}} 
			& = \left\| \int_{\mathbb{R}^d} \frac{1}{\Gamma(\sigma)} \int_{0}^{\infty} e^{-s} \left(\mathcal{W}_{\frac{t^2}{4s}}(y,x) - \mathcal{W}_{\frac{t^2}{4s}}(z,x)\right) \frac{ds}{s^{1-\sigma}} dx \right\|_\mathbb{B} \\
			& \leq \frac{1}{\Gamma(\sigma)} \int_{0}^\infty s^{\sigma-1} e^{-s} \| \mathcal{W}_t 1(y) - \mathcal{W}_t 1(z) \|_{\mathbb{B}} ds\lesssim \left(\frac{r}{\rho(x_0)}\right)^{\beta}.
		\end{align*}
		By averaging over $y,z \in B$, we obtain the desired result.
	\end{proof}

	\subsection{Littlewood-Paley function for the heat semigroup}
	
	The Littlewood--Paley $g$-function associated with the family $\{\mathcal{W}_t\}_{t>0}$ is defined by \[g_{\mathcal{W}}(f)(x) = \left( \int_0^\infty \left| t \partial_t \mathcal{W}_t f(x) \right|^2 \, \frac{dt}{t} \right)^{1/2} = \|t \partial_t\mathcal{W}_t f(x) \|_\mathbb{F},\]
	 where $\mathbb{F}=L^2\left((0, \infty), \frac{dt}{t}\right)$.

	 For $d\mu(x)=V(x)dx$  with $V \in RH_q$, the boundedness of $g_{\mathcal{W}}$ on $L^p(w)$  for weights $w \in A_p^\rho$ was established in \cite{BHS11} and unweighted estimates in the $\BMO_\rho^\alpha(\mathbb R^d)$ space can be found in \cite{MSTZ14}. To extend this to general measures, we adopt the vector-valued framework using the Banach space $\mathbb{F}$. Thus, since $\{t\partial_t \mathcal{W}_t\}_{t>0}$ is bounded from $L^2$ to $L^2_\mathbb{F}$ it is enough to verify that the kernel satisfies the size and smoothness conditions  in the $\mathbb{F}$ norm to obtain that it is a vector-valued Calderón--Zygmund operator and bounded from $L^p$ to $L^p_\mathbb{F}$ for all $1 < p < \infty$.
	 
	Moreover,  as a consequence of Proposition~\ref{prop: class LPg heat} we shall prove  that \eqref{eq: size pointwise fract} and \eqref{eq: smoothness pointwise fract} are satisfied in $\mathbb{F}$, with parameters $c=c_0$ and $m=\frac{1}{k_0+1}$. This implies that $g_{\mathcal W}$ is bounded on $L^p(w)$ for $w\in H^{\rho, m^*}_{p,c^*}$, where $m^*=\frac{1}{(k_0+1)^2}$ and $c^*<c_0(2^{3k_0+7}C_0^{k_0+3})^{-m^*}$ (see \cite[Theorem~4.1 and Proposition~4.2]{DLT25}).

	 \begin{prop}\label{prop: class LPg heat} 
	 	Let $x_0, x, y \in \mathbb{R}^d$ and $t > 0$.
	 	\begin{enumerate}\item \label{itm: size LPg heat} There exists a constant $C$ such that 	
	 		\[ \|t \partial_t\mathcal{W}_t(x,y)\|_{\mathbb{F}} \leq \frac{C}{|x-y|^{d}}\exp\left(-c_0\left(1+\frac{|x-y|}{\rho(x)}\right)^{\frac{1}{k_0+1}}\right),\]
	 		where $c_0$ is the constant from Lemma~\ref{lem: heat semigroup estimates}\ref{itm: size W_t}.	\item \label{itm: smoothness LPg heat} For every $0 < \delta < \min\{1,\delta_\mu\}$ there exists a constant $C$ such that
			\[
			\|t \partial_t\mathcal{W}_t(x,y) -t \partial_t \mathcal{W}_t(x_0,y)\|_{\mathbb{F}} \leq 
			C \frac{|x-x_0|^\delta}{|x-y|^{d+\delta}},
			\]
			whenever $|x-y| > 2|x-x_0|$.	
		\end{enumerate}
	\end{prop}
	
	\begin{proof} The proof is analogous to \cite[Proposition~5.8]{DLT25}. For condition \ref{itm: size LPg heat} we use  Lemma~\ref{lem: Q_t estimates}\ref{item: t deriv parcial k_t}. Then,  there exist $C$ and $c$ such that
		\begin{align*}\|t\partial_t \mathcal W_t(x,y)\|_{\mathbb{F}}^2
		&\lesssim \exp\left(-2c_0\left(1+\frac{|x-y|}{\rho(x)}\right)^{\frac{1}{k_0+1}}\right)\int_0^{\infty} t^{-d-1} \exp\left(-2c\frac{ |x-y|^2}{t}\right)dt
		 \\&\lesssim \frac{1}{|x-y|^{2d}} \exp\left(-2c_0\left(1+\frac{|x-y|}{\rho(x)}\right)^{\frac{1}{k_0+1}}\right).\end{align*}
		 
		 For condition \ref{itm: smoothness LPg heat}, consider $|x-y| > 2|x-x_0|$. From Lemma~\ref{lem: Q_t estimates}\ref{itm: t diff partial W_t}  it follows, on one hand, that
		\begin{equation*}
			\int_{|x-x_0|^2}^\infty |t \partial_t\mathcal{W}_t(x,y)-t \partial_t\mathcal{W}_t(x_0,y)|^2 \frac{dt}{t}\lesssim\frac{1}{|x-y|^{2d}} \left(\frac{|x-x_0|}{|x-y|}\right)^{2\delta},
		\end{equation*}
		for every $0 < \delta < \min\{1,\delta_\mu\}$ and some constant $C$.
		
		On the other hand, keeping in mind that $|x_0-y| > \frac12|x-y|$, from Lemma~\ref{lem: Q_t estimates}\ref{item: t deriv parcial k_t} we have
		\begin{align*}
			\int_0^{|x-x_0|^2} |t \partial_t\mathcal{W}_t(x,y)-t \partial_t\mathcal{W}_t(x_0,y)|^2 \frac{dt}{t} 
			&\leq \int_0^{|x-x_0|^2} |t \partial_t\mathcal{W}_t(x,y)|^2 \frac{dt}{t}+ \int_0^{|x-x_0|^2} |t \partial_t\mathcal{W}_t(x_0,y)|^2 \frac{dt}{t}\\
			& \lesssim \left(\frac{|x-x_0|}{|x-y|}\right)^{2\delta} \int_0^\infty t^{-d-1} \exp\left(-c\frac{|x-y|^2}{t}\right) dt\\
			& \lesssim \frac{1}{|x-y|^{2d}} \left(\frac{|x-x_0|}{|x-y|}\right)^{2\delta},
		\end{align*}
		for every $0 < \delta < \min\{1,\delta_\mu\}$. 
	\end{proof}

	\begin{thm}Let $m=\frac{1}{k_0+1}$, $c_0$ be the constant from Lemma~\ref{lem: heat semigroup estimates}\ref{itm: size W_t}, and $\delta = \min\{1, \delta_\mu\}$. Then, for every $0 \leq \alpha < \delta$, the operator $g_\mathcal{W}$ is bounded on $\BMO_\rho^\alpha(w)$ provided that $w \in E_{\infty,c_1,c_2}^{\rho,m} \cap D_{\kappa,c_3}^{\rho,m}$ with $1 \leq \kappa < \frac{\delta - \alpha}{d}+1$ and $c_1, c_2, c_3 \geq 0$ such that ${(c_2+c_3)} < c_0 \left(1 - \frac{d(\kappa-1)+\alpha}{\delta}\right) (4C_0)^{-m}$.\end{thm}
	
	\begin{proof}By Corollary~\ref{cor: T1 criterion infty} and  Remarks~\ref{obs: reduce to beta=0} and  \ref{obs: vector valued}, we only need to prove  that there exists  $C$ such that
		\begin{equation}\label{eq: T1 cond LPg heat}\frac{1}{|B|} \int_{B} \| t \partial_t\mathcal{W}_t 1(y) - (t \partial_t\mathcal{W}_t 1)_B \|_\mathbb{F} dy \leq C \left(\frac{r}{\rho(x_0)}\right)^{\beta},
		\end{equation}
		for every ball $B=B(x_0,r)$ with $0 < r \leq \frac{1}{2}\rho(x_0)$ and $0 < \beta < \min\{1, \delta_\mu\}$.
		From Minkowski's inequality,  it will suffice to show that
		\begin{equation*}
			\| t \partial_t\mathcal{W}_t 1(y) - t \partial_t\mathcal{W}_t 1(z) \|_{\mathbb F} \leq C \left(\frac{r}{\rho(x_0)}\right)^{\beta},
		\end{equation*}
		for all $y,z \in B$ and every $0 \leq \beta < \min\{1, \delta_\mu\}$. Note that
		\begin{align*}
			\| t \partial_t\mathcal{W}_t 1(y) - t \partial_t\mathcal{W}_t 1(z)\|_\mathbb{F}^2 = \int_{0}^{\infty} \left| \int_{\mathbb{R}^d} (t \partial_t\mathcal{W}_t(y,x) - t \partial_t\mathcal{W}_t(z,x)) dx \right|^2 \frac{dt}{t}.
			\end{align*}
			We proceed as in the proof of Theorem~\ref{thm: BMO Laplace multipliers}, splitting the integral in the $t$ variable into three regions.
		
		For $0 < t \leq 4r^2$, using that $\rho(y) \sim \rho(x_0) \sim \rho(z)$ since $y,z \in B$, and applying Lemma~\ref{lem: Q_t estimates}\ref{itm: integral Q_t} twice, we have
		\begin{align*}
			\int_{0}^{4r^2} \left| \int_{\mathbb{R}^d} (t \partial_t\mathcal{W}_t(y,x) - t \partial_t\mathcal{W}_t(z,x)) dx \right|^2 \frac{dt}{t} 
			&\lesssim \int_{0}^{4r^2} \left(\frac{\sqrt{t}}{\rho(x_0)}\right)^{2\delta_\mu} \frac{dt}{t} \sim  \left(\frac{r}{\rho(x_0)}\right)^{2\delta_\mu}.
		\end{align*}
		
		Next, if $4r^2 < t \leq \rho(x_0)^2$, we use \eqref{eq: integral W_t dmu } to rewrite the integral in $x$. Since $|y-z| \leq \sqrt{t}$, we can apply Lemma~\ref{lem: Q_t estimates}\ref{itm: t diff partial W_t}. Next, we  use \eqref{eq: integral exp heat bound dmu} because $t < C\rho(y)^2$ in this region and thus, for every $0 < \beta < \min\{1, \delta_\mu\}$,
		\begin{align*}
			\int_{4r^2}^{\rho(x_0)^2} \Bigg| \int_{\mathbb{R}^d} (t \partial_t\mathcal{W}_t(y,x) &- t \partial_t\mathcal{W}_t(z,x)) dx \Bigg|^2 \frac{dt}{t} 
			= \int_{4r^2}^{\rho(x_0)^2} t \left| \int_{\mathbb{R}^d} (\mathcal{W}_t(y,x) - \mathcal{W}_t(z,x)) d\mu(x) \right|^2 dt \\
			&\lesssim |y-z|^{2\beta} \int_{4r^2}^{\rho(x_0)^2} t^{1-\beta} \left| \int_{\mathbb{R}^d} t^{-d/2} \exp\left(- c\frac{|y-x|^2}{t}\right) d\mu(x) \right|^2 dt \\
			&\lesssim r^{2\beta} \int_{0}^{\rho(x_0)^2} t^{-\beta-1} \left(\frac{\sqrt{t}}{\rho(y)}\right)^{2\delta_\mu} dt \sim \left(\frac{r}{\rho(x_0)}\right)^{2\beta}.
		\end{align*}
		
		Finally, to analyze the integral when $t > \rho(x_0)^2$, we see that $|y-z| < 2r \leq \rho(x_0) < \sqrt{t}$. Then, by Lemma~\ref{lem: Q_t estimates}\ref{itm: t diff partial W_t}
		\begin{align*}
			\int_{\rho(x_0)^2}^{\infty} \Bigg| \int_{\mathbb{R}^d} (t \partial_t\mathcal{W}_t(y,x)& - t \partial_t\mathcal{W}_t(z,x)) dx \Bigg|^2 \frac{dt}{t} \\
			&\lesssim \int_{\rho(x_0)^2}^{\infty} \left(\frac{|y-z|}{\sqrt{t}}\right)^{2\beta} \left( \int_{\mathbb{R}^d} t^{-d/2} \exp\left(-c\frac{|y-x|^2}{t}\right) dx \right)^2 \frac{dt}{t} \\
			&\lesssim r^{2\beta} \int_{\rho(x_0)^2}^{\infty} t^{-\beta-1} dt \sim \left(\frac{r}{\rho(x_0)}\right)^{2\beta},
		\end{align*}
		for every $0 < \beta < \min\{1, \delta_\mu\}$, which concludes the proof of \eqref{eq: T1 cond LPg heat}.
	\end{proof}

\subsection{Littlewood-Paley function for the Poisson semigroup}

The Littlewood-Paley $g$-function associated with the Poisson semigroup is defined  for $f\in L^2(\mathbb{R}^d)$ as 

\begin{equation*} 
	g_{\mathcal{P}}(f)(x)= \left( \int_{0}^{\infty}|t \partial_t\mathcal{P}_tf(x)|^2\frac{dt}{t}\right)^{1/2}=\|t \partial_t \mathcal{P}tf(x)\|_{\mathbb F}, \quad x\in \mathbb{R}^d,
\end{equation*}
where $\mathbb{F}=L^2\left((0, \infty), \frac{dt}{t}\right)$, and $\mathcal{P}_t$ is the classical Poisson semigroup  $\mathcal{P}^{1/2}_t$.

By the  subordination representation of the Poisson kernel for $\sigma = \tfrac{1}{2}$ given in \eqref{eq: Poisson sigma kernel}, the estimates for weighted  Lebesgue and BMO spaces are obtained in the same manner as those for $g_\mathcal{W}$. Specifically, we derive the same size and smoothness estimates corresponding to \eqref{eq: size pointwise fract} and \eqref{eq: smoothness pointwise fract} considering the norm $\|\cdot\|_\mathbb{F}$, with parameters $c=c_0$ and $m=\frac{1}{k_0+1}$. This implies that $g_\mathcal{P}$ is bounded on $L^p(w)$ for $w\in H^{\rho, m^*}_{p,c^*}$, where $m^*=\frac{1}{(k_0+1)^2}$ and $c^*<c_0(2^{3k_0+7}C_0^{k_0+3})^{-m^*}$
 (see \cite[Theorem~4.1 and Proposition~4.2]{DLT25}).
The proofs of the following results are omitted since they follow the same lines as the ones for $g_\mathcal{W}$.

 \begin{prop}\label{prop: class LPg Poisson}Let $x_0, x, y \in\mathbb{R}^d$ and $t>0$.
 	\begin{enumerate}
 		\item \label{itm: size LPg Poisson} There exists a constant $C$ such that
 		\[\|t \partial_t\mathcal{P}_t(x,y)\|_\mathbb{F} \leq C \frac{1}{|x-y|^{d}}\exp\left(-c_0\left(1+\frac{|x-y|}{\rho(x)}\right)^{\frac{1}{k_0+1}}\right),\]where $c_0$ is the constant from Lemma~\ref{lem: heat semigroup estimates}\ref{itm: size W_t}.
 		\item \label{itm: smoothness LPg Poisson} For every $0<\delta<\min\{1,\delta_\mu\}$, there exists a constant $C$ such that
 		\[\|t \partial_t\mathcal{P}_t(x,y) - t \partial_t\mathcal{P}t(x_0,y)\|_\mathbb{F} \leq C \frac{|x-x_0|^\delta}{|x-y|^{d+\delta}},\]whenever $|x-y|>2|x-x_0|$.
 			\end{enumerate}
 	\end{prop}

\begin{thm}Let $m=\frac{1}{k_0+1}$, $c_0$ be the constant from Lemma~\ref{lem: heat semigroup estimates}\ref{itm: size W_t}, and $\delta = \min\{1, \delta_\mu\}$. Then, for every $0 \leq \alpha < \delta$, the operator $g_\mathcal{P}$ is bounded on $\BMO_\rho^\alpha(w)$ provided that $w \in E_{\infty,c_1,c_2}^{\rho,m} \cap D_{\kappa,c_3}^{\rho,m}$ with $1 \leq \kappa < \frac{\delta - \alpha}{d}+1$ and $c_1, c_2, c_3 \geq 0$ such that ${(c_2+c_3)} < c_0 \left(1 - \frac{d(\kappa-1)+\alpha}{\delta}\right) (4C_0)^{-m}$.\end{thm}

\subsection{Fractional integral operators}

Given $\gamma>0$, the negative powers of $\mathcal{L}$  admit the following integral representation  for $f\in L^2(\mathbb{R}^d)$ and $x\in\mathbb{R}^d$:\[\mathcal{L}_\mu^{-\gamma} f(x) = \frac{1}{\Gamma(\gamma)} \int_0^\infty e^{-t\mathcal{L}_\mu}f(x) \frac{dt}{t^{1-\gamma}}= \int_{\mathbb{R}^d} {\mathcal{J}_\gamma}(x,y) f(y)dy,\] 
where
\begin{equation}\label{eq: Kgamma kernel}{\mathcal{J}_\gamma}(x,y) = \frac{1}{\Gamma(\gamma)} \int_0^\infty \mathcal{W}_t(x,y) \frac{dt}{t^{1-\gamma}}, \qquad x,y\in\mathbb{R}^d.
\end{equation}

The case of negative powers for the operator $\mathcal{L}_V$ was first analyzed in \cite{BHS11}, where weighted Lebesgue estimates for $A_p^{\rho}$ were obtained, alongside the corresponding weak-type $(1, \frac{d}{d-2\gamma})$ results for $w\in A_1^\rho$. Subsequently, the authors in \cite{BCH13Extrapol} established bounds from weighted Lebesgue spaces into $\BMO_\rho(w)$ spaces for $w \in A_p^\rho$. On the other hand, \cite{MSTZ14} obtained estimates for $\mathcal{L}_V^{-\gamma}$ between unweighted $\BMO_\rho^\alpha(\mathbb R^d)$ spaces using a $T1$-type theorem, while \cite{BHS08} provided $\BMO_\rho^\alpha(w)$ estimates using techniques adapted to the operator.

The boundedness for $\mathcal{L}_\mu^{-\gamma}$ in both $L^p(w)$ and $\BMO_\rho^\alpha(w)$ spaces, where $w$ belongs to the $H_{p,c}^{\rho,m}$ weight class is given below.

\begin{prop}\label{prop: class negative powers}Let $0<\gamma<d/2$. Then, $\mathcal{L}_\mu^{-\gamma}$ is an exponential SCZO operator of $(2\gamma,\infty,\delta)$ type for all $0<\delta<\min\{1,\delta_\mu\}$, with parameters $c=c_0$ (from Lemma~\ref{lem: heat semigroup estimates}\ref{itm: size W_t}) and $m=\frac{1}{k_0+1}$.
\end{prop}

\begin{proof}

	 From Lemma~\ref{lem: heat semigroup estimates}\ref{itm: size W_t} it follows that
	 	\begin{align}\label{eq: size negative powers}
	 	|{\mathcal{J}_\gamma}(x,y)|
	 	&\leq \frac{1}{\Gamma(\gamma)} \exp\left(-c_0\left(1+\frac{|x-y|}{\rho(x)}\right)^{\frac{1}{k_0+1}}\right) \int_0^\infty t^{-\frac d2+\gamma-1} \exp\left(-c\frac{|x-y|^2}{t}\right) dt \nonumber \\ 
	 	& \lesssim \frac{1}{|x-y|^{d-2\gamma}} \exp\left(-c_0\left(1+\frac{|x-y|}{\rho(x)}\right)^{\frac{1}{k_0+1}}\right),
	 \end{align}
	 and  \eqref{eq: size pointwise fract} is verified. 
	
	For the smoothness condition  \eqref{eq: smoothness pointwise fract}, let $|x-y|>2|x-x_0|$. Taking into account  
	\begin{align*}
		|{\mathcal{J}_\gamma}(x, y)-{\mathcal{J}_\gamma}(x_0, y)|
		&= \frac{1}{\Gamma(\gamma)} \int_0^\infty |\mathcal{W}_t(x, y)-\mathcal{W}_t(x_0, y)|  \frac{dt}{t^{1 - \gamma}},
	\end{align*}
	we consider two cases.
	
	If $t\leq|x-x_0|^2$,  we apply Lemma~\ref{lem: heat semigroup estimates}\ref{itm: size W_t} and estimate \eqref{eq: exp polynom} with $N=\delta/2$ , $0<~\delta<~\min\{1,\delta_\mu\}$,  
\begin{align*}
	|\mathcal{W}_t(x,y)|
	&
	\lesssim t^{-d/2} \exp\left(-c\frac{|x-y|^2}{t}\right) \left(\frac{|x-x_0|}{\sqrt{t}}\right)^{\delta}\\
	&
	\lesssim t^{-d/2}\exp\left(-c\frac{|x-y|^2}{2t}\right)  \left(\frac{|x-y|}{\sqrt{t}}\right)^{-\delta}\left(\frac{|x-x_0|}{\sqrt{t}}\right)^{\delta}\\ 
	&
	\lesssim 
	  t^{-d/2} \left(\frac{|x-x_0|}{|x-y|}\right)^{\delta} \exp\left(-c\frac{|x-y|^2}{2t}\right), 
\end{align*}
and  we obtain the same bound for $|\mathcal{W}_t(x_0,y)|$ since $|x_0-y|>\frac{1}{2}|x-y|$.
	
	For $t>|x-x_0|^2$, by Lemma~\ref{lem: heat semigroup estimates}\ref{itm: smoothnes W_t} and \eqref{eq: exp polynom} for $N=\delta/2$, we have for 
 $0<\delta<\min\{1,\delta_\mu\}$
\begin{align*}
	| \mathcal{W}_t(x,y)- \mathcal{W}_t(x_0,y)|
	&\lesssim t^{-d/2} \left(\frac{|x-x_0|}{\sqrt{t}}\right)^{\delta} \exp\left(-c\frac{|x-y|^2}{t}\right)\\
	&\lesssim \left(\frac{|x-x_0|}{|x-y|}\right)^{\delta} t^{-d/2}\exp\left(-c\frac{|x-y|^2}{t}\right).
\end{align*} 

Then, integrating in $t$ and making the change of variable $u=\frac{|x-y|^2}{2t}$,
\begin{align*}
	\int_{0}^\infty |\mathcal{W}_t(x, y)-\mathcal{W}_t(x_0, y)|  \frac{dt}{t^{1 - \gamma}}
	&
	\lesssim\frac{1}{|x - y|^{d - 2\gamma}} \left(\frac{|x-x_0|}{|x-y|}\right)^{\delta}, \quad 0<\delta<\min\{1,\delta_\mu\}.
\end{align*}

	Finally, we prove that $\mathcal{L}_\mu^{-\gamma}$ is bounded from $L^1(\mathbb{R}^d)$ into $L^{\frac{d}{d-2\gamma},\infty}(\mathbb{R}^d)$. This follows from \eqref{eq: size negative powers} which implies $|{\mathcal{J}_\gamma}(x,y)| \leq C|x-y|^{-(d-2\gamma)}$. Thus, $|\mathcal{L}_\mu^{-\gamma}f(x)| \le C I_{2\gamma}(|f|)(x)$, where $I_{2\gamma}$ is the classical fractional integral. By the Hardy--Littlewood--Sobolev Theorem, this operator is bounded from $L^1(\mathbb R^d)$ to $L^{\frac{d}{d-2\gamma},\infty}(\mathbb R^d)$, completing the proof.
\end{proof}

As a consequence of Proposition~\ref{prop: class negative powers} we have that, for $0<\gamma<d/2$, $\mathcal{L}_\mu^{-\gamma}$ is bounded from $L^p(w^p)$ into $L^q(w^q)$ for all $1<p<\frac{d}{2\gamma}$ and $\frac{1}{q}=\frac{1}{p}-\frac{2\gamma}{d}$, and any weight $w$ such that $w^{-p'}\in H_{1+\frac{p'}{q},c}^{\rho,m^*}$ with $m^*=\frac{1}{(k_0+1)^2}$ and $c<c_0\frac{d}{d-2\gamma}(2^{3k_0+7}C_0^{k_0+3})^{-m^*}$.  The proof of this result is omitted, as this follows the same lines as the singular case given in  \cite[Theorem~4.1, Proposition~4.2]{DLT25}.

\begin{thm}\label{thm: BMO negative powers}
	Let $0<\gamma<d/2$, $\delta = \min\{1,\delta_\mu\}$, $m=\frac{1}{k_0+1}$, and $c_0$ be the constant from Lemma~\ref{lem: heat semigroup estimates}\ref{itm: size W_t}. Then, for all $0 \leq \alpha+2\gamma< \delta$, the operator $\mathcal{L}_\mu^{-\gamma}$ is bounded from $\BMO_\rho^\alpha(w)$ into $\BMO_\rho^{\alpha+2\gamma}(w)$ provided that $w\in E_{\infty,c_1,c_2}^{\rho,m}\cap D_{\kappa,c_3}^{\rho,m}$ with $1 \leq \kappa < \frac{\delta - 2\gamma-\alpha}{d}+1$ and $c_1,c_2,c_3\geq 0$ such that ${(c_2+c_3)}<c_0\left(1-\frac{d(\kappa-1)+\alpha+2\gamma}{\delta}\right)(4C_0)^{-m}$.
\end{thm}

\begin{proof} By Corollary~\ref{cor: T1 criterion infty} and Remark~\ref{obs: reduce to beta=0}, it suffices to show that there exists a constant $C$ such that
	\begin{equation}\label{eq: T1 cond negative powers}\frac{1}{|B|^{1+{2\gamma}/d}} \int_{B} \left| \mathcal{L}_\mu^{-\gamma}1(y) - \left(\mathcal{L}_\mu^{-\gamma}1\right)_B \right| dx \leq C \left( \frac{r}{\rho(x_0)} \right)^{{\beta}},
	\end{equation}
	for every ball $B=B(x_0,r)$ with $0<r\leq \frac{1}{2}\rho(x_0)$ and $0<{\beta}<\min\left\{1,\delta_\mu\right\}-{2\gamma}$. 
	
	Given $y,z \in B$, we note that 
	\begin{align*}
		|\mathcal{L}_\mu^{-\gamma}1(y) - \mathcal{L}_\mu^{-\gamma}1(z)|
		&\leq \int_0^\infty \left|\mathcal{W}_t1(y)-\mathcal{W}_t1(z)\right|t^{\gamma-1} dt.
	\end{align*}
	
	Then, for $0<t\leq \rho(x_0)^2$, we use \eqref{eq: diff heat kernel Banach} to obtain
	\begin{equation*}
		\int_0^{ \rho(x_0)^2} \left|\mathcal{W}_t1(y)-\mathcal{W}_t1(z)\right|t^{\gamma-1} dt \lesssim \left(\frac{r}{\rho(x_0)}\right)^{\tilde{\beta}}\int_0^{\rho(x_0)^2} t^{\gamma-1} dt\sim  \left(\frac{r}{\rho(x_0)}\right)^{\tilde{\beta}} \rho(x_0)^{2\gamma},
	\end{equation*}
	provided that $0<\tilde{\beta}<\min\left\{1,\delta_\mu\right\}$. 
	
	When $t>\rho(x_0)^2$, we use \eqref{eq: 2nd estimate diff W_t1}, which yields
	\begin{align*}
		\int_{\rho(x_0)^2}^\infty \left|\mathcal{W}_t1(y)-\mathcal{W}_t1(z)\right|t^{\gamma-1} dt 
		&\lesssim \int_{\rho(x_0)^2}^\infty\left(\frac{r}{\sqrt{t}}\right)^{\tilde{\beta}}  t^{\gamma-1} dt
	\sim \left(\frac{r}{\rho(x_0)}\right)^{\tilde{\beta}}\rho(x_0)^{2\gamma},
	\end{align*}
	for all ${2\gamma}<\tilde{\beta}<\min\left\{1,\delta_\mu\right\}$.
	
	Combining both estimates and setting ${\beta}=\tilde{\beta}-{2\gamma}$, it follows that \eqref{eq: T1 cond negative powers} holds for every \linebreak $0<{\beta}+2\gamma<\min\left\{1,\delta_\mu\right\}$.
	\end{proof}
	
	\subsection{Operators associated to a potential}
	In the particular case where  $d\mu(x)= V(x) \,dx$, with $V$ being a non-negative function in the classical reverse--H\"older class $RH_q$ for $q>\frac{d}{2}$, we can study the operators $T_\gamma= \mathcal{L}_V^{-\gamma} V^{\gamma}$ for $0<\gamma<d/2$. These are defined for $f \in L^2(\mathbb{R}^d)$ by
	\begin{align*}T_\gamma f(x) = \int_{\mathbb{R}^d} {\mathcal{J}_\gamma}(x,y)V^\gamma(y) f(y) dy,
	\end{align*}
	where $\mathcal{J}_\gamma(x,y)$ is the kernel of the operator $\mathcal{L}_V^{-\gamma}$ defined in \eqref{eq: Kgamma kernel}.
	
In the cases $\gamma=1/2$ and $\gamma=1$, $L^p(\mathbb{R}^d)$ boundedness was established in \cite{Shen95} for $1\leq p\leq 2q$ and $1\leq p\leq q$, respectively. Subsequently, \cite{GLP08} provided pointwise size and smoothness kernel estimates featuring polynomial decay associated with the critical radius function $\rho$.
	 From these estimates, \cite{BCH13Extrapol} proved the boundedness of the operator on $L^p(w)$ for weights in the $A_p^{\rho}$ class. The general case $0<\gamma<d/2$ was studied in \cite{BHQ19}, where the authors obtained  the boundedness of $T_\gamma$ on $L^p(w)$ for $(q/\gamma)'<p<\infty$ and $w\in A_{p/(q/\gamma)'}^\rho$. They provided the estimates \eqref{eq: smoothness Hormander fract} for the kernel $\mathcal{J}_\gamma V^\gamma$ with $s=q/\gamma>1$ and $\delta=\min\{1,2-d/q\}$.	 
	 Consequently, by proving the size condition \eqref{eq: size Hormander fract} for the kernel ${\mathcal{J}_\gamma} V^\gamma$ we obtain Proposition~\ref{prop: class op with potential}, which is stated below. Then, following the same lines as the singular case given in  \cite[Theorem~4.1, Proposition~4.2]{DLT25}  we obtain  that  $T_\gamma$ is bounded on $L^p(w)$ for all $(q/\gamma)'<p<\infty$ and $w \in H^{\rho,m^*}_{p/(q/\gamma)',c^*}$, with $m^*=\frac{1}{(k_0+1)^2}$ and $c^*<c_02^{-(m+1)}(2^{2k_0+6}C_0^{k_0+3})^{-{m^*}}$. This result extends the one given in \cite{BHQ19}.

	 	 \begin{prop}\label{prop: class op with potential} Let $V\in RH_q$, $q>\frac d2$, $V\not\equiv 0$ and $0<\gamma<\frac d2$. Then, $T_\gamma$ is an exponential SCZO of $(0,\frac{q}{\gamma},\delta)$ type, with $\delta=\min\left\{1,2-\frac dq\right\}$, constants $m=\frac{1}{k_0+1}$, and $c=c_02^{-(m+1)}$.
	 \end{prop}

\begin{proof}
	As we saw before, it is enough to prove \eqref{eq: size Hormander fract}. For that, by \eqref{eq: size negative powers}, and using that $V\in RH_q$ it follows that
	 	\begin{align*}
	 		&\left(\frac{1}{R^d}\int_{R<|x_0-y|\leq 2R}|{\mathcal{J}_\gamma}(x,y)V^\gamma(y)|^{q/\gamma} dy \right)^{\gamma/q}\\ 
	 		& \lesssim \frac{1}{R^{d-2\gamma}} \exp\left(-c_02^{-m}\left(1+\frac{R}{\rho(x)}\right)^m\right)\left(\frac{1}{R^d}\int_{B(x,3R)} V(y) dy \right)^{\gamma}\\
	 		&\lesssim \frac{1}{R^{d-2\gamma}} 		\exp\left(-c_02^{-m}\left(1+\frac{R}{\rho(x)}\right)^m\right) R^{-2\gamma}\left(1+\frac{R}{\rho(x)}\right)^{\log_2(D_V)\gamma}\\
	 		& \lesssim \frac{1}{R^{d}} \exp\left(-c_02^{-m-1}\left(1+\frac{R}{\rho(x)}\right)^m\right),
	 	\end{align*}
	 where we also used \cite[Lemma~5.2]{DLT25} with  $d\mu(x)=V(x)dx$ and $D_V$ the doubling constant of $V(x) dx$. 
		 \end{proof}
	 
	 Additionally, in \cite{BHQ19} a $T1$-type condition was derived, enabling the boundedness of $T_\gamma$ on weighted $\BMO_\rho^\alpha(w)$ spaces for certain weights in $A_p^\rho$ . Specifically, the authors proved condition \eqref{eq: T1 power condition} with $0<\alpha+d(\kappa-1)<\tilde{\delta}$, where  $\tilde{\delta} = \min\{1, 2-d/q, \gamma (2 - d/q) \}$. 
		As a consequence of this fact, Proposition~\ref{prop: class op with potential} and Theorem~\ref{thm: T1 criterion} we obtain the following result, which broadens the classes of weights given in \cite{BHQ19} for $\BMO_\rho^\alpha(w)$.
		
	\begin{thm}Let $V\in RH_q$, $q>\frac d2$, $V\not\equiv 0$ and $0<\gamma<\frac d2$. For $0 \leq \alpha < \tilde{\delta}$, $m=\frac{1}{k_0+1}$ and $c=c_02^{-(m+1)}$,  $T_\gamma$ is bounded on $\BMO_\rho^\alpha(w)$ whenever $w \in E_{\frac{q}{\gamma},c_1,c_2}^{\rho,m} \cap D_{\kappa,c_3}^{\rho,m}$, with $1 \leq \kappa < \frac{\tilde{\delta} - \alpha}{d}+1$ and $(c_2+c_3)<c\left(1-\frac{d(\kappa-1)+\alpha}{\tilde{\delta}}\right)(4C_0)^{-m}$.
	\end{thm}

\subsection{The operator \texorpdfstring{$\mathcal{L}_\mu^{-\gamma}\nabla$}{Lμ-γ∇}}

Note that $\mathcal{L}_\mu^{-\gamma}\nabla$ can be considered to have fractional order $\nu=2\gamma - 1$, since $\mathcal{L}_\mu^{-\gamma}$ is of order $2\gamma$. 
 Thus, $\nu>0$ provided that $1/2 < \gamma \leq 1$. When $\gamma=1$, 
 \[\mathcal{L}_\mu^{-1}\nabla f(x)=\int_{\mathbb{R}^d} \Gamma_\mu(x,y)\nabla f(y)dy= \int_{\mathbb{R}^d} \mathcal{K}_{1}(x,y) f(y)dy,\]
 for $f\in L^2(\mathbb{R}^d)$ and ${\mathcal{K}_1}(x,y)=-\nabla_2\Gamma_\mu(x,y),$
 where $\nabla_2$ denotes the gradient with respect to the second variable.
 But, by the symmetry of $\Gamma_\mu$ (see \cite[Theorem~2.18]{Shen99}), we have $\nabla_2\Gamma_\mu(x,y)=\nabla_1\Gamma_\mu(y,x)$ and applying \cite[Theorem~4.2]{Bailey21} 
 we have 
  that  there exist positive constants $C$ and $\epsilon$ such that
 \begin{equation}\label{eq: size K1}|{\mathcal{K}_1}(x,y)|=|\nabla_1\Gamma_\mu(y,x)|  \leq \frac{C e^{-\epsilon d_{\mu}(y,x)}}{|y-x|^{d-2}} \left( \int_{B(y,\frac{|y-x|}{2})} \frac{d\mu(z)}{|z-y|^{d-1}} + \frac{1}{|y-x|} \right),
 \end{equation}for all $x,y \in \mathbb{R}^d$ with $x \neq y$. 
 Moreover, if $\delta_\mu>1$, then
 \begin{equation*}
 	|{\mathcal{K}_1}(x,y)| \leq \frac{C e^{- \epsilon d_{\mu}(y,x)}}{|y-x|^{d-1}}.
 \end{equation*}
 
  For $1/2 < \gamma < 1$, we use Balakrishnan's formula (see \cite[Eq.~(3.53), p.~286]{Kato66}) to express the operator as
  \[
  \mathcal{L}_\mu^{-\gamma} = \frac{\sin(\pi\gamma)}{\pi} \int_0^\infty \lambda^{-\gamma} \mathcal{L}_{\mu+\lambda}^{-1} \, d\lambda,
  \]
  and then 
  \begin{equation*}\mathcal{L}_\mu^{-\gamma}\nabla f(x)
   = \int_{\mathbb{R}^d} {\mathcal{K}_\gamma}(x,y) f(y) \, dy,
\end{equation*}
where
\begin{equation}\label{eq: L-gamma nabla kernel}
	{\mathcal{K}_\gamma}(x,y) = -\frac{\sin(\pi\gamma)}{\pi} \int_0^\infty \lambda^{-\gamma} \nabla_1 \Gamma_{\mu + \lambda}(y,x)d\lambda.
	\end{equation}

	Based on these estimates, we can establish the following classification for $\mathcal{L}_\mu^{-\gamma}\nabla$. 
	
	\begin{prop}\label{prop: class L-gamma nabla} Let $m_0=\frac{1}{k_0+1}$ and   $1/2 < \gamma \leq 1$.
		\begin{enumerate}
			\item \label{itm: class L-gamma nabla deltamu>1}
	If  $\delta_\mu > 1$, $\mathcal{L}_\mu^{-\gamma}\nabla$ is an exponential SCZO of $(2\gamma-1, \infty, \delta)$ type, with parameters $c = \frac{\epsilon}{2D_1}$ and $m_0$.
	\item \label{itm: class L-gamma nabla deltamu<1}
If $0<\delta_\mu<1$,  $\mathcal{L}_\mu^{-\gamma}\nabla$ is an exponential SCZO of $(2\gamma-1,s,\delta)$ type for   ${\frac{d}{d-(2\gamma-1)}}<s<\frac{2-\delta_\mu}{1-\delta_\mu}$ and $\delta \in (0,1)$, with parameters $c={\frac{\epsilon}{4D_1}}$ and $m_0$.
\end{enumerate}
Here, $\delta\in (0,1)$ is as in \cite[Eq. (7.29)]{Shen99} and $D_1$ is the constant given by \cite[Lemma~2.3]{Bailey21}. 
	\end{prop}

	\begin{proof}First, note that from \eqref{eq: size K1} and \eqref{eq: metric mu+lambda} we obtain, for every $\lambda\geq 0$ and $\delta_\mu>1$, that
		\begin{align}\label{eq: size gradient fundamental solution}|\nabla_1\Gamma_{\mu+\lambda}(y,x)|
			 &\lesssim \frac{e^{-{\frac{\epsilon}{2}}\sqrt{\omega_d}\sqrt{\lambda}|x-y|}}{|y-x|^{d-1}} \exp\left(-{\frac{\epsilon}{2D_1}}\left(1+\frac{|y-x|}{\rho(x)}\right)^{\frac{1}{k_0+1}}\right),
			 \end{align}
			 where $\rho(x)$ can be replaced by $\rho(y)$ due to the symmetry of the distance $d_\mu$. It follows immediately that for $\gamma=1$ (and thus $\nu=2\gamma-1=1$) 
			 \begin{equation*}|{\mathcal{K}_1}(x,y)|=|\nabla_1\Gamma_{\mu}(y,x)| \lesssim \frac{1}{|y-x|^{d-1}} \exp\left(-{\frac{\epsilon}{2D_1}}\left(1+\frac{|y-x|}{\rho(x)}\right)^{\frac{1}{k_0+1}}\right).
			 \end{equation*}
			 
			 In the case $1/2<\gamma<1$, we integrate \eqref{eq: size gradient fundamental solution} with respect to $\lambda$ using the change of variables $u={\frac{\epsilon}{2}}\sqrt{\omega_d}\sqrt{\lambda}|x-y|$, which yields
			 \begin{align*}|{\mathcal{K}_\gamma}(x,y)|
			& \lesssim \frac{1}{|x-y|^{d-1}} \exp\left(-{\frac{\epsilon}{2D_1}}\left(1+\frac{|y-x|}{\rho(x)}\right)^{\frac{1}{k_0+1}}\right)\int_{0}^{\infty}\lambda^{-\gamma} e^{-{\frac{\epsilon}{2}}\sqrt{\omega_d}\sqrt{\lambda}|x-y|}d\lambda\\
			& \lesssim \frac{1}{|x-y|^{d-(2\gamma-1)}} \exp\left(-{\frac{\epsilon}{2D_1}}\left(1+\frac{|y-x|}{\rho(x)}\right)^{\frac{1}{k_0+1}}\right),\end{align*}
			where we have used the fact that the integral converges since $\gamma<1$.
			
			To prove the corresponding smoothness condition, we consider $|x-x_0|<|y-x|/8$. Applying the estimate given in \cite[p.~563]{Shen99} 
			and using 
			\eqref{eq: size gradient fundamental solution}, we have\begin{align*}| \nabla_1 \Gamma_{\mu+ \lambda}(y,x) - \nabla_1 \Gamma_{\mu+ \lambda }(y,x_0)|
			& \lesssim \left(\frac{|x-x_0|}{|x-y|}\right)^{\delta}\sup_{z\in B(x,|x-y|/2)} |\nabla_1 \Gamma_{\mu+\lambda}(y,z)|\\
			&\lesssim \left(\frac{|x-x_0|}{|x-y|}\right)^{\delta} \frac{e^{-{\frac{\epsilon}{4}}\sqrt{\omega_d}\sqrt{\lambda}|x-y|}}{|x-y|^{d-1}}.
			\end{align*}For $\gamma=1$, we set $\lambda=0$ in the inequality above, leading to\[|{\mathcal{K}_1}(x,y)-{\mathcal{K}_1}(x_0,y)|\lesssim \frac{1}{|y-x|^{d-1}} \left(\frac{|x-x_0|}{|x-y|}\right)^{\delta},\]
			whenever $|x-y|>8|x-x_0|$. For the remaining cases, we integrate as before to obtain
			\begin{align*}|{\mathcal{K}_\gamma}(x,y)- {\mathcal{K}_\gamma}(x_0,y)|& \lesssim \int_0^\infty \lambda^{-\gamma}| \nabla_1 \Gamma_{\mu+ \lambda} (y,x) - \nabla_1 \Gamma_{\mu+ \lambda }(y,x_0)|  d\lambda\\
			& \lesssim \frac{1}{|x-y|^{d-1}}\left(\frac{|x-x_0|}{|x-y|}\right)^{\delta} \int_0^\infty \lambda^{-\gamma} e^{-{\frac{\epsilon}{4}}\sqrt{\omega_d}\sqrt{\lambda}|x-y|}  d\lambda\\
			& \lesssim \frac{1}{|y-x|^{d-(2\gamma-1)}} \left(\frac{|x-x_0|}{|x-y|}\right)^{\delta},
			\end{align*}
			provided $|x-y|>8|x-x_0|$.
			
			Finally, to show that $\mathcal{L}_\mu^{-\gamma}\nabla$ is bounded from $L^1(\mathbb{R}^d)$ into $L^{\frac{d}{d-(2\gamma-1)},\infty}(\mathbb{R}^d)$, we proceed as in the case of fractional integrals. Given the pointwise estimates obtained above, for $f\in L^1(\mathbb{R}^d)$ and any $1/2<\gamma\leq 1$, we write
			\begin{align*}| \mathcal{L}_\mu^{-\gamma}\nabla f(x)|
			&\leq \int_{\mathbb{R}^d}|{\mathcal{K}_\gamma}(x,y)||f(y)|  dy\\
			&\lesssim \int_{\mathbb{R}^d}\frac{1}{|x-y|^{d-(2\gamma-1)}} \exp\left(-{\frac{\epsilon}{2D_1}}\left(1+\frac{|y-x|}{\rho(x)}\right)^{\frac{1}{k_0+1}}\right)|f(y)|  dy\\
			&\lesssim \int_{\mathbb{R}^d} \frac{|f(y)|}{|x-y|^{d-(2\gamma-1)}} dy\\
			& \lesssim I_{2\gamma-1}(|f|)(x),
			\end{align*}where $I_{2\gamma-1}$ is the classical fractional integral of order $2\gamma-1$. Since the latter is of weak type $\left(1, \tfrac{d}{d-(2\gamma-1)}\right)$, it follows that $\mathcal{L}_\mu^{-\gamma}\nabla$ is also of this type.
			
			To prove \ref{itm: class L-gamma nabla deltamu<1} we need to check that the kernel $\mathcal{K}_\gamma$ satisfies the integral size condition \eqref{eq: size Hormander fract}. By \eqref{eq: size K1} and \eqref{eq: metric mu+lambda}, for any $\lambda\geq 0$ and $0<\delta_\mu<1$  we have
			\begin{align*}|\nabla_1\Gamma_{\mu+\lambda}(y,x)|
				&\lesssim \frac{e^{-\frac{\epsilon}{4}\sqrt{\omega_d}\sqrt{\lambda}|x-y|}}{|x-y|^{d-2}}\exp\left(-{\frac{\epsilon}{2D_1}}\left(1+\frac{|y-x|}{\rho(x)}\right)^{\frac{1}{k_0+1}}\right)\left(\int_{B \left(y,\frac{|x-y|}{2}\right)}\frac{d\mu(z)}{|z-y|^{d-1}}+ \frac{1}{|x-y|}\right).
				\end{align*}
				
				This immediately yields the bound for $\gamma=1$. For $\frac12<\gamma<1$, integrating against $\lambda^{-\gamma} d\lambda$ gives 
				\begin{align*}|\mathcal{K}_\gamma(x,y)|
					 &\lesssim \frac{\exp\left(-{\frac{\epsilon}{2D_1}}\left(1+\frac{|y-x|}{\rho(x)}\right)^{\frac{1}{k_0+1}}\right)}{|x-y|^{d-(2\gamma-1)-1}}\left(\int_{B \left(y,\frac{|x-y|}{2}\right)}\frac{d\mu(z)}{|z-y|^{d-1}}+ \frac{1}{|x-y|}\right).
					 \end{align*}
					 Then, for each $\frac12<\gamma\leq 1$, $|x-x_0|<R/2$, and $1<s<\frac{2-\delta_\mu}{1-\delta_\mu}$, proceeding as in the Riesz transform $\mathcal{R}^*_\mu$ (see \cite[Proposition~5.5]{DLT25}) we have
					  \begin{equation*}\left(\frac{1}{R^d}\int_{R<|x_0-y|\leq 2R}|\mathcal{K}_\gamma(x,y)|^s dy \right)^{1/s}\lesssim \frac{1}{R^{d-(2\gamma-1)}}\exp\left(-{\frac{\epsilon}{4D_1}}\left(1+\frac{R}{\rho(x)}\right)^{\frac{1}{k_0+1}}\right).
					  \end{equation*}
					 For  the smoothness condition we consider  $|x-x_0|<r<\rho(x_0)$ and $0<r<R/16$. Then proceeding as in the proof of \cite[Theorem~7.18]{Shen99} we have an estimate for $|\nabla_1\Gamma_{\mu+\lambda}(y,x)-\nabla_1\Gamma_{\mu+\lambda}(y,x_0)|$ (see \cite[p.~24]{DLT25} for completeness) and consequently that 
					 \begin{align*}|\mathcal{K}_\gamma(x,y)- \mathcal{K}_\gamma(x_0,y)|
					 	 &\lesssim\left(\frac{|x-x_0|}{|x-y|}\right)^{\delta} \frac{e^{-\epsilon_0\left(1+\frac{R}{\rho(x)}\right)^{\frac{1}{(k_0+1)^2}}}}{|x-y|^{d-(2\gamma-1)-1}}\left(\int_{B(y,|y-x|)}\frac{d\mu(z)}{|z-y|^{d-1}}+ \frac{1}{|x-y|}\right),
					 	 \end{align*}
					 	 	for some positive $\epsilon_0$, which implies condition \eqref{eq: smoothness Hormander fract} with $\nu=2\gamma-1$ for $1<s<\frac{2-\delta_\mu}{1-\delta_\mu}$.
					 	 	
					 	 	Finally, as a consequence of the estimates  given above, for  $f\in L^{s'}(\mathbb{R}^d)$ and $0<\delta_\mu<1$,
					 	 	\begin{align*}| \mathcal{L}_\mu^{-\gamma}\nabla f(x)|
				&\lesssim \sum_{j=1}^\infty \left(\int_{2^j\rho(x)<|y-x|\leq 2^{j+1}\rho(x)}|\mathcal{K}_\gamma(x,y)|^s dy\right)^{1/s}\left(\int_{B(x, 2^{j+1}\rho(x))}|f(y)|^{s'}\right)^{1/s'}
				\\
				&
				\lesssim\sum_{j =1}^\infty e^{-{\frac{\epsilon}{4D_1}}2^{j\frac{1}{k_0+1}}}|B(x,2^{j+1}\rho(x))|^{\frac{2\gamma-1}{d}}\left(\fint_{B(x, 2^{j+1}\rho(x))}|f(y)|^{s'}\right)^{1/s'}.
				\end{align*}
			 Since the classical fractional maximal operator is of weak type $(1, \frac{d}{d-(2\gamma-1)})$, we conclude that $\mathcal{L}_\mu^{-\gamma}\nabla$ is of weak type $\left(s',\frac{ds'}{d-(2\gamma-1)s'}\right)$, provided $(2\gamma-1)s'<d$.
						\end{proof}

		In the potential case $d\mu(x)=V(x)dx$ with $V\in RH_q$ and $\frac d2<q<d$, a sharper estimate for $\mathcal{L}_V^{-\gamma}\nabla$ can be obtained via \cite[Proposition~5.5]{DLT25}. 
		Indeed, the optimal endpoint for the Riesz transforms, $p_0 =\left(\frac1q-\frac1d\right)^{-1}$, automatically satisfies $p_0>\frac{d}{d-(2\gamma-1)}$ for any $\frac12<\gamma\leq 1$. This leads to the following classification.
		
		\begin{prop}\label{prop: class L-gamma nabla potential} Let $V\in RH_q$ with $\frac d2<q<d$, $V\not\equiv 0$, and $\frac12<\gamma\leq 1$. Then, $\mathcal{L}_V^{-\gamma}\nabla$ is an exponential SCZO of $(2\gamma-1,p_0,\delta)$ type, with parameters $c={\frac{\epsilon}{4D_1}}$ and $m=\frac{1}{k_0+1}$.
		\end{prop}
		As a consequence, and following the same lines as in the singular case given in \cite[Theorem~4.1, Proposition~4.2]{DLT25} we establish the weighted boundedness results for $\mathcal{L}_\mu^{-\gamma}\nabla$.
		
		\begin{thm}\label{thm: BMO L-gamma nabla}Let $\frac{1}{2}<\gamma\leq1$, $m=\frac{1}{k_0+1}$, $m^*=\frac{1}{(k_0+1)^2}$, $1<p<\frac{d}{2\gamma-1}$, and $\frac{1}{q}=\frac{1}{p}-\frac{2\gamma-1}{d}$. The operator $\mathcal{L}_\mu^{-\gamma}\nabla$ is bounded from $L^p(w^p)$ to $L^q(w^q)$ for any weight $w$ such that $w^{-p'}\in H_{1+\frac{p'}{q},c}^{\rho,m^*}$ where the constant $c$ satisfies:
			\begin{enumerate}
				\item $c < \frac{\epsilon}{2D_1}\frac{d}{d-(2\gamma-1)}(2^{3k_0+7}C_0^{k_0+3})^{-m^*}$ when $\delta_\mu>1$;
				\item $c < \frac{\epsilon}{4D_1}\frac{ds'}{ d-(2\gamma-1)s'}(2^{2k_0+6}C_0^{k_0+3})^{-m^*}$ when $0<\delta_\mu<1$.
				\end{enumerate}
				\end{thm}
				Finally, since $\mathcal{L}_\mu^{-\gamma}\nabla(1)=0$, we have the next result. 
				
				\begin{thm}Let $0<\delta_\mu<1 $, $\frac{1}{2}<\gamma\leq 1$ and $m=\frac{1}{k_0+1}$. For any $0 \leq \alpha+(2\gamma-1) <\delta$, the operator $\mathcal{L}_\mu^{-\gamma}\nabla $ is bounded from $\BMO_\rho^\alpha(w)$ to $\BMO_\rho^{\alpha+2\gamma-1}(w)$ whenever $w\in E_{s,c_1,c_2}^{\rho,m}\cap D_{\kappa,c_3}^{\rho,m}$ with $1 \leq \kappa <1+ \frac{\delta - 2\gamma-1-\alpha}{d}$ and $c_1,c_2,c_3\geq 0$ such that ${(c_2+c_3)}<c_0\left(1-\frac{d(\kappa-1)+\alpha+2\gamma-1}{\delta}\right)(4C_0)^{-m}$.
				\end{thm}

%\section*{Declarations}
%
%\subsection*{Conflicts of interest} 
%The authors declare that they have no competing interests.
%
\subsection*{Funding} This work was partially supported by Universidad Nacional del Litoral [Grant CAI+D 2024 50320260100272LI].

\bibliographystyle{acm}%estilo de la biblio
%\bibliography{referencias}

%------------------Referencias incluidas desde el bbl

\end{document}